\documentclass[10pt,oneside,shortlabels, reqno]{amsart}
\usepackage{color}
\usepackage{enumitem}
\usepackage{tikz}
\usepackage{tikz-cd}
\usepackage{amsbsy}
\usepackage{amstext}
\usepackage{amsthm}
\usepackage{amssymb}
\usepackage[unicode=true,pdfusetitle,
bookmarks=true,bookmarksnumbered=false,bookmarksopen=false,
breaklinks=false,pdfborder={0 0 0},pdfborderstyle={},backref=page,colorlinks=true]
{hyperref}
\hypersetup{
	linkcolor=blue}

\makeatletter
\newcommand{\sideremark}[1]{\ifvmode\leavevmode\fi\vadjust{\vbox to0pt{\vss 
			\hbox to 0pt{\hskip\hsize\hskip1em           
				\vbox{\hsize2cm\tiny\raggedright\pretolerance10000
					\noindent #1\hfill}\hss}\vbox to8pt{\vfil}\vss}}}%

\usepackage{mathrsfs}

\usepackage[all]{xy}
\newcommand{\xyL}[1]{%
	\xydef@\xymatrixrowsep@{#1}
} 
\newcommand{\xyC}[1]{%
	\xydef@\xymatrixcolsep@{#1}
} 

\newcommand\thmsname{Theorem}
\newcommand\nm@thmtype{thm}
\theoremstyle{plain}

\newenvironment{namedthm}[1]{
	\renewcommand\thmsname{#1}\renewcommand\nm@thmtype{namedtheorem}
	\begin{\nm@thmtype}
	}
	{\end{\nm@thmtype}
}
\theoremstyle{plain}
\newtheorem{thm}{\protect\theoremname}[section]
\theoremstyle{definition}
\newtheorem{defn}[thm]{\protect\definitionname}
\theoremstyle{remark}
\newtheorem{rem}[thm]{\protect\remarkname}
\theoremstyle{definition}
\newtheorem{example}[thm]{\protect\examplename}
\theoremstyle{remark}

\theoremstyle{plain}
\newtheorem{prop}[thm]{\protect\propositionname}
\theoremstyle{plain}
\newtheorem{lem}[thm]{\protect\lemmaname}

\numberwithin{equation}{section}

\makeatother

\providecommand{\definitionname}{Definition}
\providecommand{\remarkname}{Remark}
\providecommand{\examplename}{Example}
\providecommand{\lemmaname}{Lemma}
\providecommand{\notationname}{Notation}
\providecommand{\propositionname}{Proposition}
\providecommand{\theoremname}{Theorem}

\newtheorem*{thmA}{Theorem~A}
\newtheorem*{thmB}{Theorem~B}
\newtheorem*{thmC}{Theorem~C}

\begin{document}

\title[]{Normalization of strongly hyperbolic logarithmic transseries and complex Dulac germs}

\author{D. Peran$^1$}

\thanks{This research of D. Peran is fully supported by the Croatian Science Foundation (HRZZ) grant UIP-2017-05-1020. It is also supported by the Hubert-Curien `Cogito’ grant 2021/2022 \emph{Fractal and transserial approach to differential equations}.}

\begin{abstract}
	We give normal forms for strongly hyperbolic logarithmic transseries $f=z^{\alpha }+\ldots \, $, where $\alpha >0$, $\alpha \neq 1$, with respect to parabolic logarithmic normalizations. These normalizations are obtained using fixed point theorems, and are given algorithmically, as limits of Picard sequences in appropriate formal topologies. The results are applied to describe the supports of normalizations and to prove that the strongly hyperbolic complex Dulac germs are analytically normalizable on standard quadratic domains inside the class of complex Dulac germs.
\end{abstract}

\maketitle
\tableofcontents{}

\section{Introduction}

	In a full generality transseries are formal sums of formal products of powers, iterated exponentials and iterated logarithms (see e.g. \cite{Dries}, \cite{adh13}) with real or complex coefficients. Nowdays, transseries play important roles in understanding many problems in mathematics (see e.g. \cite{Ily84}, \cite{ecalle92}) and physics (see e.g. \cite{abs19}). \\
	
	For a given $\alpha \in \mathbb{R}_{>0} \setminus \left\lbrace 1\right\rbrace $, we consider the homological equation
	\begin{align}
		\varphi (f(z)) & =(\varphi (z))^{\alpha } , \quad \varphi (z)=z+o(z), \label{IntroductionEq1}
	\end{align}
	where $f$ is a given complex (real) map with fixed point zero, or a transseries. Therefore, we distinguish two types of Equation \eqref{IntroductionEq1}: \emph{analytic type} (for complex maps $f$ analytic on the \emph{Riemann surface of the logarithm}), and \emph{formal type} (for transseries $f$ and a formal composition). \\
	
	The classical \emph{B\"ottcher Theorem} solves Equation \eqref{IntroductionEq1} for analytic complex diffeomorphisms $f$ at zero (see e.g. \cite{CarlesonG93}, \cite{Milnor06}).
	
	In this paper we solve Equation \eqref{IntroductionEq1} for analytic maps $f$ on an invariant subdomain of the Riemann surface of the logarithm $\widetilde{\mathbb{C}}$ with certain logarithmic asymptotics. In particular, we consider the so-called \emph{Dulac germs} which appear as the \emph{first return maps} (or \emph{Poincaré maps}) in the solution of the hyperbolic variant of the \emph{Dulac problem} of non-accumulation of limit cycles on hyperbolic polycycles of analytic planar vector fields (see e.g. \cite{dulac}, \cite{ecalle92}, \cite{Ily84}, \cite{Ily91}, \cite{Roussarie98}). Dulac's proof is based on the asymptotic expansion of Dulac germs at zero (see \cite{dulac}). It can be shown that Dulac germs admit the so-called \emph{Dulac series} (particular type of transseries) as their asymptotic expansions at zero. There was an imprecision in the Dulac's proof. In particular, Dulac used the fact that every Dulac germ is uniquely determined by its asymptotic expansion, without giving an explicit proof. Fortunately, Dulac problem is solved independently by \emph{Ilyashenko} (see \cite{Ily84}, \cite{Ily91}) and \emph{Écalle} (see \cite{ecalle92}). In particular, Ilyashenko corrected this imprecision of Dulac by proving existence of holomorphic extensions of Dulac germs on sufficiently large complex domains called \emph{standard quadratic domains} (see \cite{Ily84}, \cite{Ily91}, \cite{IlyYak08}) and using the \emph{Phragmen-Lindel\"of Theorem} (see \cite{IlyYak08}) (a maximum modulus principle on unbounded complex domains). \\
	
	The notion of Dulac series leads us naturally to the notion of transseries as their generalization. In this paper we consider \emph{logarithmic transseries}, i.e., transseries that contain only powers and iterated logarithms (introduced in \cite{adh13} as ``purely logarithmic transseries"). The motivation for iterated logarithms comes from Poincar\'e maps of saddle-node polycycle (\cite{Ily84}, \cite{Ily91}). As in \cite{PRRSFormal21} we generalize Dulac germs to the complex Dulac germs on standard quadratic domains with Dulac series with possibly complex coefficients as their asymptotic expansions. Using \cite[Section 7]{Loray21}, \cite[Section 22C]{IlyYak08} and \cite[Section 3]{Roussarie98}, it is explained in \cite[Subsection 2.1]{PRRSDulac21} that complex Dulac germs on standard quadratic domains appear as \emph{corner maps} in hyperbolic complex saddles in $\mathbb{C}^{2}$.
	
	More about logarithmic transseries, complex Dulac germs and their formal and analytic normal forms, as well as more detailed proofs of the results stated in this paper can be found in the thesis of Peran \cite{Peran21Thesis}. \\
	
	In order to prove the existence and the uniqueness of the solution of Equation \eqref{IntroductionEq1} in the class of complex Dulac germs, we split our proof in two parts: the \emph{formal} and the \emph{analytic} part. These parts are proved in two main results of this paper: Theorem~A and (resp.) Theorem~C, which are stated in their full form in Section~\ref{SectionMainResults}.
	
	We first solve Equation \eqref{IntroductionEq1} in the class of logarithmic transseries $f$ (Theorem~A) in Section~\ref{sec:proofThmA}. In Theorem~B, in Section~\ref{SectionNormalizationAnalytic}, we solve Equation \eqref{IntroductionEq1} for analytic complex maps $f$ on suitable $f$-invariant complex domains. Finally, we relate these two independent results to prove Theorem~C, in Section~\ref{sec:proofThmB}. Theorems A and C are stated in Section~\ref{SectionMainResults}. \\
	
	\noindent \emph{\textbf{The formal part.}} \\
	
	We state here Theorem~A in its short form, since it is the main result of the formal part.
	
	\begin{namedthm}{Theorem~A (short form)}
		Every strongly hyperbolic logarithmic transseries $f=z^{\alpha }+o(z^{\alpha })$, $\alpha >0$, $\alpha \neq 1$, can be formally normalized to its first term $f_{0}:=z^{\alpha }$ by the unique parabolic logarithmic change of variables $\varphi $.
		
		Furthermore, if $\alpha >1$, $\varphi $ is the limit of the so-called \emph{B\"ottcher sequence}
		\begin{align*}
			& \left( z^{\frac{1}{\alpha ^{n}}} \circ h\circ f^{\circ n} \right) _{n}
		\end{align*}
		in the appropriate formal topology, for every initial parabolic condition $h$.
	\end{namedthm}
	
	Although Ilyashenko and Écalle considered the Dulac problem even for semi-hyperbolic polycycles (in a full generality), those results are much less understood than those for hyperbolic polycycles. In particular, in the semi-hyperbolic variant of the Dulac problem, the iterated logarithms appear in the asymptotic expansions of the first return maps. Therefore, we prove Theorem~A in a full generality (even for logarithmic transseries involving iterated logarithms).
	
	It turns out that Equation \eqref{IntroductionEq1} can be solved formally only if logarithmic transseries $f$ is of the form $f=\lambda z^{\alpha }+\textrm{``higher order terms"}$, $\lambda >0$. We call such $f$ a \emph{strongly hyperbolic} logarithmic transseries.
	
	Theorem~A is a generalization of results obtained in \cite{mrrz1} for strongly hyperbolic logarithmic transseries that do not contain iterated logarithms, and the results from e.g. \cite{CarlesonG93}, \cite{Loray98}, \cite{Milnor06}, for standard strongly hyperbolic power series. In \cite{mrrz1} the unique solution of Equation \eqref{IntroductionEq1} is obtained for logarithmic transseries $f$ that do not contain iterated logarithms, using transfinite compositions of elementary changes of variables.
	
	In \cite{mrrz2} and \cite{mrrz3} the formal and the analytic class of \emph{parabolic Dulac germs} (i.e., Dulac germs of a form $f(z)=z+o(z)$) are obtained, but not for their complex counterparts. It turns out that, in order to describe these classes, one must consider logarithmic transseries involving iterated logarithms.
	
	Most recently, in \cite{Peran21Parabolic} the formal normal form for parabolic logarithmic transseries is obtained, and in \cite{PRRSFormal21} and \cite{PRRSDulac21} the \emph{Schr\"oder's equation} (see e.g. \cite{Schroder70}) is solved formally and analytically for hyperbolic logarithmic transseries (logarithmic transseries of a form $f=\lambda z+\textrm{``higher order terms"}$, $\lambda \in \mathbb{R}_{>0}\setminus \left\lbrace 1\right\rbrace $) and for hyperbolic complex Dulac germs. In particular, Theorem~A can be viewed as a strongly hyperbolic analogue of the Main Theorem in \cite{PRRSFormal21}. Theorem~A is proved using the fixed point theorem (Proposition~\ref{PropFixedPoint}) from \cite[Proposition 4.2]{PRRSFormal21} and using the contractibility of the so-called \emph{B\"ottcher operator} $\mathcal{P}_{f}(h):=\left( h\circ f\right) ^{\frac{1}{\alpha }}$ in a suitable formal metric. \\
	
	\noindent \emph{\textbf{The analytic part.}} \\
	
	We state here Theorem~C in its short form, since it is the main result of the analytic part.
	
	\begin{namedthm}{Theorem~C (short form)}
		Every strongly hyperbolic (complex) Dulac germ $f(z)=z^{\alpha }+o(z^{\alpha })$, $\alpha >1$, on a standard quadratic domain (given in the standard $z$-chart), can be analytically normalized to the map $z\mapsto z^{\alpha }$ by the unique parabolic (complex) Dulac germ.
	\end{namedthm}

	Although Theorem~B is not the main result of the analytic part, we state it here in its short form, since it is an interesting result by itself and it is crucial for proving Theorem~C.
	
	\begin{namedthm}{Theorem~B (short form)}
		Let $\alpha >1$, $\varepsilon >0$ and $k\in \mathbb{N}$. Every analytic germ $f(\zeta)= \alpha \zeta +o\Big( \big( \log ^{\circ k}\zeta \big) ^{-\varepsilon } \Big) $ defined on an admissible domain $D\subseteq \mathbb{C}^{+}$ of type $(\alpha ,\varepsilon ,k)$ can be analytically linearized by the unique analytic germ $\varphi (\zeta )=\zeta +o(1)$.
	\end{namedthm}
	Theorem~C and Theorem~B can be viewed as strongly hyperbolic versions of their hyperbolic counterparts from \cite{PRRSDulac21}. As opposed to Theorem~C, where we request the complex Dulac germs to have full asymptotic expansions, analytic maps in Theorem~B have only the initial parts of the potential asymptotic expansions. Motivated by the B\"ottcher Theorem for analytic strongly hyperbolic germs, in Theorem~B we solve Equation \eqref{IntroductionEq1} in the $\zeta := - \log z$-chart which is global for the Riemann surface of the logarithm. Note that Equation \eqref{IntroductionEq1} in the $\zeta $-chart becomes the Schr\"oder's type equation $\varphi (f(\zeta )) = \alpha \cdot \varphi (\zeta )$, where $\zeta $ is a variable at infinity. In Theorem~B we prove the convergence of the \emph{Koenigs sequence} $\left( \frac{1}{\alpha ^{n}}f^{\circ n}(\zeta )\right) _{n}$ (from the classical \emph{Koenigs Theorem} (see e.g. \cite{Koenigs84}, \cite{CarlesonG93}, \cite{Milnor06})) towards the solution of Equation \eqref{IntroductionEq1} in the $\zeta $-chart. Theorem~B can be viewed as a generalization of the classical B\"ottcher Theorem to complex analytic maps with logarithmic asymptotics on invariant complex domains.
	
	Finally, in Section~\ref{sec:proofThmB}, Theorem~A and B are related using a suitable \emph{Schr\"oder's type homological equation} in order to prove the Dulac type of the solution in Theorem~C. Thus, in Theorem~C, we also prove that the formal class of a strongly hyperbolic complex Dulac germ completely determines its analytic class. \\
	
	In the future work, we hope to be able to use Theorem~A, \cite[Main Theorem]{PRRSFormal21}, \cite[Main Theorem]{Peran21Parabolic} and the ideas from the proof of Theorem~C, \cite{PRRSDulac21} and \cite{mrrz2}, to understand better the dynamics of the first return maps of semi-hyperbolic polycycles, which are, in general, more complicated than Dulac germs, since they can include also iterated logarithms.  

\medskip

\section{Prerequisites}\label{sec:basicnotions}

This section serves as a prerequisite for Sections~\ref{SectionMainResults}-\ref{sec:proofThmB} where the main theorems are stated and proved. We briefly recall notions of the differential algebra of logarithmic transseries, power-metric and weak topology, composition of logarithmic transseries, complex (real) Dulac germs and series (see \cite{PRRSFormal21} and \cite{PRRSDulac21}).

\medskip

\subsection{Differential algebras $\mathfrak L$ and $\mathfrak L^{\infty }$} In this subsection we briefly recall differential algebras $\mathfrak L$ and $\mathfrak L^{\infty }$ that are introduced in \cite{PRRSFormal21} as special models of a more general structure of the \emph{logarithmic-exponential field} introduced in \cite{Dries}. Furthermore, $\mathfrak L$ is contained in the set $\mathbb{T}_{\log }$ of ``purely logarithmic transseries" introduced in \cite{adh13}.

Let $\mathcal{L}_{k}$, $k\in \mathbb{N}$, be the set of all \emph{logarithmic transseries} $f$ \emph{of depth} $k$ (referring to the depth of the iterated logarithm), where
\begin{align}
	f := \sum _{(\alpha ,\mathbf{n})\in \mathbb{R}\times \mathbb{Z}^{k}}a_{\alpha ,\mathbf{n}}z^{\alpha }\boldsymbol{\ell}_{1}^{n_{1}}\cdots \boldsymbol{\ell}_{k}^{n_{k}} . \label{TransseriesDef} 
\end{align}
Here, $\boldsymbol{\ell }_{1}:=-\frac{1}{\log z}$, $\boldsymbol{\ell}_{n+1}:=\boldsymbol{\ell }_{1} \circ \boldsymbol{\ell }_{n}$, $n\in \mathbb{N} $, $a_{\alpha ,\mathbf{n}} \in \mathbb{R} $, $\mathbf{n}:=(n_{1},\ldots ,n_{k}) \in \mathbb{Z}^{k}$ is a multi-index, and
\begin{align*}
	\mathrm{Supp} \, (f) := \left\lbrace (\alpha , \mathbf{n}) \in \mathbb{R}\times \mathbb{Z}^{k} : a_{\alpha , \mathbf{n}} \neq 0\right\rbrace 
\end{align*}
is a \emph{well-ordered} subset of $\mathbb{R}\times \mathbb{Z}^{k}$ such that $\min \mathrm{Supp}\, (f) > \mathbf{0}_{k+1}$\footnote{Here, $\boldsymbol{0}_{k+1}:=(0,\ldots ,0) \in \mathbb{R}\times \mathbb{Z}^{k}$ is a multi-index.}, with respect to the \emph{lexicographic order} (if $\mathrm{Supp}\, (f) \neq \emptyset $). If $\mathrm{Supp} \, (f)=\emptyset $, then we call $f$ the \emph{zero logarithmic transseries} and denote it by $0$. We call $\mathrm{Supp} \, (f)$ the \emph{support} of $f$.

We call $a_{\alpha ,\mathbf{n}}z^{\alpha }\boldsymbol{\ell}_{1}^{n_{1}}\cdots \boldsymbol{\ell}_{k}^{n_{k}}$ in \eqref{TransseriesDef} a \emph{term} in $f$. Furthermore, we call $a_{\alpha ,\mathbf{n}}$ a \emph{coefficient of the term} $a_{\alpha ,\mathbf{n}}z^{\alpha }\boldsymbol{\ell}_{1}^{n_{1}}\cdots \boldsymbol{\ell}_{k}^{n_{k}}$ in $f$, and denote it by $\left[ f\right] _{\alpha , \mathbf{n}}$.

If $f\neq 0$, we define the \emph{order} of the logarithmic transseries $f\in \mathcal{L}_{k}$, $k\in \mathbb{N}$, by $\mathrm{ord} \, (f):=\min \mathrm{Supp} \, (f)$. If $f$ is the zero transseries, we put $\mathrm{ord}\, (f):=+\infty $. Furthermore, we define the \emph{leading term of} $f$ as the term of order $\mathrm{ord} \, (f)$ in the logarithmic transseries $f$ (if $f\neq 0$), and denote it by $\mathrm{Lt} \, (f)$.

By the usual identification of $\mathbb{R}\times \mathbb{Z}^{k}$ as a subset of $\mathbb{R}\times \mathbb{Z}^{k+1}$, note that $\mathcal{L}_{k} \subseteq \mathcal{L}_{k+1}$, $k\in \mathbb{N}$. Now, by $\mathfrak L$ we define an increasing union:
\begin{align*}
	\mathfrak{L} & := \bigcup _{k\in \mathbb{N}}\mathcal{L}_{k} .
\end{align*}
Note that $\mathfrak L$ is a differential algebra, with respect to the derivation $\frac{d}{dz}$, and that $\mathcal{L}_{k}$, $k\in \mathbb{N}$, are its subalgebras. We call $\mathfrak L$ the \emph{differential algebra of logarithmic transseries}.

For convenience, we often use the so-called \emph{blockwise} notation (see e.g. \cite{PRRSFormal21}, \cite{mrrz2}):
\begin{align*}
	f := \sum _{\alpha \in \mathbb{R}}z^{\alpha }R_{\alpha } ,
\end{align*} 
where
\begin{align*}
	R_{\alpha }:=\sum _{(\alpha ,\mathbf{n}) \in \mathrm{Supp} \, (f)}a_{\alpha ,\mathbf{n}}\boldsymbol{\ell}_{1}^{n_{1}}\cdots \boldsymbol{\ell}_{k}^{n_{k}}, \quad \alpha \in \mathbb{R} , \, \boldsymbol{n}=(n_{1},\ldots ,n_{k}),
\end{align*}
and
\begin{align*}
	\mathrm{Supp}_{z} \, (f) := \left\lbrace \alpha \in \mathbb{R} : R_{\alpha } \neq 0 \right\rbrace 
\end{align*}
is a well-ordered subset of $\mathbb{R} _{\geq 0}$. We call $\mathrm{Supp}_{z} \, (f)$ the \emph{support of} $f$ \emph{in} $z$, and $z^{\alpha }R_{\alpha }$, $\alpha \in \mathrm{Supp}_{z} \, (f)$, the $\alpha $\emph{-block} of a logarithmic transseries $f$.

We define the \emph{order of} $f$ \emph{in} $z$ as $\mathrm{ord}_{z} \, (f):=\min \mathrm{Supp} _{z} \, (f)$. Put $\gamma := \mathrm{ord}_{z} \, (f)$. We call the $\gamma $-block of $f$ the \emph{leading block} of a logarithmic transseries $f$.

We use the following acronyms:
\begin{enumerate}[1., font=\textup, topsep=0.2cm, itemsep=0.2cm, leftmargin=0.6cm]
	\item $f=g+\mathrm{h.o.t.}$ (which means: \emph{higher order terms}) if the order of every term in $g$ is strictly smaller than $\mathrm{ord} \, (f-g)$.
	\item $f=g+\mathrm{h.o.b.}(z)$ (which means: \emph{higher order blocks in} $z$) if the order in $z$ of every block in $g$ is strictly smaller than $\mathrm{ord}_{z} \, (f-g)$.
\end{enumerate}

If we allow in the definition of logarithmic transseries $f$ of depth $k\in \mathbb{N}$ that $\mathrm{Supp}\, (f)$ contains elements of $\mathbb{R}\times \mathbb{Z}^{k}$ that are of order less than or equal to the $\mathbf{0}_{k+1}$, then $\mathcal{L}_{k}$ is extended to the differential algebra that we denote by $\mathcal{L}_{k}^{\infty }$. Put
\begin{align*}
	\mathfrak L^{\infty } & :=\bigcup _{k\in \mathbb{N}}\mathcal{L}_{k}^{\infty } .
\end{align*}
Note that $\mathfrak L^{\infty }$ is a differential algebra, and $\mathfrak L$ and $\mathcal{L}_{k}^{\infty }$, $k\in \mathbb{N}$, are its subalgebras. Now, notions of \emph{leading term} (\emph{block}), \emph{order} (\emph{in} $z$), are defined analogously in the larger differential algebra $\mathfrak L^{\infty }$.

\medskip

\subsection{Two topologies on the differential algebra $\mathfrak L^{\infty }$}\label{SubsectionMetricTopology}

In this subsection we briefly recall the two topologies on the differential algebra $\mathfrak L^{\infty }$ that are introduced in \cite{PRRSFormal21}: the \emph{power-metric topology} (originally introduced in \cite{Dries} as the \emph{valuation topology}) and the \emph{weak topology} (originally introduced in \cite{mrrz1} on the differential algebra $\mathcal{L}_{1}$). \\

\noindent \emph{The power-metric topology.} We define $d_{z}:\mathfrak L^{\infty }\times \mathfrak L^{\infty }\to \mathbb{R} $ by:
$$
d_{z}(f,g):=\left\lbrace \begin{array}{ll}
	2^{-\mathrm{ord}_{z}(f-g)}, & f\neq g, \\
	0, & f=g.
\end{array} \right. 
$$
We call $d_{z}$ the \emph{power-metric}, and the induced topology, the \emph{power-metric topology} that we denote by $\mathcal{T}_{d_{z}}$. \\

\noindent \emph{The weak topology.} We consider $\mathfrak L^{\infty }$ as a subspace of the space $\mathbb{R} ^{\mathbb{R}\times \mathbb{Z} ^{\mathbb{N}}}$, with respect to the product topology, where the Euclidean topology is taken on every coordinate space. We call the related relative topology on $\mathfrak L^{\infty }$ the \emph{weak topology}, and denote it by $\mathcal{T}_{w}$.

From the definition of a product topology, it follows that a sequence $(\varphi _{n})$ converges to $\varphi $ in the weak topology on $\mathfrak L^{\infty }$ if and only if the sequence of coefficients $\left( \left[ \varphi _{n}\right] _{w}\right) _{n}$ converges towards the coefficient $\left[ \varphi \right] _{w}$ in the Euclidean topology, for every $w\in \mathbb{R}\times \mathbb{Z} ^{\mathbb{N}}$. Furthermore, it follows that
\begin{align*}
	\mathrm{Supp} \, (\varphi ) & \subseteq \bigcup _{n\in \mathbb{N}}\mathrm{Supp} \, (\varphi _{n}) .
\end{align*}

Now, suppose that $(\varphi _{n})$ is a sequence in $\mathfrak L^{\infty }$ that converges on the product $\mathbb{R} ^{\mathbb{R}\times \mathbb{Z} ^{\mathbb{N}}}$ towards $\varphi $. Then $\varphi $ is not necessarily in $\mathfrak L^{\infty }$, but, if we additionally assume that $\bigcup _{n\in \mathbb{N}}\mathrm{Supp} \, (\varphi _{n})$ is a well-ordered subset of $\mathbb{R}\times \mathbb{Z}^{\mathbb{N}}$ (with respect to the lexicographic order), then $\varphi \in \mathfrak L^{\infty }$ (for details see \cite{Peran21Thesis}). \\

Note that $\mathcal{T}_{w} \subsetneqq \mathcal{T}_{d_{z}}$ (see \cite{mrrz1} and \cite{Peran21Thesis} for counter-examples).

\medskip

\subsection{Groups $\mathcal{L}_{k}^{H}$ and $\mathcal{D}$}\label{SubsectionGroupsL}

Since $\boldsymbol{\ell }_{k+1}=\boldsymbol{\ell }_{1} \circ \boldsymbol{\ell }_{k}$, $k\in \mathbb{N}$, it is easy to see that the composition is not well-defined on the differential algebra $\mathcal{L}_{k}$, $k\in \mathbb{N} $, so, as in \cite{PRRSFormal21} and \cite{mrrz1} (in $\mathcal{L}_{1}$), we restrict ourselves on the subset $\mathcal{L}_{k}^{H}\subseteq \mathcal{L}_{k}$ of all logarithmic transseries $f$ that do not contain logarithms in their leading terms, i.e., $f=\lambda z^{\alpha }+\mathrm{h.o.t.}$, for $\alpha , \lambda >0$. \\

Let $h \in \mathcal{L}_{k}^{H}$ such that $h:=\lambda z^{\alpha }+h_{1}$, where $\alpha , \lambda >0$ and $\mathrm{ord} \, (h_{1}) > (\alpha , \boldsymbol{0}_{k})$. Now, put:
\begin{align}
	& \log z \circ (\lambda z^{\alpha }+h_{1}) := \log \lambda + \alpha \log z +\sum _{i\geq 1}\frac{(-1)^{i+1}}{i}\left( \frac{h_{1}}{\lambda z^{\alpha }}\right) ^{i} , \nonumber \\
	& z^{\beta } \circ (\lambda z^{\alpha }+h_{1}) := \lambda ^{\beta }z^{\alpha \beta } \cdot \sum _{i\geq 0}{\beta \choose i}\left( \frac{h_{1}}{\lambda z^{\alpha }}\right) ^{i} , \label{CompEq1}
\end{align}
for $\beta >0$. It is easy to see that $\log z \circ (\lambda z^{\alpha }+h_{1})$ and $z^{\beta } \circ (\lambda z^{\alpha }+h_{1})$ are logarithmic transseries. \\

Let $g \in \mathcal{L}_{k}$, $k\in \mathbb{N}$, and $f \in \mathcal{L}_{k}^{H}$ such that $f:=\lambda z^{\alpha }+f_{1}$, where $\alpha , \lambda >0$ and $\mathrm{ord} \, (f_{1}) > (\alpha , \boldsymbol{0}_{k})$. We use formulas \eqref{CompEq1} \emph{term-by-term} to define $g(\lambda z^{\alpha })$ (for more details see Lemma~$A.3.1$ and Lemma~$A.3.2$ in \cite{Peran21Thesis}). Now, by the \emph{Taylor Theorem} (see \cite[Proposition 3.3]{PRRSFormal21}, or more generally \cite{Dries}), we define the \emph{composition of transseries} $g$ and $f$ by:
\begin{align}
	g \circ f & := g(\lambda z^{\alpha }) + \sum _{i\geq 1}\frac{g^{(i)}(\lambda z^{\alpha })}{i!}f_{1}^{i} . \label{CompEq2} 
\end{align}
Here, the above series in \eqref{CompEq1} and \eqref{CompEq2} converge in the weak topology. This is proven in \cite[Proposition 3.3]{PRRSFormal21} using the \emph{Neumann Lemma} (see \cite{Neumann49}) and the notion of \emph{summable families} defined in \cite{Dries}.

It can be shown that $\mathcal{L}_{k}^{H}$, $k\in \mathbb{N}$, is a group under composition (see e.g. \cite{Dries}). \\

Now, put
\begin{align*}
	\mathfrak L^{H} & :=\bigcup _{k\in \mathbb{N}}\mathcal{L}_{k}^{H} .
\end{align*}

As in \cite{PRRSFormal21} and \cite{mrrz1} (in $\mathcal{L}_{1}$) we define three types of logarithmic transseries $f\in \mathfrak L^{H}$, $f:=\lambda z^{\alpha }+\mathrm{h.o.t.}$, for $\alpha , \lambda > 0$:
\begin{enumerate}[1., font=\textup, topsep=0.2cm, itemsep=0.2cm, leftmargin=0.6cm]
	\item \emph{parabolic} (or \emph{tangent to the identity}), if $\alpha =\lambda =1$,
	\item \emph{hyperbolic}, if $\alpha =1$, $\lambda \neq 1$,
	\item \emph{strongly hyperbolic}, if $\alpha \neq 1$.
\end{enumerate}
We denote by $\mathcal{L}_{k}^{0}$, $k\in \mathbb{N}$, the set of all parabolic logarithmic transseries in $\mathcal{L}_{k}^{H}$. Now, put
\begin{align*}
	\mathfrak L^{0} & :=\bigcup _{k\in \mathbb{N}}\mathcal{L}_{k}^{0} .
\end{align*}
It is easy to see that $\mathfrak L^{0}$ and $\mathcal{L}_{k}^{0}$, $k\in \mathbb{N}$, are subgroups of the group $\mathfrak L^{H}$. \\

Following \cite{IlyYak08} and \cite{mrrz2}, we define the \emph{Dulac series} $f\in \mathcal{L}_{1}^{H}$ as logarithmic series of the form:
\begin{align*}
	f := \lambda z^{\alpha } + \sum _{n\in \mathbb{N} _{\geq 1}}z^{\alpha _{n}}P_{n}(\boldsymbol{\ell}_{1}^{-1}) ,
\end{align*}
where $\alpha , \lambda > 0$, and $(\alpha _{n})$ is a strictly increasing sequence of real numbers strictly bigger than $\alpha $ and tending to $+\infty $. Furthermore, $(P_{n})$ is a sequence of real polynomials in the variable $\boldsymbol{\ell}_{1}^{-1}=-\log z$.

We denote by $\mathcal{D}$ the \emph{set of Dulac series}. It is easy to see that $\mathcal{D}$ is a subgroup of the group $\mathcal{L}_{1}^{H}$ (see \cite[Section 0]{Ily91}).

\medskip

\subsection{Complex Dulac germs}

In this subsection we recall from \cite{PRRSDulac21} the notion of complex Dulac germs. We first recall the notion of analytic germs on spiraling domains around the origin of the Riemann surface of the logarithm.

\smallskip

\subsubsection{Analytic germs on spiraling domains}

Let us recall a few basic notions from \cite[Section 1]{PRRSDulac21}. The \emph{Riemann surface of the logarithm} is the set
\begin{align*}
	\widetilde{\mathbb{C}} := \left\lbrace (r,\theta ) : r\in \mathbb{R}_{>0}, \, \theta \in \mathbb{R} \right\rbrace 
\end{align*}
endowed with the structure of one-dimensional analytic Riemann manifold whose atlas contains only one chart $-\log :\widetilde{\mathbb{C}}\to \mathbb{C}$ called the \emph{logarithmic chart} (or the $\zeta $-chart) where
\begin{align}
	& (r,\theta ) \to \zeta := -\log r - \mathrm{i} \cdot \theta . \label{Surface1}
\end{align}
By the classical abuse of the notation we identify $z:=r\cdot \mathrm{e}^{\mathrm{i}\cdot \theta }$, where we distinguish $\mathrm{e}^{\mathrm{i}\cdot \theta }$ and $\mathrm{e}^{\mathrm{i}\cdot (\theta +2k\pi )}$, for every $k\in \mathbb{Z}$. In this case we say that elements of $\widetilde{\mathbb{C}}$ are written in the \emph{standard} $z$\emph{-chart}. Using the notation from \eqref{Surface1}, we get $\zeta = - \log z$ and $z=\mathrm{e}^{-\zeta }$. \\

\smallskip

A \emph{spiraling domain} $\mathcal{N}$ around the origin of the Riemann surface of the logarithm is defined as the set
\begin{align*}
	\mathcal{N} & := \left\lbrace r\cdot \mathrm{e}^{\mathrm{i}\cdot \theta } : 0<r<h(\theta ) \right\rbrace 
\end{align*}
for a continuous map $h:\mathbb{R}\to \left( 0,+\infty \right) $ (see \cite[Subsection 6.8]{MiSa}). \\

We define \emph{germs} on spiraling domains standardly (see \cite[Subsection 6.8]{MiSa}) as equivalence classes of maps on spiraling domains, where two maps $f$ and $g$ are representatives of the same germ if they coincide on some spiraling domain. By the classical abuse of notation we often identify germs with their representatives.

We say that a germ on a spiraling domain around the origin of the Riemann surface of the logarithm is \emph{analytic} if it has a representative $f$ on some spiraling domain $\mathcal{N}$ such that the related map $f(\zeta ):=-\log \left( f\left( \mathrm{e}^{-\zeta }\right) \right) $ is analytic on the domain $-\log (\mathcal{N})$ in the $\zeta $-chart. \\

\smallskip

Let $\mathbb{C}^{+}:=\left\lbrace z\in \mathbb{C} : \Re (z) > 0\right\rbrace $. We say that a germ $f$ on a spiraling domain $\mathcal{N}$ around the origin of the Riemann surface of the logarithm is:
\begin{enumerate}[1., font=\textup, topsep=0.4cm, itemsep=0.4cm, leftmargin=0.6cm]
	\item \emph{parabolic}, if $f(\zeta )=\zeta +o(1)$, uniformly on $-\log \left( \mathcal{N}\right) $, as $\Re (\zeta ) \to + \infty $,
	\item \emph{hyperbolic}, if $f(\zeta )=\zeta + \beta +o(1)$, for $\beta \in \mathbb{C}\setminus \mathrm{i}\mathbb{R}$, uniformly on $-\log \left( \mathcal{N}\right) $, as $\Re (\zeta ) \to + \infty $,
	\item \emph{strongly hyperbolic}, if $f(\zeta )=\alpha \zeta +\beta +o(1)$, $\alpha \in \mathbb{R}_{>0}\setminus \left\lbrace 1\right\rbrace $, $\beta \in \mathbb{C}$, uniformly on $-\log \left( \mathcal{N}\right) $, as $\Re (\zeta ) \to + \infty $.
\end{enumerate}

\begin{rem}
	Switching from the $\zeta $-chart to the $z$-chart, in the above definition, we get that a germ $f$ on a spiraling domain $\mathcal{N}$ (given in the $z$-chart) is:
	\begin{enumerate}[1., font=\textup, topsep=0.4cm, itemsep=0.4cm, leftmargin=0.6cm]
		\item \emph{parabolic}, if and only if $f(z)=z+o(z)$ uniformly on $\mathcal{N}$, as $\left| z\right| \to 0$,
		\item \emph{hyperbolic}, if and only if $f(z)=\mathrm{e}^{-\beta }z+o(z)$, for $\beta \in \mathbb{C}\setminus \mathrm{i}\mathbb{R}$, uniformly on $\mathcal{N}$, as $\left| z\right| \to 0$,
		\item \emph{strongly hyperbolic}, if and only if $f(z)=\mathrm{e}^{-\beta }z^{\alpha }+o\left( z^{\alpha }\right) $, $\alpha \in \mathbb{R}_{>0}\setminus \left\lbrace 1\right\rbrace $, $\beta \in \mathbb{C}$, uniformly on $\mathcal{N}$, as $\left| z\right| \to 0$.
	\end{enumerate}
	The above notions are motivated by the definitions of parabolic, hyperbolic and strongly hyperbolic power series (see e.g. \cite{CarlesonG93}, \cite{Milnor06}) and logarithmic transseries (Subsection~\ref{SubsectionGroupsL}). 
\end{rem}

\smallskip

\subsubsection{Complex Dulac germs and complex Dulac series}\label{SubsubsectionDulac}

Let $\mathcal{L}_{k}(\mathbb{C})$ be the set of all logarithmic transseries of depth $k\in \mathbb{N}$ with complex coefficients. Note that $\mathcal{L}_{k}$ is a subalgebra of the differential algebra $\mathcal{L}_{k}(\mathbb{C})$. Similarly, we define $\mathcal{L}_{k}^{\infty }(\mathbb{C})$, $k\in \mathbb{N}$, $\mathfrak L(\mathbb{C})$ and $\mathfrak L^{\infty }(\mathbb{C})$. Furthermore, we denote by $\mathcal{D}(\mathbb{C})$ the group of all Dulac series with complex coefficients. We call every element of $\mathcal{D}(\mathbb{C})$ a \emph{complex Dulac series}.

Let
\begin{align}
	f & := \lambda z^{\alpha }+\sum _{i=1}^{+\infty }z^{\alpha _{i}}P_{i}\left( \boldsymbol{\ell }_{1}^{-1}\right) \label{Transformation1}
\end{align}
be a complex Dulac series, where $(\alpha _{i})$ is a strictly increasing sequence of real numbers strictly bigger than $\alpha $ and tending to $+\infty $. Furthermore, $(P_{i})$ is a sequence of polynomials in the variable $\boldsymbol{\ell }_{1}^{-1}=-\log z$ with complex coefficients. By the formal change of variables $\zeta =-\log z$, we get
\begin{align}
	f(\zeta ) & = -\log \left( f\left( \mathrm{e}^{-\zeta }\right) \right) \nonumber \\
	& = \alpha \zeta - \log \lambda +\sum _{i=1}^{+\infty }\mathrm{e}^{-\beta _{i}\zeta }Q_{i}\left( \zeta \right) , \label{Transformation2}
\end{align}
where $(\beta _{i})$ is a strictly increasing sequence of positive real numbers tending to $+\infty $. Furthermore, $(Q_{i})$ is a sequence of polynomials in the variable $\zeta $ with complex coefficients. \\

In Remark~\ref{RemarkTransformation} we relate asymptotic expansions in the $z$-chart and the $\zeta $-chart.

\begin{rem}\label{RemarkTransformation}
	Let $f:\left( 0,d\right) \to \mathbb{R}$ be a map, for $d\in \mathbb{R}_{>0}$, and let $\widehat{f}(z)$ be a Dulac series (with real coefficients) given by \eqref{Transformation1}. Let $\widehat{f}(\zeta )$ be obtained from $\widehat{f}(z)$ by the formal changes of variables $\zeta =-\log z$, and therefore, given by \eqref{Transformation2}. Note that $\widehat{f}(z)$ is the asymptotic expansion of $f$ on $\left( 0,d\right) $ as $z\to 0$ if and only if $\widehat{f}(\zeta )$ is the asymptotic expansion of $f(\zeta ):=-\log \left( f\left( \mathrm{e}^{-\zeta }\right) \right) $ on $\left(-\log d,+\infty \right) $ as $\zeta \to +\infty $.
\end{rem}

\begin{defn}[Standard quadratic domain, Definition 24.25 in \cite{IlyYak08}]\label{DefinitStandardQuadraticDomain}
	The \emph{standard quadratic domain} $\mathcal{R}_{C}\subseteq\widetilde{\mathbb{C}}$, $C\in \mathbb{R}_{>0}$, is the set defined in the logarithmic chart as
	\begin{align}
		& \kappa\big(\mathbb{C}^{+}\big),\text{ where }\kappa(\zeta)=\zeta+C(\zeta+1)^{\frac{1}{2}},\label{eq:sqd}
	\end{align}
	(see Figure~\ref{Figure1}). 
\end{defn}

\begin{figure}[h!]
	\centering
	\begin{tikzpicture}
		\filldraw[draw=none,fill=blue!30!white] (0,-2.4) rectangle (2.2,2.4);
		\draw (3.65,1.5) node{$\kappa $};
		\filldraw[draw=none,fill=blue!30!white] (9.1,-3.1) -- (6.6,-0.96) -- (6.6,0.96) -- (9.1,3.1) -- cycle;
		\filldraw[draw=black,thick,fill=blue!30!white] (6.9,0) circle (1cm);
		\filldraw[draw=none,thick,fill=blue!30!white,rounded corners] (6.62,-1.05) rectangle (8.5,1.05);
		\filldraw[draw=black,thick,fill=white] (6.5,0.92) arc (290:330:5cm);
		\filldraw[draw=black,thick,fill=white] (6.5,-0.92) arc (70:30:5cm);
		\draw[red!40!white,very thick,->] (2.8,1) to [in=150,out=30] (4.5,1);
		\draw (5.9,0) node[draw=black,thick,fill=black,circle,scale=0.4]{};
		\draw (5.9,0) node[below left]{$C$} (1.1,1.7) node{$\mathbb C^+$} (8.4,1.4) node{$\mathcal{R}_{C}$};
		\draw[black,thick] (-1,0) -- (2.5,0) (0,-2.7) -- (0,2.7) (4,0) -- (9.4,0) (5,-3.2) -- (5,3.2);
	\end{tikzpicture}
	\caption{The standard quadratic domain $\mathcal{R}_{C}$, $C>0$, in the $\zeta $-chart (\cite[Figure 1]{PRRSDulac21}).}
	\label{Figure1}
\end{figure}
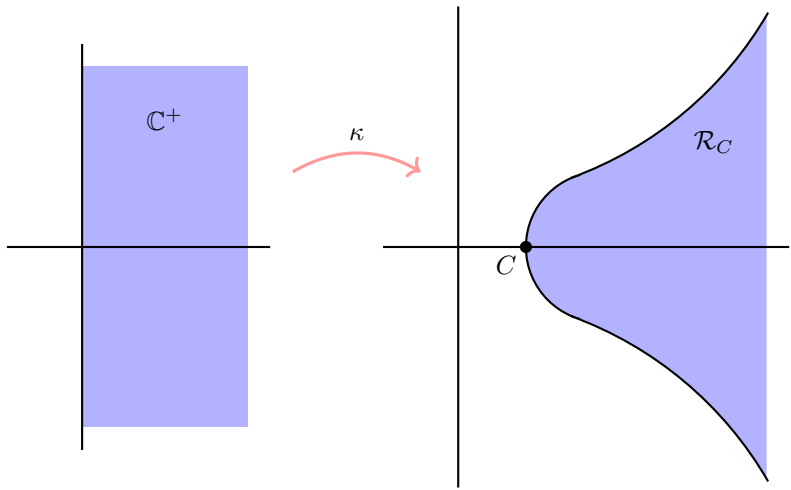

In statement 1 of Remark~\ref{RemarkStandardQ} we prove that standard quadratic domains are examples of spiraling domains around the origin of the Riemann surface of the logarithm. In statement 2 of Remark~\ref{RemarkStandardQ} we point out some important property of standard quadratic domains that will be used in the proof of Theorem~C in Section~\ref{sec:proofThmB}.

\begin{rem}\label{RemarkStandardQ}\hfill
	\begin{enumerate}[1., font=\textup, topsep=0.4cm, itemsep=0.4cm, leftmargin=0.6cm]
		\item As in \cite[Example (3)]{PRRSDulac21}, by a direct computation, the upper half of the boundary of the standard quadratic domain $\mathcal{R}_{C}$, $C>0$, is parametrized by:
		\begin{align*}
			r\mapsto x(r)+\mathrm{i}\cdot y(r) & =C\sqrt[4]{r^{2}+1}\cos\left(\frac{1}{2}\mathrm{arctg}\,r\right)+ \\
			& +\mathrm{i}\cdot\left(r+C\sqrt[4]{r^{2}+1}\sin\left(\frac{1}{2}\mathrm{arctg}\,r\right)\right),\ r\in\left[0,+\infty\right).
		\end{align*}
		The lower half of the boundary of the standard quadratic domain $\mathcal{R}_{C}$ is parametrized by $r\mapsto x(r)-\mathrm{i}\cdot y(r)$, $r\in \left[ 0,+\infty \right) $. Note that $r\mapsto x(r)$ and $r\mapsto y(r)$ are continuous maps. It can be shown that $r\mapsto x(r)$ and $r\mapsto y(r)$ are strictly increasing maps, so by putting $h:=x\circ y^{-1}$, on $\left[ 0,+\infty \right) $ and $h(y):=h(-y)$, for $y\in \left( -\infty ,0\right] $, it follows that standard quadratic domain $\mathcal{R}_{C}$ is a spiraling domain around the origin of the Riemann surface of the logarithm.
		\item Let $\mathcal{R}_{C}$, $C>0$, be a standard quadratic domain given in the $\zeta $-chart and let $R>0$. Note that there exists $C_{1}>C,R$ large enough such that $\mathcal{R}_{C_{1}}\subseteq \mathcal{R}_{C}$ and $\Re (\zeta ) \geq R$, for every $\zeta \in \mathcal{R}_{C_{1}}$.
	\end{enumerate}
\end{rem}

Motivated by Remark~\ref{RemarkTransformation}, for easier computations we work in the $\zeta $-chart. Therefore, we introduce the following conventions. In order to distinguish a germ $f$ and its logarithmic asymptotic expansion, we use notation $\widehat{f}$ for its asymptotic expansion. Furthermore, if an asymptotic expansion $\widehat{f}$ is given in the formal variable $z$ at zero, then we denote it simply by $\widehat{f}$, and if it is given in the formal variable $\zeta $ at infinity, then we denote it by $\widehat{f}(\zeta )$. \\

Let $f$ be a germ on a spiraling domain $\mathcal{N}$ (given in the $z$-chart) around the origin of the Riemann surface of the logarithm, and let
\begin{align*}
	\widehat{f}(\zeta ) & :=\sum _{i=0}^{+\infty }\mathrm{e}^{-\beta _{i}\zeta }Q_{i}(\zeta ),
\end{align*}
be a complex Dulac series given in the $\zeta $-chart, where $\beta _{0}=0$ and $Q_{0}(\zeta )=\alpha \zeta +\beta $, $\alpha \in \mathbb{R}_{>0}$, $\beta \in \mathbb{C}$. We say that complex Dulac series $\widehat{f}$ is the \emph{asymptotic expansion} of the germ $f$, and write $f\sim \widehat{f}$, if for every $\alpha > 0$, there exists $n_{\alpha }\in \mathbb{N}$, such that
\begin{align*}
	\left| f(\zeta )-\sum _{i=0}^{n_{\alpha }}\mathrm{e}^{-\beta _{i}\zeta }Q_{i}(\zeta ) \right| & = o\left( \mathrm{e}^{-\alpha \zeta }\right) ,
\end{align*}
uniformly on $-\log \left( \mathcal{N}\right) $ as $\Re (\zeta ) \to +\infty $. \\

We say that $f$ is a \emph{complex Dulac germ} if $f$ has an representative that is analytic on a standard quadratic domain $\mathcal{R}_{C}$, $C>0$, and there exists a complex Dulac series $\widehat{f} \in \mathcal{D}(\mathbb{C})$, such that $f\sim \widehat{f}$, uniformly on $\mathcal{R}_{C}$ as $\Re (\zeta ) \to +\infty $ (see \cite[Subsection 2.1]{PRRSDulac21}).

If additionally $f$ preserves the positive part of the real line, then we call $f$ the \emph{real Dulac germ} (or \emph{almost regular germ}, see \cite[Definition 24.27]{IlyYak08}). In this case, note that the asymptotic expansion $\widehat{f}$ of $f$ necessarily has real coefficients. \\

Recall from the introduction that Ilyashenko proved that every Dulac germ is uniquely determined by its Dulac asymptotic expansion, see \cite{Ily91}, \cite[Section 24]{IlyYak08}. This property of Dulac germs is called \emph{quasi-analyticity}. This is proved using the fact that standard quadratic domains are biholomorphic images of $\mathbb{C}^{+}$ and using the \emph{Phragmen-Lindel\"of Theorem} (a maximum modulus principle on an unbounded complex domain) (see e.g. \cite[Theorem 24.36]{IlyYak08}). It can be seen that the \emph{quasi-analyticity} property also holds for complex Dulac germs on standard quadratic domains.

\medskip

\section{Main results}\label{SectionMainResults}

Now we state two main theorems: Theorem~A and Theorem~C. Their proofs are in Section~\ref{sec:proofThmA} and Section~\ref{sec:proofThmB}, respectively.

\medskip 

\subsection{The formal part (Theorem~A)}

\begin{thmA}[Normalization of a strongly hyperbolic logarithmic transseries]\label{ThmA}
	Let $f\in \mathfrak L^{H}$, $f=z^{\alpha }+\mathrm{h.o.t.}$, $\alpha \in \mathbb{R}_{>0}$, $\alpha \neq 1$, be a strongly hyperbolic logarithmic transseries. Then:
	\begin{enumerate}[1., font=\textup, topsep=0.4cm, itemsep=0.4cm, leftmargin=0.6cm]
		\item There exists a unique solution $\varphi  \in \mathfrak L ^{0}$ of the \emph{normalization equation}:
		\begin{align}
			\varphi  \circ f \circ \varphi ^{-1} = z^{\alpha }. \label{EqNorm}
		\end{align}
		Moreover, $\mathrm{ord}_{z} \, (\varphi - \mathrm{id}) \geq \mathrm{ord}_{z} \, (f-z^{\alpha }) - \alpha +1$. Additionally, if $f\in \mathcal{L}_{k}^{H}$, then $\varphi \in \mathcal{L}_{k}^{0}$.
		\item Let $\alpha >1$. For every initial condition $h\in \mathfrak L^{0}$, the \emph{B\"ottcher sequence}
		\begin{align}
			\left( z^{\frac{1}{\alpha ^{n}}} \circ h \circ f^{\circ n}\right) _{n} \label{BottcherSeq}
		\end{align}
		converges to the normalization $\varphi $ in the weak topology on $\mathfrak L^{0}$, as $n$ tends to $+\infty $.
		
		Moreover, the sequence \eqref{BottcherSeq} converges in the power-metric topology on $\mathfrak L^{0}$ if and only if the initial condition $h$ is such that $\mathrm{Lb}_{z}\, (h)=\mathrm{Lb}_{z} \, (\varphi )$.
		\item Let $f\in \mathcal{L}_{k}$, $k\in \mathbb{N}$. The support $\mathrm{Supp} \, (\varphi )$ is contained in the semigroup generated by $(\alpha ^{p},\mathbf{0}_{k})$, $p\in \mathbb{N}$, $(0,1,0,\ldots ,0)_{k+1}$, $\ldots \, $, $(0,0,\ldots ,0,1)_{k+1}$, and $(\alpha ^{m}(\gamma - \alpha ),\mathbf{n})$, for $(\gamma , \mathbf{n}) \in \mathrm{Supp} \, (f-z^{\alpha})$, $m\in \mathbb{N}$.
	\end{enumerate}
\end{thmA}

\begin{rem}\hfill
	\begin{enumerate}[1., font=\textup, topsep=0.4cm, itemsep=0.4cm, leftmargin=0.6cm]
		\item Note that Theorem~A is stated only for strongly hyperbolic logarithmic transseries $f$ of a form $f=z^{\alpha }+\mathrm{h.o.t.}$, $\alpha \in \mathbb{R}_{>0}\setminus \left\lbrace 1\right\rbrace $. This can be nevertheless supposed without loss of generality. Indeed, suppose that $f=\lambda z^{\alpha }+\mathrm{h.o.t.}$, $\alpha , \lambda >0$, $\alpha , \lambda \neq 1$, is strongly hyperbolic logarithmic transseries. Put $\psi :=\lambda ^{\frac{1}{\alpha -1}} \cdot z$. Note that $\psi \circ f\circ \psi ^{-1}=z^{\alpha }+\mathrm{h.o.t.}$, which transforms $f$ in the form stated in Theorem~A.
		\item Let $f=z^{\alpha }+\mathrm{h.o.t.}$, for $0<\alpha <1$. Let $\varphi $ be a parabolic logarithmic transseries. Note that $\varphi \circ f\circ \varphi ^{-1}=z^{\alpha }$ if and only if $\varphi \circ f^{-1}\circ \varphi ^{-1}=z^{\frac{1}{\alpha }}$. Therefore, without loss of generality we can assume that $\alpha >1$.
		\item Theorem~A can be viewed as an analogue of \cite[Main Theorem]{PRRSFormal21}, but for strongly hyperbolic, instead of hyperbolic logarithmic transseries. Let us point out some of the key differences between two theorems. Firstly, the statement 2 of Theorem~A is much stronger result than the statement 2 of \cite[Main Theorem]{PRRSFormal21} since the B\"ottcher sequence converges in the weak topology for any initial condition, which is not the case with hyperbolic logarithmic transseries and the generalized Koenigs sequence in \cite[Main Theorem]{PRRSFormal21}. This allows us to be able to control the supports of normalizations much easier.
		\item We call the unique solution $\varphi $ of Equation \eqref{EqNorm} the \emph{normalization of the strongly hyperbolic transseries} $f$.
		\item Note that Theorem~A is constructive. In particular, statement 2 gives an explicit algorithm for obtaining the normalization $\varphi $. For every $w$ in the union of supports of $z^{\frac{1}{\alpha ^{n}}} \circ h \circ f^{\circ n}$, $n\in \mathbb{N}$, by iterating $\left( \left[ z^{\frac{1}{\alpha ^{n}}} \circ h \circ f^{\circ n}\right] _{w}\right) _{n}$, as $n\to + \infty $, we get $\left[ \varphi \right] _{w}$.
		
		Furthermore, if the initial condition $h\in \mathfrak L^{0}$ is chosen such that $\mathrm{Lb}_{z}\, (h)=\mathrm{Lb}_{z} \, (\varphi )$, then, by the definition of the power-metric topology, every block of $\varphi $ is revealed after finitely many iterations of $\left( z^{\frac{1}{\alpha ^{n}}} \circ h \circ f^{\circ n}\right) _{n} $. 
		\item By statement 3 of Theorem~A, $\mathrm{Supp}\, (\varphi )$ depends only on $\mathrm{Supp}\, (f)$. Therefore, $\mathrm{Supp}\, (\varphi )$ is independent of the support of the initial condition $h \in \mathfrak L^{0}$.
	\end{enumerate} 
\end{rem}

\begin{rem}\label{RemarkTheoremAGeneraliz}\hfill
	\begin{enumerate}[1., font=\textup, topsep=0.4cm, itemsep=0.4cm, leftmargin=0.6cm]
		\item In the proof of Theorem~A, we only use algebraic properties of the field of real numbers that also hold in the field of complex numbers. It is easy to see that Theorem~A also holds in the differential algebra $\mathfrak L(\mathbb{C})$ of logarithmic transseries with complex coefficients. 
		\item By the formal change of variables $\zeta =-\log z$, Equation \eqref{EqNorm} becomes \emph{Schr\"oder's equation} \eqref{SchrodEq} with the variable $\zeta $ at infinity:
		\begin{align}
			& \varphi (f(\zeta )) = \alpha \cdot \varphi (\zeta ) . \label{SchrodEq}
		\end{align} 
	\end{enumerate}
\end{rem}

\medskip

\subsection{The analytic part (Theorem~C)}

We state and prove Theorem~C in the $\zeta $-chart. The $\zeta $-chart is a global chart for $\widetilde{\mathbb{C}}$ and the definition of a standard quadratic domain is given explicitely in the $\zeta $-chart as a subset of $\mathbb{C}^{+}$.

\begin{thmC}[Analytic normalization of a strongly hyperbolic complex Dulac germ]\label{ThmB}
	Let $f$ be a strongly hyperbolic complex Dulac germ and let $\widehat{f}(\zeta )=\alpha \zeta + o(1)$, $\alpha \in \mathbb{R}_{>1}$, be its asymptotic expansion in the $\zeta $-chart. Then:
	\begin{enumerate}[1., font=\textup, topsep=0.4cm, itemsep=0.4cm, leftmargin=0.6cm]
		\item There exists the unique parabolic complex Dulac germ $\varphi$ (given in the $\zeta $-chart) which is a solution of the \emph{normalization equation}:
		\begin{align}
			\varphi \circ f & =\alpha \cdot \varphi . \label{EqAnalyticNorm}
		\end{align}
		Furthermore, if $f$ is a real Dulac germ, so is $\varphi $.
		\item $\varphi \sim \widehat{\varphi }(\zeta )$, uniformly as $\Re  (\zeta ) \to +\infty $, where $\widehat{\varphi}(\zeta )$ is the unique solution of normalization equation \eqref{EqNorm} in Theorem~A (in the formal variable $\zeta $ at infinity).
	\end{enumerate}
\end{thmC}
\begin{rem}\hfill
	\begin{enumerate}[1., font=\textup, topsep=0.4cm, itemsep=0.4cm, leftmargin=0.6cm]
		\item Theorem~C can be viewed as an analogue of \cite[Theorem~B]{PRRSDulac21} for strongly hyperbolic complex Dulac germs. The main difference between the two theorems is the fact that in the proof of Theorem~C we solve suitable \emph{Schr\"oder's type equations}, as opposed to \emph{Abel's type equations} solved in \cite{PRRSDulac21}.
		\item The unique solution $\varphi $ of analytic normalization equation \eqref{EqAnalyticNorm}, will be given, in the $\zeta $-chart, as the uniform limit of the \emph{Koenigs sequence} $(\frac{1}{\alpha ^{n}}f^{\circ n}(\zeta ))_{n}$ (see e.g. \cite{CarlesonG93}, \cite{Milnor06}). That is, in the standard $z$-chart,
		\begin{align*}
			\varphi (z) & =\lim _{n\to \infty }\left( z^{\frac{1}{\alpha ^{n}}}\circ f^{\circ n}\right) (z) .
		\end{align*}
	\end{enumerate}
\end{rem}

\medskip

\section{Proof of Theorem~A}\label{sec:proofThmA}

\subsection[Proof of statement 1 of Theorem~A]{Proof of statement 1 of Theorem~A (Existence and uniqueness of the normalization)}\label{SubsectionProofOf1}

As opposed to transfinite compositions of elementary changes of variables used in \cite{mrrz1} in the differential algebra $\mathcal{L}_{1}$, we use the fixed point techniques, as in \cite{PRRSFormal21}, to solve normalization equation \eqref{EqNorm}. More precisely, we define the so-called \emph{B\"ottcher operator} on the group $\mathfrak L^{0}$ and we transform normalization equation \eqref{EqNorm} into the fixed point equation (see Proposition~\ref{PropNormFixed} in Subsection~\ref{SubsubsectionTransformNormToFixed}).

The B\"ottcher operator is not itself a contraction in the power-metric on the groups $\mathcal{L}_{k}^{0}$, $k\in \mathbb{N}_{\geq 1}$ (see Proposition~\ref{PropContractibility}). Therefore, we first \emph{prenormalize} a given strongly hyperbolic logarithmic transseries in Subsection~\ref{SubsubsectionSolvingPrenom}. Then, in Subsection~\ref{SubsubsectionProofStatement1}, we restrict the domain of the B\"ottcher operator so that it becomes a contradiction. \\

In the sequel we assume that $f=z^{\alpha }+\mathrm{h.o.t.}$, for $\alpha >1$. Otherwise, we consider the inverse $f^{-1}=z^{\frac{1}{\alpha }}+\mathrm{h.o.t.}$

\smallskip

\subsubsection{Transformation of a normalization equation into a fixed point equation}\label{SubsubsectionTransformNormToFixed}

As in \cite{PRRSFormal21}, in the next proposition, we relate a solution of normalization equation $\varphi  \circ f \circ \varphi ^{-1} = z^{\alpha }$ to finding a fixed point of an operator.
\begin{prop}\label{PropNormFixed}
	Let $f\in \mathfrak L^{H}$ be such that $f=z^{\alpha } + \mathrm{h.o.t.}$, $\alpha \in \mathbb{R}_{>1}$, and let $\mathcal{P}_{f}:\mathfrak L^{0}\to \mathfrak L^{0}$ be the operator defined by
	\begin{align}
		\mathcal{P}_{f}(h) := z^{\frac{1}{\alpha }} \circ h\circ f , \quad h \in \mathfrak L^{0} . \label{DefinitionPf}
	\end{align}
	Then, $\varphi \in \mathfrak L^{0}$ is a solution of the normalization equation \eqref{EqNorm} if and only if $\varphi $ is a fixed point of the operator $\mathcal{P}_{f}$.
\end{prop}
\begin{proof}
	Directly from \eqref{DefinitionPf} and the normalization equation \eqref{EqNorm}.
\end{proof}

\begin{rem}
	Note that operator $\mathcal{P}_{f}$ defined in \eqref{DefinitionPf} is an analogue of the \emph{generalized Koenigs operator} defined in \cite{PRRSFormal21}. We call operator $\mathcal{P}_{f}$ the \emph{B\"ottcher operator} because of the B\"ottcher Theorem (see e.g. \cite{CarlesonG93}, \cite{Milnor06}) which is a strongly hyperbolic analogue of the classical \emph{Koenigs Theorem} (see e.g. \cite{CarlesonG93}, \cite{Milnor06}) in the case of analytic diffeomorphisms.
\end{rem}

\begin{prop}\label{PropNormOrder}
	Let $f\in \mathfrak L^{H}$ be such that $f=z^{\alpha } + \mathrm{h.o.t.}$, for $\alpha \in \mathbb{R}_{>1}$, $\beta :=\mathrm{ord}_{z} \, (f-z^{\alpha })-\alpha +1$, and let $\varphi  \in \mathfrak L^{0}$ be a solution of normalization equation \eqref{EqNorm}. Then $\mathrm{ord}_{z} \, (\varphi -\mathrm{id}) \geq \beta $.
\end{prop}
\begin{proof}
	If $\beta =1$, then obviously $\mathrm{ord}_{z}\, (\varphi -\mathrm{id})\geq \beta $. Suppose that $\beta >1$. Let $\varphi $ be a solution of the normalization equation \eqref{EqNorm}. Put $\varphi _{1}:=\varphi -\mathrm{id}$ and $f_{1}:=f-z^{\alpha }$. Note that $\mathrm{ord}_{z} \, (f_{1})=\alpha + \beta -1$. From the normalization equation \eqref{EqNorm} and Taylor Theorem we get:
	\begin{align}
		z^{\frac{1}{\alpha }} \circ (z^{\alpha }+f_{1}+\varphi _{1}\circ f) &= \mathrm{id}+\varphi _{1} , \nonumber \\
		z\Big(1+\frac{f_{1}}{z^{\alpha }}+ \frac{\varphi _{1}(z^{\alpha })}{z^{\alpha }}+\sum _{i\geq 1}\frac{\varphi _{1}^{(i)}(z^{\alpha })}{z^{\alpha }i!}f_{1}^{i}\Big) ^{\frac{1}{\alpha }} &= \mathrm{id}+\varphi _{1} . \label{IdentityOrder}
	\end{align}
	Suppose that $\mathrm{ord}_{z} \, \, (\varphi -\mathrm{id}) < \beta $, i.e., $\mathrm{ord}_{z} \, \, (\varphi _{1}) < \beta $. Note that $\mathrm{ord} \, \Big( \frac{f_{1}}{z^{\alpha -1}}\Big) =\beta $. Since $\mathrm{ord}_{z} \, \, (\varphi -\mathrm{id}) < \beta $, by \eqref{IdentityOrder} and the second formula in \eqref{CompEq1}, it follows that
	\begin{align*}
		\mathrm{id}+ \mathrm{Lt} \, \Big( \frac{\varphi _{1}(z^{\alpha })}{\alpha z^{\alpha -1}}\Big) + \mathrm{h.o.t.} &= \mathrm{id}+\varphi _{1} ,
	\end{align*}
	i.e.,
	\begin{align}
		\mathrm{Lt} \, \Big( \frac{\varphi _{1}(z^{\alpha })}{\alpha z^{\alpha -1}}\Big) & = \mathrm{Lt}\left( \varphi _{1}\right) . \label{EquationLeadingTerm}
	\end{align}
	Therefore, the order (in $z$) of $\frac{\varphi _{1}(z^{\alpha })}{\alpha z^{\alpha -1}}$ is equal to the order (in $z$) of $\varphi _{1}$, which implies that $\alpha \cdot \mathrm{ord}_{z}\, (\varphi _{1})-\alpha +1=\mathrm{ord}_{z}\, (\varphi _{1})$, i.e., $\mathrm{ord}_{z}\, (\varphi _{1})=1$. Since $\mathrm{ord}\, (\varphi _{1})>(1,\mathbf{0}_{k})$, by Lemma $A.3.2$ in \cite{Peran21Thesis} and by multiplication by $\frac{1}{\alpha }$, we get $\mathrm{Lt} \, \Big( \frac{\varphi _{1}(z^{\alpha })}{\alpha z^{\alpha -1}}\Big) = \frac{1}{\alpha ^{n+1}}\mathrm{Lt} \, (\varphi _{1})$, for some $n\in \mathbb{N}$, which is a contradiction with \eqref{EquationLeadingTerm}. Therefore, $\mathrm{ord}_{z} \, (\varphi -\mathrm{id}) \geq \beta $.
\end{proof}

Let $\beta \geq 1$ and let $\mathrm{id}+\mathcal{L}_{k}^{\beta }$, $k\in \mathbb{N}$, be the set of all parabolic logarithmic transseries $h \in \mathcal{L}_{k}^{0}$, such that $\mathrm{ord}_{z}(h-\mathrm{id}) \geq \beta $. In \cite[Proposition 6.2]{PRRSFormal21} it is proved that $(\mathcal{L}_{k},d_{z})$ is complete. Thus, $\mathrm{id}+\mathcal{L}_{k}^{\beta }$, $\beta \geq 1$, are complete spaces as its closed subspaces.

\begin{prop}[Contractibility of the operator $\mathcal{P}_{f}$]\label{PropContractibility}
	Let $f\in \mathcal{L}_{m}^{H}$, $m\in \mathbb{N}$, be such that $f=z^{\alpha }+\mathrm{h.o.t.}$, $\alpha \in \mathbb{R}_{>1}$, and let $\mathcal{P}_{f}$ be the B\"ottcher operator defined by \eqref{DefinitionPf}. Then:
	\begin{enumerate}[1., font=\textup, topsep=0.4cm, itemsep=0.4cm, leftmargin=0.6cm]
		\item The space $\mathrm{id}+\mathcal{L}_{k}^{\beta }$ is $\mathcal{P}_{f}$-invariant, for every $1 \leq \beta \leq \mathrm{ord}_{z} \, (f-z^{\alpha })-\alpha +1$ and $k\geq m$.
		\item For every $k\geq m$, the operator $\mathcal{P}_{f}$ is a contraction on the space $(\mathrm{id}+\mathcal{L}_{k}^{\beta },d_{z})$ if and only if $\beta >1$. In that case, the operator $\mathcal{P}_{f}$ is a $\frac{1}{2^{(\alpha -1)(\beta -1)}}$-contraction on the space $(\mathrm{id}+\mathcal{L}_{k}^{\beta },d_{z})$, $k\geq m$.
	\end{enumerate}
\end{prop}
\begin{proof}
	1. Let $f_{1}:=f-z^{\alpha }$ and let $h \in \mathrm{id}+\mathcal{L}_{k}^{\beta }$. Put $h_{1}:=h-\mathrm{id}$. Now, by Taylor Theorem we get:
	\begin{align}
		\mathcal{P}_{f}(h) &= z^{\frac{1}{\alpha }} \circ h \circ f \nonumber \\
		&= z^{\frac{1}{\alpha }} \circ (z^{\alpha }+f_{1}+h_{1}\circ f) \nonumber \\
		&= z\Big(1+ \frac{f_{1}}{z^{\alpha }}+ \frac{h_{1}(z^{\alpha })}{z^{\alpha }}+\sum _{i\geq 1}\frac{h_{1}^{(i)}(z^{\alpha })}{z^{\alpha }i!}f_{1}^{i}\Big) ^{\frac{1}{\alpha }} . \label{Equati3}
	\end{align}
	Let
	\begin{align*}
		\mathcal{N}(h) & :=\frac{f_{1}}{z^{\alpha }}+ \frac{h_{1}(z^{\alpha })}{z^{\alpha }}+\sum _{i\geq 1}\frac{h_{1}^{(i)}(z^{\alpha })}{z^{\alpha }i!}f_{1}^{i} .
	\end{align*}
	By the second formula in \eqref{CompEq1}, we get:
	\begin{align}
		\mathcal{P}_{f}(h) &= \mathrm{id}+z\sum _{j\geq 1}{\frac{1}{\alpha } \choose j}(\mathcal{N}(h))^{j} . \label{Eqq11}
	\end{align}
	It is obvious that $\mathcal{P}_{f}(h) \in \mathcal{L}_{k}^{0}$ and since $\mathrm{ord}_{z} \, (h_{1}) \geq \beta $ and $\mathrm{ord}_{z} \, (f_{1}) \geq \alpha + \beta -1$, it follows that:
	\begin{align}
		\mathrm{ord}_{z} \, \Big( \frac{f_{1}}{z^{\alpha -1}}\Big) & \geq \beta , \label{Eqq222}
	\end{align}
	\begin{align}
		\mathrm{ord}_{z} \, \Big( \frac{h_{1}(z^{\alpha })}{z^{\alpha -1}}\Big) &= \alpha \cdot \mathrm{ord}_{z} \, (h_{1}) - \alpha +1 \nonumber \\
		& = \mathrm{ord}_{z} \, (h_{1}) + (\alpha -1)\cdot \mathrm{ord}_{z} \, (h_{1})-(\alpha -1) \nonumber \\
		& \geq \mathrm{ord}_{z} \, (h_{1}) + (\alpha -1)(\beta -1) \label{Eqq22} 
	\end{align}
	and
	\begin{align}
		\mathrm{ord}_{z} \, \Big( \sum _{i\geq 1}\frac{h_{1}^{(i)}(z^{\alpha })}{z^{\alpha -1}i!}f_{1}^{i} \Big) &\geq \alpha \cdot (\mathrm{ord}_{z} \, (h_{1}) -1)-(\alpha -1)+\mathrm{ord}_{z} \, (f_{1}) \nonumber \\
		& \geq \mathrm{ord}_{z} \, (h_{1}) + (\alpha -1)\cdot \mathrm{ord}_{z} \, (h_{1})-(\alpha -1) \nonumber \\
		& \geq \mathrm{ord}_{z} \, (h_{1}) + (\alpha -1)(\beta -1) . \label{Eq33} 
	\end{align}
	Note that, if $\mathrm{ord}_{z}\, (h_{1})=\beta $, equality holds in \eqref{Eqq22}. From \eqref{Eqq222}, \eqref{Eqq22} and \eqref{Eq33}, we get that $\mathrm{ord}_{z} \, (z\mathcal{N}(h)) \geq \beta $, and, therefore, by \eqref{Eqq11},
	\begin{align}
		\mathrm{ord}_{z} \, \big( \mathcal{P}_{f}(h)-\mathrm{id} \big) \geq \beta . \nonumber 
	\end{align}
	This implies that $\mathcal{P}_{f}(\mathrm{id}+\mathcal{L}_{k}^{\beta }) \subseteq \mathrm{id}+\mathcal{L}_{k}^{\beta }$. \\
	
	2. Let $\mathrm{id}+h_{1}, \mathrm{id}+h_{2} \in \mathrm{id}+\mathcal{L}_{k}^{\beta }$. Using \eqref{Eqq11}, \eqref{Eqq22}, \eqref{Eq33}, we get:
	\begin{align}
		\mathrm{ord}_{z} \, \big( \mathcal{P}_{f}(\mathrm{id}+h_{1} \big) -\mathcal{P}_{f}(\mathrm{id}+h_{2})) &= \mathrm{ord}_{z} \, \Big( \frac{1}{\alpha }z\mathcal{N}(h_{1})- \frac{1}{\alpha }z\mathcal{N}(h_{2}) \Big) \nonumber \\
		& \geq \mathrm{ord}_{z} \, (h_{1}-h_{2}) + (\alpha -1)(\beta -1) . \label{Eqq44}
	\end{align}
	From \eqref{Eqq44}, we conclude that $\mathcal{P}_{f}$ is a $\frac{1}{2^{(\alpha -1)(\beta -1)}}$-Lipschitz map. Suppose that $\mathrm{ord}_{z}\, (h_{1}-h_{2})=\beta $. Then, the equality holds in \eqref{Eqq22}. Hence, the equality holds also in \eqref{Eqq44}. Therefore, $\frac{1}{2^{(\alpha -1)(\beta -1)}}$ is the minimal Lipschitz coefficient\footnote{Let $(X,d)$ be a metric space and $\mathcal{S}:X\to X$ a \emph{Lipschitz map}. If there exists $\lambda >0$ minimal such that $d(\mathcal{S}(x_{1}),\mathcal{S}(x_{2})) \leq \lambda \cdot d(x_{1},x_{2})$, for $x_{1},x_{2}\in X$, then we call $\lambda $ the \emph{minimal Lipschitz coefficient} of $\mathcal{S}$.} of $\mathcal{P}_{f}$. Consequently, it follows that $\mathcal{P}_{f}$ is a contraction on the space $\mathrm{id}+\mathcal{L}_{k}^{\beta }$ if and only if $\beta >1$. In that case, $\mathcal{P}_{f}$ is a $\frac{1}{2^{(\alpha -1)(\beta -1)}}$-contraction on the space $\mathrm{id}+\mathcal{L}_{k}^{\beta }$.
\end{proof}

Let $f\in \mathcal{L}_{k}^{H}$, $k\in \mathbb{N}$, be such that $f=z^{\alpha }+\mathrm{h.o.t.}$, $\alpha >1$, and let $\beta := \mathrm{ord}_{z}(f-z^{\alpha })- \alpha +1$. By Proposition~\ref{PropContractibility}, the operator $\mathcal{P}_{f}$ is a contraction if and only if $\beta >1$, i.e., $\mathrm{ord}_{z}(f-z^{\alpha }) > \alpha $.

Let $m\geq k$ be arbitrary. Therefore, if $\mathrm{ord}_{z}(f-z^{\alpha }) > \alpha $, by the Banach Fixed Point Theorem, there exists a unique fixed point $\varphi  \in \mathrm{id}+ \mathcal{L}_{m}^{\beta }$ of the operator $\mathcal{P}_{f}$, which implies, by Proposition~\ref{PropNormFixed}, that $\varphi $ is the unique solution of the normalization equation \eqref{EqNorm} in the space $\mathrm{id}+\mathcal{L}_{m}^{\beta }$, for each $m\geq k$. Now, by Proposition~\ref{PropNormOrder}, it follows that $\varphi $ is unique solution of Equation \eqref{EqNorm} in the space $\mathfrak L^{0}$.

On the other hand, suppose that $\mathrm{ord}_{z}(f-z^{\alpha }) = \alpha $, i.e., $\beta =1$. By Proposition~\ref{PropContractibility}, the operator $\mathcal{P}_{f}$ is not a contraction on the space $(\mathrm{id}+\mathcal{L}_{k}^{1},d_{z})$. Consequently, we split the proof of the existence of a normalization in the following two steps:
\begin{enumerate}[1., font=\textup, topsep=0.4cm, itemsep=0.4cm, leftmargin=0.6cm]
	\item \emph{The prenormalization.} We solve a \emph{prenormalization equation}
	\begin{align}
		\varphi _{1}\circ f\circ \varphi _{1}^{-1} & = z^{\alpha }+ \mathrm{h.o.b.}(z) , \label{EqPreNorm}
	\end{align}
	where $\varphi _{1} \in \mathcal{L}_{k}^{0}$. Furthermore, $\varphi _{1} \in \mathcal{L}_{k}^{0}$ is not unique in $\mathcal{L}_{k}^{0}$ since Equation \eqref{EqPreNorm} depends on higher order blocks. If we additionally impose the \emph{canonical form} of $\varphi _{1}$, i.e., $\varphi _{1}:= \mathrm{id}+zS$, for a $1$-block $zS$, then $\varphi _{1}$ is unique. This is proved in Proposition~\ref{PropSolutionPrenorm} in Subsection~\ref{SubsubsectionSolvingPrenom}.
	\item \emph{The normalization.} Using the contractibility of the B\"ottcher operator $\mathcal{P}_{\varphi _{1}\circ f\circ \varphi _{1}^{-1}}$ we solve the normalization equation
	\begin{align*}
		\varphi _{2}\circ \left( \varphi _{1}\circ f\circ \varphi _{1}^{-1} \right) \circ \varphi _{2}^{-1} & = z^{\alpha } ,
	\end{align*}
	where $\varphi _{2} \in \mathrm{id}+\mathcal{L}_{k}^{\beta }$, for $\beta :=\mathrm{ord}_{z}(\varphi _{1}\circ f\circ \varphi _{1}^{-1} -z^{\alpha })- \alpha +1$. The proof of existence and uniqueness of the normalization is in Subsection~\ref{SubsubsectionProofStatement1}.
\end{enumerate}

\smallskip

\subsubsection{Differential algebras of blocks}\label{SubsubsectionBlocks}

For the purpose of solving a prenormalization equation \eqref{EqPreNorm}, we briefly recall the definitions of spaces of blocks $\mathcal{B}_{m}$, $m\in \mathbb{N} $, introduced in \cite[Subsection 3.4]{PRRSFormal21}.

For a fixed $k\in \mathbb{N}_{\geq 1}$, let $\mathcal{B}_{m}$, $m\in \left\lbrace 1,\ldots ,k\right\rbrace $, be the set of all $R \in \mathcal{L}_{k}^{\infty }$, such that:
\begin{align*}
	R & :=\sum _{(n_{m},\ldots ,n_{k}) \in \mathrm{Supp}\, (R)}a_{n_{m},\ldots ,n_{k}}\boldsymbol{\ell}_{m}^{n_{m}}\cdots \boldsymbol{\ell}_{k}^{n_{k}} .
\end{align*}
If $R\neq 0$, we define the \emph{order of} $R$ \emph{in} $\boldsymbol{\ell }_{m}$ as the minimal exponent of $R$ in $\boldsymbol{\ell }_{m}$ and denote it by $\mathrm{ord}_{\boldsymbol{\ell }_{m}} (R)$. If $R=0$, we simply put $\mathrm{ord}_{\boldsymbol{\ell }_{m}} (R):=+\infty $. Note that $\mathcal{B}_{m}$, $m\in \left\lbrace 1,\ldots ,k\right\rbrace $, is a subalgebra of $\mathcal{L}_{k}^{\infty }$.

We define $d_{m}:\mathcal{B}_{m}\times \mathcal{B}_{m} \to \mathbb{R}$, $m\in \left\lbrace 1,\ldots ,k\right\rbrace $, by:
$$
d_{m}(R,S):=\left\lbrace \begin{array}{ll}
	2^{-\mathrm{ord}_{\boldsymbol{\ell }_{m}}(R-S)}, & R\neq S, \\
	0, & R=S .
\end{array} \right. 
$$
By \cite[Proposition 6.2]{PRRSFormal21} the metric spaces $(\mathcal{B}_{m},d_{m})$, $m\in \left\lbrace 1,\ldots ,k\right\rbrace $, are complete.

We define derivations $D_{m}:\mathcal{B}_{m} \to \mathcal{B}_{m}$, $m\in \left\lbrace 1,\ldots ,k\right\rbrace $, by
\begin{align*}
	D_{m} & :=\boldsymbol{\ell }_{m}^{2}\cdot \frac{d}{d\boldsymbol{\ell }_{m}} .
\end{align*}
It is easy to see that $D_{m}$ satisfies the \emph{Newton-Leibnitz formula}, so $\mathcal{B}_{m}$ is a differential algebra. Note that the derivation $D_{m}$ is a $\frac{1}{2}$-contraction on the space $(\mathcal{B}_{m},d_{m})$. This is the reason why we choose $D_{m}$ over the derivation $\frac{d}{d\boldsymbol{\ell }_{m}}$, $m\in \left\lbrace 1,\ldots ,k\right\rbrace $. \\

Let
\begin{align*}
	\mathcal{B}_{m}^{+} := \left\lbrace R \in \mathcal{B}_{m}\subseteq \mathcal{L}_{k}^{\infty } : \mathrm{ord}_{\boldsymbol{\ell }_{m}}(R) \geq  1\right\rbrace 
\end{align*}
and
\begin{align*}
	\mathcal{B}_{\geq m}^{+} := \left\lbrace R \in \mathcal{B}_{m}\subseteq \mathcal{L}_{k}^{\infty } : \mathrm{ord} \, (R) > \boldsymbol{0}_{k+1}\right\rbrace ,
\end{align*}
for $m\in \left\lbrace 1,\ldots ,k\right\rbrace $ and $k\in \mathbb{N}_{\geq 1}$.

For every $m\in \left\lbrace 1,\ldots ,k\right\rbrace $, note that $\mathcal{B}_{m}^{+}\subseteq \mathcal{B}_{\geq m}^{+}$, and $\mathcal{B}_{m}^{+}$ and $\mathcal{B}_{\geq m}^{+}$ are subalgebras of the differential algebra $\mathcal{B}_{m}$.

It is easy to see that the following direct decomposition holds:
\begin{align*}
	\mathcal{B}_{\geq m}^{+} = \mathcal{B}_{m}^{+} \oplus \cdots \oplus \mathcal{B}_{k}^{+} ,
\end{align*}
for $m\in \left\lbrace 1, \ldots ,k\right\rbrace $.

Finally, note that $\mathcal{B}_{m}$, $m\in \left\lbrace 1, \ldots ,k\right\rbrace $, depends on the chosen $k\in \mathbb{N}_{\geq 1}$, but in the sequel it will be clear from the context which $k$ we consider.

\smallskip

\subsubsection{Step 1 - solving a prenormalization equation}\label{SubsubsectionSolvingPrenom}

Recall that we prove normalization step directly using the contractibility of the B\"ottcher operator from Proposition~\ref{PropContractibility} if $\beta >1$. In case $\beta =1$ we use the same idea for the prenormalization. For that purpose we define the analogue of the \emph{B\"ottcher operator on the set} $\mathrm{id}+z\mathcal{B}_{\geq 1}\subseteq \mathcal{L}_{k}^{0}$.

\begin{lem}\label{LemmaPfPrenorm}
	Let $f\in \mathcal{L}_{k}^{H}$, $k\in \mathbb{N}_{\geq 1}$, be such that $f=z^{\alpha }+z^{\alpha }R_{\alpha }+\mathrm{h.o.b.}(z)$, for $\alpha >1$, $R_{\alpha }\in \mathcal{B}^{+}_{\geq 1}\setminus \left\lbrace 0\right\rbrace \subseteq \mathcal{L}_{k}$. Let $\mathcal{R}_{f}:\mathrm{id}+z\mathcal{B}_{\geq 1}^{+} \to \mathrm{id}+z\mathcal{B}_{\geq 1}^{+}$, be the operator defined by:
	\begin{align}
		\mathcal{R}_{f}(\mathrm{id}+zT) := z^{\frac{1}{\alpha}} \circ (\mathrm{id}+zT) \circ (z^{\alpha }+z^{\alpha }R_{\alpha }) , \quad T\in \mathcal{B}_{\geq 1}^{+} \subseteq \mathcal{L}_{k} . \label{OperatorR}
	\end{align}
	Logarithmic transseries $\varphi _{1}=\mathrm{id}+zS$, $S \in \mathcal{B}_{\geq 1}^{+}\subseteq \mathcal{L}_{k}$, is a solution of a prenormalization equation $\varphi _{1} \circ f\circ \varphi _{1}^{-1}=z^{\alpha }+\mathrm{h.o.b.}(z)$, if and only if $\varphi _{1}$ is a fixed point of the operator $\mathcal{R}_{f}$. 
\end{lem}
\noindent The proof of Lemma~\ref{LemmaPfPrenorm} is technical and can be found in \cite[Lemma~$2.2.8$]{Peran21Thesis}. \\

In the next examples we show that the operator $\mathcal{R}_{f}$, defined by \eqref{OperatorR}, is not a contraction on the space $\mathrm{id}+z\mathcal{B}_{\geq 1}^{+}$, with respect to the standard metrics.

\begin{example}[Non-contractibility of the operator $\mathcal{R}_{f}$]\hfill
	\begin{enumerate}[1., font=\textup, topsep=0.4cm, itemsep=0.4cm, leftmargin=0.6cm]
		\item Take metric $d$ on the space $\mathrm{id}+z\mathcal{B}_{\geq 1}^{+}\subseteq \mathcal{L}_{k}^{0}$, $k\in \mathbb{N}_{\geq 1}$, defined by:
		$$
		d(\mathrm{id}+zR,\mathrm{id}+zS):=\left\lbrace \begin{array}{ll}
			2^{-\mathrm{ord}_{\boldsymbol{\ell }_{1}}(R-S)}, & R\neq S, \\
			0, & R=S .
		\end{array} \right. 
		$$
		Now, consider $\mathrm{id}$ and $\mathrm{id}+z\boldsymbol{\ell }_{1}$. It is easy to see that $d(\mathrm{id},\mathrm{id}+z\boldsymbol{\ell}_{1})=\frac{1}{2}$. Take $f:=z^{\alpha }+z^{\alpha }R_{\alpha }+\mathrm{h.o.b}(z)$, $\alpha \in \mathbb{R}_{>1}$, $R_{\alpha }\in \mathcal{B}_{\geq 1}^{+}\subseteq \mathcal{L}_{k}$, $\mathrm{ord}_{\boldsymbol{\ell }_{1}}(R_{\alpha }) \geq 2$. Note that:
		\begin{align*}
			\mathcal{R}_{f}(\mathrm{id}) &= \mathrm{id}+z\sum _{i\geq 1}{\frac{1}{\alpha } \choose i}R_{\alpha }^{i} ,
		\end{align*}
		and
		\begin{align*}
			\mathcal{R}_{f}(\mathrm{id}+z\boldsymbol{\ell }_{1}) &= z^{\frac{1}{\alpha }}\circ (\mathrm{id}+z\boldsymbol{\ell }_{1}) \circ (z^{\alpha}+z^{\alpha }R_{\alpha }) \\
			&= \mathrm{id} + \frac{1}{\alpha ^{2}}z\boldsymbol{\ell }_{1} + \mathrm{h.o.t.}
		\end{align*}
		Therefore, $\mathcal{R}_{f}(\mathrm{id}+z\boldsymbol{\ell }_{1}) - \mathcal{R}_{f}(\mathrm{id}) = \frac{1}{\alpha ^{2}}z\boldsymbol{\ell }_{1}+\mathrm{h.o.t.}$, which implies that $d(\mathcal{R}_{f}(\mathrm{id}+z\boldsymbol{\ell }_{1}),\mathcal{R}_{f}(\mathrm{id}))=\frac{1}{2}$. Thus, $\mathcal{R}_{f}$ is not a contraction on the space $\mathrm{id}+z\mathcal{B}_{\geq 1}^{+}\subseteq \mathcal{L}_{k}^{0}$, with respect to the metric $d$.
		\item Take $\mathcal{B}_{\geq 1}^{+} \subseteq \mathcal{L}_{1}$ and let $d$ be the metric on the space $\mathrm{id}+ z\mathcal{B}_{\geq 1}^{+}\subseteq \mathcal{L}_{1}^{0}$, defined by:
		\begin{align}
			d(\mathrm{id}+zT_{1},\mathrm{id}+zT_{2}) &:= \sum _{i=1}^{+\infty }\frac{1}{2^{i}}\cdot \frac{\left| a_{1,i}-a_{2,i}\right| }{1+\left| a_{1,i}-a_{2,i}\right| } , \nonumber 
		\end{align}
		where $T_{1}:=\sum _{i=1}^{+ \infty }a_{1,i}\boldsymbol{\ell }_{1}^{i}$, $T_{2}:=\sum _{i=1}^{+ \infty }a_{2,i}\boldsymbol{\ell }_{1}^{i}$. Let $f:=z^{\alpha }+az^{\alpha }\boldsymbol{\ell }_{1}$, $\alpha \in \mathbb{R}_{>1}$, $a\in \mathbb{R}_{>0}$. Now, consider $\mathrm{id}$ and $\mathrm{id}+z\boldsymbol{\ell }_{1}$. It is easy to see that $d(\mathrm{id},\mathrm{id}+z\boldsymbol{\ell }_{1})=\frac{1}{4}$. Notice that:
		\begin{align}
			& \mathcal{R}_{f}(\mathrm{id}) = \mathrm{id}+z\sum _{i\geq 1}{\frac{1}{\alpha } \choose i}(a\boldsymbol{\ell }_{1})^{i} = \mathrm{id} + \frac{a}{\alpha }z\pmb{\ell }_{1} + \frac{1}{2\alpha }\Big( \frac{1}{\alpha }-1 \Big) a^{2}z\boldsymbol{\ell }_{1}^{2}+\mathrm{h.o.t.} \label{ExampleEq2}
		\end{align}
		and
		\begin{align}
			\mathcal{R}_{f}(\mathrm{id}+z\boldsymbol{\ell }_{1}) &= z^{\frac{1}{\alpha }}\circ (\mathrm{id}+z\boldsymbol{\ell }_{1}) \circ (z^{\alpha}+az^{\alpha }\boldsymbol{\ell }_{1}) \nonumber \\
			&= \mathrm{id} + \frac{1}{\alpha }\Big( a+\frac{1}{\alpha }\Big) z\boldsymbol{\ell }_{1} + \Bigg( \frac{1}{2\alpha }\Big( \frac{1}{\alpha }-1\Big) \Big( a+\frac{1}{\alpha }\Big) ^{2}+\frac{a}{\alpha ^{2}} \Bigg) z\boldsymbol{\ell }_{1}^{2} + \mathrm{h.o.t.} \label{ExampleEq3}
		\end{align}
		By \eqref{ExampleEq2} and \eqref{ExampleEq3}, it follows that:
		\begin{align*}
			d\big( \mathcal{R}_{f}(\mathrm{id}) , \mathcal{R}_{f}(\mathrm{id}+z\boldsymbol{\ell }_{1}) \big) & \geq \frac{1}{2}\cdot \frac{\frac{1}{\alpha ^{2}}}{1+ \frac{1}{\alpha ^{2}}} + \frac{1}{4} \cdot \frac{\left| \Big( \frac{1}{\alpha }-1\Big) \Big( \frac{a}{\alpha ^{2}}+\frac{1}{2\alpha ^{3}} \Big) +\frac{a}{\alpha ^{2}}\right| }{1+\left| \Big( \frac{1}{\alpha }-1\Big) \Big( \frac{a}{\alpha ^{2}}+\frac{1}{2\alpha ^{3}}\Big) +\frac{a}{\alpha ^{2}}\right| } \, .  
		\end{align*}
		As $\alpha $ tends to $1$ and $a$ tends to $+\infty $, $d(\mathcal{R}_{f}(\mathrm{id}) , \mathcal{R}_{f}(\mathrm{id}+z\boldsymbol{\ell }_{1}) )$ tends to $\frac{1}{2} \cdot \frac{1}{2}+ \frac{1}{4}\cdot 1= \frac{1}{2}$. Since $d(\mathrm{id},\mathrm{id}+z\boldsymbol{\ell }_{1})=\frac{1}{4}$, there exists sufficiently small $\alpha \in \mathbb{R}_{>1}$ and sufficiently big $a\in \mathbb{R}_{>0}$, such that $\mathcal{R}_{f}$ is not a contraction on the space $(\mathrm{id}+z\mathcal{B}_{\geq 1}^{+},d)$.
	\end{enumerate}
\end{example} 

Since $\mathcal{R}_{f}$ is not a contraction in any of the standard metrics, in Lemma~\ref{LemmaPrenormTS} we reformulate the prenormalization equation by introducing the operators $\mathcal{T}_{f}$ and $\mathcal{S}_{f}$, and applying the fixed point theorem from Proposition~\ref{PropFixedPoint} below.

\begin{prop}[A fixed point theorem]\label{PropFixedPoint}
	Let $X$ be a complete metric space, $\mathcal{T}:X\to X$ an isometry and a surjection, and $\mathcal{S}:X\to X$ a contraction. Then the equation $\mathcal{T}(x)=\mathcal{S}(x)$, $x\in X$, has a unique solution.
	
	Moreover, the solution $x\in X$ is the limit of the \emph{Picard sequence}, i.e.,
	\begin{align*}
		x & = \lim _{n\to \infty }\left( \left( \mathcal{T}^{-1}\circ \mathcal{S}\right) ^{\circ n}(y)\right) _{n} ,
	\end{align*}
	for any initial point $y\in X$.
\end{prop}

\noindent Proposition~\ref{PropFixedPoint} is a special case of \cite[Proposition 3.2]{PRRSFormal21}, which is an easy consequence of the \emph{Banach Fixed Point Theorem}, and is motivated by the \emph{Krasnoselskii Fixed Point Theorem} (see e.g. \cite{XiangG16}).

In Lemma~\ref{LemmaPrenormTS} we transform prenormalization equation \eqref{EqPreNorm} into a fixed point equation.

\begin{lem}[A fixed point equation for prenormalization]\label{LemmaPrenormTS}
	Let $f\in \mathcal{L}_{k}^{H}$, $k\in \mathbb{N}_{\geq 1}$, be such that $f= z^{\alpha }+z^{\alpha }R_{\alpha }+\mathrm{h.o.b}(z)$, for $\alpha >1$, $R_{\alpha }\in \mathcal{B}_{\geq 1}^{+}\subseteq \mathcal{L}_{k}$. A logarithmic transseries $\varphi _{1}=\mathrm{id}+zS$, $S \in \mathcal{B}_{\geq 1}^{+}\subseteq \mathcal{L}_{k}$, is a solution of the prenormalization equation
	\begin{align}
		\varphi _{1} \circ f\circ \varphi _{1}^{-1} = z^{\alpha }+\mathrm{h.o.b}(z) \nonumber 
	\end{align}
	if and only if $S$ is a solution of the \emph{fixed point equation}
	\begin{align}
		\mathcal{T}_{f}(S) = \mathcal{S}_{f}(S) , \label{EqFixedPoint}
	\end{align}
	where $\mathcal{T}_{f}, \mathcal{S}_{f}:\mathcal{B}_{\geq 1}^{+} \to \mathcal{B}_{\geq 1}^{+}$ are the operators defined by:
	\begin{align}
		\mathcal{T}_{f}(S) := S \circ z^{\alpha } + (S\circ z^{\alpha })R_{\alpha }- \sum _{i\geq 1}{\alpha \choose i}S^{i}, \quad S \in \mathcal{B}_{\geq 1}^{+} \subseteq \mathcal{L}_{k} , \label{Tdef}
	\end{align}
	and
	\begin{align}
		\mathcal{S}_{f}(S) := - R_{\alpha } - ((D_{1}S) \circ z^{\alpha })R_{\alpha } - \mathcal{K}_{f}(S) , \quad S \in \mathcal{B}_{\geq 1}^{+} \subseteq \mathcal{L}_{k} , \label{Sdef}
	\end{align}
	where $\mathcal{K}_{f}$ is a suitable $\frac{1}{2}$-contraction on the space $(\mathcal{B}_{\geq 1}^{+},d_{1})$ whose explicit form is revealed in the proof.
\end{lem}

\noindent The proof of Lemma~\ref{LemmaPrenormTS} is technical and can be found in \cite[Lemma~$2.2.10$]{Peran21Thesis}. \\

In Lemma~\ref{LemmaConditionsPrenorm} we prove that the operators $\mathcal{T}_{f}$ and $\mathcal{S}_{f}$ defined by \eqref{Tdef} and \eqref{Sdef}, respectively, satisfy the assumptions of Proposition~\ref{PropFixedPoint}.

\begin{lem}[Properties of the operators $\mathcal{T}_{f}$ and $\mathcal{S}_{f}$]\label{LemmaConditionsPrenorm}
	Let $f\in \mathcal{L}_{k}^{H}$, $k\in \mathbb{N}_{\geq 1}$, be such that $f=z^{\alpha }+z^{\alpha }R_{\alpha }+\mathrm{h.o.b.}(z)$, $\alpha \in \mathbb{R}_{>1}$, and $R_{\alpha }\in \mathcal{B}_{\geq 1}^{+}\subseteq \mathcal{L}_{k}$. Let $\mathcal{T}_{f}, \mathcal{S}_{f}:\mathcal{B}_{\geq 1}^{+}\to \mathcal{B}_{\geq 1}^{+}$ be the operators defined by \eqref{Tdef} and \eqref{Sdef}, respectively. Then:
	\begin{enumerate}[1., font=\textup, topsep=0.4cm, itemsep=0.4cm, leftmargin=0.6cm]
		\item $\mathcal{T}_{f}$ is an isometry and a surjection on the space $(\mathcal{B}_{\geq 1}^{+},d_{1})$,
		\item $\mathcal{S}_{f}$ is a $\frac{1}{2}$-contraction on the space $(\mathcal{B}_{\geq 1}^{+},d_{1})$.
	\end{enumerate}
\end{lem}

\noindent The proof of Lemma~\ref{LemmaConditionsPrenorm} is technical and can be found in \cite[Proposition~$2.2.11$]{Peran21Thesis}. \\

In Proposition~\ref{PropSolutionPrenorm} below, we prove existence and uniqueness of the prenormalization $\varphi _{1}$ given in the \emph{canonical form}.

\begin{prop}[A unique \emph{canonical} solution of a prenormalization equation]\label{PropSolutionPrenorm}
	Let $f\in \mathcal{L}_{k}^{H}$, $k\in \mathbb{N}_{\geq 1}$, be such that $f=z^{\alpha }+z^{\alpha }R_{\alpha }+\mathrm{h.o.b.}(z)$, for $\alpha >1$, $R_{\alpha }\in \mathcal{B}_{\geq 1}^{+}\subseteq \mathcal{L}_{k}$. There exists a unique parabolic logarithmic transseries $\varphi _{1}\in \mathcal{L}_{k}^{0}$ of the canonical form $\varphi _{1}=\mathrm{id}+zS$, $S\in \mathcal{B}_{\geq 1}^{+}\subseteq \mathcal{L}_{k}$, such that $\varphi _{1} \circ f\circ \varphi _{1}^{-1}=z^{\alpha }+ \mathrm{h.o.b.}(z)$.
\end{prop}
\begin{proof}
	Directly from Lemma~\ref{LemmaPrenormTS}, Lemma~\ref{LemmaConditionsPrenorm} and Proposition~\ref{PropFixedPoint}.
\end{proof}

\begin{rem}[The algorithm for the canonical prenormalization]
	Let $\varphi _{1}=\mathrm{id}+zS$, $S\in \mathcal{B}_{\geq 1}^{+}\subseteq \mathcal{L}_{k}$, be the prenormalization obtained in Proposition~\ref{PropSolutionPrenorm}. By the fixed point theorem stated in Proposition~\ref{PropFixedPoint}, it follows that $S$ is obtained as the limit of the \emph{Picard sequence}
	\begin{align*}
		\varphi _{1} &= \lim _{n\to \infty } \left( \left( \mathcal{T}_{f}^{-1}\circ \mathcal{S}_{f}\right) ^{\circ n}(H)\right) _{n} ,
	\end{align*}
	for any initial condition $H\in \mathcal{B}_{\geq 1}^{+}\subseteq \mathcal{L}_{k}$, with respect to the metric $d_{1}$ on $\mathcal{B}_{\geq 1}^{+}\subseteq \mathcal{L}_{k}$.
\end{rem}

\medskip

\subsubsection{Proof of statement 1 of Theorem~A}\label{SubsubsectionProofStatement1}

\begin{proof}
	Let $f\in \mathcal{L}_{k}^{H}$, $k\in \mathbb{N}_{\geq 1}$, be such that $f=z^{\alpha }+z^{\alpha }R_{\alpha }+\mathrm{h.o.b.}(z)$, $\alpha >1$, $R_{\alpha }\in \mathcal{B}_{\geq 1}^{+}\subseteq \mathcal{L}_{k}$. By Proposition~\ref{PropSolutionPrenorm}, there exists a unique $\varphi _{1}=\mathrm{id}+zS$, $S \in \mathcal{B}_{\geq 1}^{+}\subseteq \mathcal{L}_{k}$, such that $\varphi _{1} \circ f\circ \varphi _{1}^{-1}=z^{\alpha }+\mathrm{h.o.b.}(z)$. Note that $\varphi _{1} =\mathrm{id}$ and $\varphi _{1}\circ f \circ \varphi _{1}^{-1}=f$, if $z^{\alpha }R_{\alpha }=0$. By Proposition~\ref{PropContractibility}, the operator $\mathcal{P}_{\varphi _{1}\circ f\circ \varphi _{1}^{-1}}$ is a contraction on the space $\mathrm{id}+\mathcal{L}_{k}^{\beta }$, where $\beta := \mathrm{ord}_{z}(\varphi _{1}\circ f\circ \varphi _{1}^{-1}-z^{\alpha })-\alpha +1$. By the Banach Fixed Point Theorem and Proposition~\ref{PropNormFixed}, there exists a unique $\varphi _{2} \in \mathrm{id}+\mathcal{L}_{k}^{\beta }$, such that:
	\begin{align*}
		\varphi _{2}\circ (\varphi _{1}\circ f\circ \varphi _{1}^{-1})\circ \varphi _{2}^{-1} &= z^{\alpha } .
	\end{align*}
	Now, for $\varphi := \varphi _{2}\circ \varphi _{1}$, we get $\varphi  \in \mathcal{L}_{k}^{0}$ and $\varphi  \circ f \circ \varphi ^{-1} = z^{\alpha }$.
	
	Now we prove the uniqueness of $\varphi  \in \mathfrak L^{0}$. Suppose that $\psi  \in \mathfrak L^{0}$ such that $\psi  \circ f\circ \psi ^{-1}=z^{\alpha }$. Then there exists $k_{1}\in \mathbb{N}$, such that $\psi \in \mathcal{L}_{k_{1}}^{0}$. Put $k_{0}:=\max \left\lbrace k,k_{1}\right\rbrace $. Since $\mathcal{L}_{k},\mathcal{L}_{k_{1}}\subseteq \mathcal{L}_{k_{0}}$, we get $f,\varphi ,\psi \in \mathcal{L}_{k_{0}}$. Now, we take $\mathcal{L}_{k_{0}}$ as an \emph{ambient space}. Let $T \in \mathcal{B}_{\geq 1}^{+}\subseteq \mathcal{L}_{k_{0}}$ be such that $\psi = \mathrm{id} + zT+\mathrm{h.o.b.}(z)$. Put $\psi _{1}:=\mathrm{id}+zT$ and $\psi _{2}:=\mathrm{id}+(\psi -\psi _{1})\circ \psi _{1}^{-1}$. Note that $\mathrm{ord}_{z}(\psi _{2}-\mathrm{id}) >1$ and $\psi _{2}\circ \psi _{1} = \psi _{1}+\psi -\psi _{1}= \psi $. Since
	\begin{align}
		z^{\alpha }= \psi \circ f\circ \psi ^{-1} = \psi _{2}\circ \psi _{1} \circ f\circ \psi _{1}^{-1}\circ \psi _{2}^{-1} \nonumber 
	\end{align}
	and $\mathrm{ord}_{z}(\psi _{2} - \mathrm{id}) >1$, it follows that:
	\begin{align}
		\psi _{1}\circ f\circ \psi _{1}^{-1} = \psi _{2}^{-1} \circ z^{\alpha } \circ \psi _{2} = z^{\alpha }+ \mathrm{h.o.b.}(z). \nonumber
	\end{align}
	By Proposition~\ref{PropSolutionPrenorm} we get $\psi _{1} = \varphi _{1}$. Put $\beta := \mathrm{ord}_{z}(\varphi _{1}\circ f\circ \varphi _{1}^{-1}-z^{\alpha }) - \alpha +1$. By Proposition~\ref{PropNormOrder}, it follows that $\mathrm{ord}_{z}(\psi _{2}-\mathrm{id}) \geq \beta >1$, so $\psi _{2} \in \mathrm{id}+\mathcal{L}_{k_{0}}^{\beta }$ and $\mathcal{P}_{\varphi _{1}\circ f\circ \varphi _{1}^{-1}}(\psi _{2})=\psi _{2}$. Now, Proposition~\ref{PropContractibility} implies that $\mathcal{P}_{\varphi _{1}\circ f\circ \varphi _{1}^{-1}}$ is a contraction. Therefore, $\psi _{2} = \varphi _{2}$, so $\psi = \varphi $.
\end{proof}

\medskip

\subsection[Proof of statement 2 of Theorem~A]{Proof of statement 2 of Theorem~A (The convergence of the B\"ottcher sequence)}\label{SubsectionProofOf2}

Let $f\in \mathcal{L}_{k}^{H}$, $k\in \mathbb{N}$, $f=z^{\alpha }+\mathrm{h.o.t.}$, $\alpha >1$, and $\mathrm{ord}_{z}\, (f-z^{\alpha })>\alpha $. By Proposition~\ref{PropContractibility} the B\"ottcher operator $\mathcal{P}_{f}$ is, in that case, a contraction on the space $(\mathrm{id}+\mathcal{L}_{k}^{\beta },d_{z})$, for $\beta :=\mathrm{ord}_{z} \, (f-z^{\alpha })- \alpha +1$. This implies that the B\"ottcher sequence $\left( \mathcal{P}_{f}^{\circ n}(h)\right) _{n}$ converges towards the normalization $\varphi $ on the space $(\mathrm{id}+\mathcal{L}_{k}^{\beta },d_{z})$, for each $h\in \mathrm{id}+\mathcal{L}_{k}^{\beta }$.

However, this is not true if $\mathrm{ord}_{z}\, (f-z^{\alpha })=\alpha $, i.e., $\beta =1$. Therefore, if $\beta =1$, we first prove in Lemma~\ref{PropConvergenceFirst} the convergence of the B\"ottcher sequence $\left( \mathcal{R}_{f}^{\circ n}(z+zH)\right) _{n}$ on the space $\mathrm{id}+z\mathcal{B}_{\geq 1}^{+}$ (in the weak topology) towards the prenormalization $\varphi _{1}$ given in the canonical form, for each $z+zH\in \mathrm{id}+z\mathcal{B}_{\geq 1}^{+}$. And then, at the and of the subsection, we apply the discussion above. \\

We first prove some technical general result for the convergence of a sequence of compositions in the weak topology.

\begin{prop}[Convergence of a sequence of compositions]\label{PropConvComp}
	Let $(\psi _{n})$, $(\varphi _{n})$ be sequences in $\mathcal{L}_{k}^{0}$, $k\in \mathbb{N}$, which converge in the weak topology to $\psi $, $\varphi  \in \mathcal{L}_{k}^{0}$, respectively, as $n\to +\infty $. Suppose that there exists a well-ordered subset $W\subseteq \mathbb{R}_{\geq 0}\times \mathbb{Z} ^{k}$, such that $\min W>\mathrm{0}_{k+1}$ and $\mathrm{Supp} \, (\psi _{n}), \mathrm{Supp} \, (\varphi _{n}) \subseteq W$, $n\in \mathbb{N}$. Then the sequence $(\psi _{n}\circ \varphi _{n})$ converges to $\psi  \circ \varphi $ in the weak topology on $\mathcal{L}_{k}^{0}$. \\
	In particular, the sequences, $(\psi _{n} \circ \varphi )$ and $(\psi  \circ \varphi _{n})$ converge to $\psi  \circ \varphi $.
\end{prop}
\begin{proof}
	First we give a general estimate of the support of composition of two logarithmic transseries whose supports belong to a well-ordered set $W$. Let $S\subseteq \mathbb{R}_{\geq 0}\times \mathbb{Z} ^{k}$ be the semigroup generated with $(0,1,0,\ldots ,0)_{k+1} \ldots , (0,0,\ldots ,0,1)_{k+1}$ and let $W\subseteq \mathbb{R}_{\geq 0}\times \mathbb{Z}^{k}$ be a well-ordered subset, such that $\min W> \mathbf{0}_{k+1}$. Let $g,h \in \mathcal{L}_{k}^{0}$ be such that $\mathrm{Supp} \, (g),\mathrm{Supp} \, (h) \subseteq W$ and let $g_{1}:=g-\mathrm{id}$, $h_{1}:= h-\mathrm{id}$. By the definition of the composition, it follows that:
	\begin{align*}
		g \circ h & = h + g_{1} + \sum _{i\geq 1}\frac{g_{1}^{(i)}}{i!}h_{1}^{i} . 
	\end{align*}
	Hence, every element of $\mathrm{Supp} \, (g\circ h)$ is an element of $\mathrm{Supp} \, (h) \cup \mathrm{Supp} \, (g_{1})$ or can be written in the form:
	\begin{align*}
		& (\beta , \mathbf{n}) + (0,\mathbf{m}) + (\beta _{1}-1,\mathbf{n}_{1})+\cdots + (\beta _{i}-1,\mathbf{n}_{i}) ,
	\end{align*}
	where $(\beta , \mathbf{n}) \in \mathrm{Supp} \, (g_{1})$, $(0,\mathbf{m}) \in S$, $(\beta _{1},\mathbf{n}_{1}),\ldots ,(\beta _{i},\mathbf{n}_{i}) \in \mathrm{Supp} \, (h_{1})$ and $i\in \mathbb{N}_{\geq 1}$.
	
	Let $W_{1}$ be the semigroup generated by $(\beta -1, \mathbf{n})$, for $(\beta , \mathbf{n}) \in W$ such that $(\beta , \mathbf{n})>(1, \mathbf{0}_{k}) $, and let $\overline{W}$ be the semigroup generated by $W\cup W_{1}\cup S$. By the Neumann Lemma (see \cite{Neumann49}, \cite{Dries}, \cite{mrrz1}), it follows that $W_{1}$ and $\overline{W}$ are well-ordered sets. By the above analysis, it follows that $\mathrm{Supp} \, (g\circ h) \subseteq \overline{W}$.
	
	Now, for every $f\in \mathcal{L}_{k}$ such that $\mathrm{Supp} \, (f)\subseteq \overline{W}$ and every $w\in \overline{W}$ we define $\left[ f\right] _{w}$ to be the coefficient of $f$ if $w\in \mathrm{Supp}\, (f)$, and $0$ otherwise. By the Neumann Lemma, for every $w\in \overline{W}$ there exist finitely many tuples with elements in  $W\cup W_{1}\cup S$ whose sum eguals to $w$. Hence, there exist $m,r\in \mathbb{N}$ and $w_{1},\ldots ,w_{m},t_{1},\ldots ,t_{r}\in \overline{W}$, such that $\left[ g\circ h\right] _{w}$ depends only on $[g]_{w_{1}},\ldots , [g]_{w_{m}}, [h]_{t_{1}}, \ldots , [h]_{t_{r}}$. That is, there exists a polynomial $P_{w}$ in $m+r$ variables, with real coefficients, that does not depend on $g$ nor $h$, such that
	\begin{align*}
		[g\circ  h]_{w} & =P_{w}\big( [g]_{w_{1}},\ldots , [g]_{w_{m}}, [h]_{t_{1}}, \ldots , [h]_{t_{r}} \big) .
	\end{align*}
	
	Now we prove the proposition. Recall that $\mathrm{Supp} \, (\psi _{n}), \, \mathrm{Supp} \, (\varphi _{n}) \subseteq W$, for every $n\in \mathbb{N}$. Since $(\psi _{n})\to \psi $ and $(\varphi _{n}) \to \varphi $ in the weak topology, it follows that $\mathrm{Supp} \, (\psi ), \, \mathrm{Supp} \, (\varphi ) \subseteq W$ (see Subsection \ref{SubsectionMetricTopology}). From the consideration above, we conclude that $\mathrm{Supp} \, (\psi \circ \varphi ) \subseteq \overline{W}$ and $\mathrm{Supp} \, (\psi _{n}\circ \varphi _{n}) \subseteq \overline{W}$, for every $n\in \mathbb{N}$. Since the polynomial $P_{w}$, $w\in \overline{W}$, defined above, depends only on $w_{1},\ldots ,w_{m},t_{1},\ldots ,t_{r}\in \overline{W}$ above, we get:
	\begin{align*}
		[\psi _{n}\circ  \varphi _{n}]_{w} &=P_{w}\big( [\psi _{n}]_{w_{1}},\ldots , [\psi _{n}]_{w_{m}}, [\varphi _{n}]_{t_{1}}, \ldots , [\varphi _{n}]_{t_{r}} \big) 
	\end{align*}
	and
	\begin{align*}
		[\psi \circ  \varphi ]_{w} &=P_{w}\big( [\psi ]_{w_{1}},\ldots , [\psi ]_{w_{m}}, [\varphi ]_{t_{1}}, \ldots , [\varphi ]_{t_{r}} \big) .
	\end{align*}
	Since $(\varphi _{n}) \to \varphi $ and $(\psi _{n})\to \psi $ in the weak topology, it follows that $(\left[ \varphi _{n}\right] _{w}) \to \left[ \varphi \right] _{w}$ and $(\left[ \psi _{n}\right] _{w})\to \left[ \psi _{n}\right] _{w}$ in the Euclidean topology. Hence, by the continuity of the polynomial map $P_{w}$, it follows that:
	\begin{align*}
		\lim _{n\to \infty } [\psi _{n}\circ  \varphi _{n}]_{w} & = \lim _{n\to \infty } P_{w}([\psi _{n}]_{w_{1}},\ldots , [\psi _{n}]_{w_{m}}, [\varphi _{n}]_{t_{1}}, \ldots , [\varphi _{n}]_{t_{r}}) \\
		& =P_{w}([\psi ]_{w_{1}},\ldots , [\psi ]_{w_{m}}, [\varphi ]_{t_{1}}, \ldots , [\varphi ]_{t_{r}}) \\
		&= [\psi \circ  \varphi ]_{w} , \quad w\in \overline{W} .
	\end{align*}
	This proves the convergence of the sequence $(\psi _{n}\circ \varphi _{n})_{n}$ to $\psi  \circ \varphi $ in the weak topology on $\mathcal{L}_{k}^{0}$.
\end{proof}

\begin{lem}[Convergence of the B\"ottcher sequence on $\mathrm{id}+z\mathcal{B}_{\geq 1}^{+}$]\label{PropConvergenceFirst}
	Let $f\in \mathcal{L}_{k}^{H}$, $k\in \mathbb{N}$, be such that $f=z^{\alpha }+z^{\alpha }R_{\alpha }+\mathrm{h.o.b.}(z)$, $\alpha \in \mathbb{R}_{>1}$, $R_{\alpha }\in \mathcal{B}_{\geq 1}^{+}\subseteq \mathcal{L}_{k}$, and let $h= \mathrm{id}+zH$, $H \in \mathcal{B}_{\geq 1}^{+}\subseteq \mathcal{L}_{k}$. Then the sequence $(\mathcal{R}_{f}^{\circ n}(\mathrm{id}+zH))_{n}$ converges to the prenormalization $\varphi _{1}$ of $f$ given in the canonical form $\varphi _{1}=\mathrm{id}+zS$, $S\in \mathcal{B}_{\geq 1}^{+}\subseteq \mathcal{L}_{k}$, in the weak topology on the space $\mathrm{id}+z\mathcal{B}_{\geq 1}^{+}\subseteq \mathcal{L}_{k}^{0}$, where $\mathcal{R}_{f}$ is the B\"ottcher operator on $\mathrm{id}+z\mathcal{B}_{\geq 1}^{+}$ defined in \eqref{OperatorR}.
\end{lem}
\begin{proof}
	Since $\varphi _{1}\circ (z^{\alpha }+z^{\alpha }R_{\alpha })\circ \varphi _{1}^{-1}=z^{\alpha }$, note that
	\begin{align*}
		\mathcal{R}_{f}^{\circ n}(h) \circ \varphi _{1}^{-1} & = z^{\frac{1}{\alpha ^{n}}} \circ (h\circ \varphi _{1}^{-1})\circ \big( \varphi _{1}\circ (z^{\alpha }+z^{\alpha }R_{\alpha })\circ \varphi _{1}^{-1} \big) ^{\circ n} \\
		& = z^{\frac{1}{\alpha ^{n}}} \circ (h\circ \varphi _{1}^{-1})\circ z^{\alpha ^{n}} , 
	\end{align*}
	for $n\in \mathbb{N}$. By Proposition~\ref{PropConvComp}, it follows that $(\mathcal{R}_{f}^{\circ n}(\mathrm{id}+zH))_{n}$ converges to $\varphi _{1}$ if and only if $(z^{\frac{1}{\alpha ^{n}}} \circ (h\circ \varphi _{1}^{-1})\circ z^{\alpha ^{n}})_{n}$ converges to $\mathrm{id}$ in the weak topology on the space $\mathrm{id}+z\mathcal{B}_{\geq 1}^{+}$. Now suppose that $h\circ \varphi _{1}^{-1} = \mathrm{id}+zK$, $K\in \mathcal{B}_{\geq 1}^{+}\subseteq \mathcal{L}_{k}$. We prove that $(z^{\frac{1}{\alpha ^{n}}} \circ (\mathrm{id}+zK)\circ z^{\alpha ^{n}})_{n}$ converges to $\mathrm{id}$ in the weak topology on $\mathrm{id}+z\mathcal{B}_{\geq 1}^{+}\subseteq \mathcal{L}_{k}^{0}$. Note that:
	\begin{align}
		z^{\frac{1}{\alpha ^{n}}} \circ (\mathrm{id}+zK) \circ z^{\alpha ^{n}} &= (z^{\alpha ^{n}}+z^{\alpha ^{n}}K(z^{\alpha ^{n}}))^{\frac{1}{\alpha ^{n}}} \nonumber \\
		&= \mathrm{id}+z\sum _{i\geq 1}{\frac{1}{\alpha ^{n}} \choose i}(K(z^{\alpha ^{n}}))^{i} . \label{Identity1}
	\end{align}
	Note that:
	\begin{align}
		{\frac{1}{\alpha ^{n}} \choose i} &= \frac{\frac{1}{\alpha ^{n}}\cdot (\frac{1}{\alpha ^{n}}-1)\cdots (\frac{1}{\alpha ^{n}}-(i-1))}{i!} \nonumber \\
		&= \frac{1}{i\alpha ^{n}}\cdot \frac{\frac{1}{\alpha ^{n}}-1}{1}\cdots \frac{\frac{1}{\alpha ^{n}}-(i-1)}{i-1} , \label{Identity2}
	\end{align}
	for $i\in \mathbb{N}_{\geq 1}$ and $n\in \mathbb{N}$. Since, $0< \frac{j-\frac{1}{\alpha ^{n}}}{j} < 1$, $j\in \left\lbrace 1, \ldots , i-1\right\rbrace $, it follows that:
	\begin{align}
		\left| {\frac{1}{\alpha ^{n}} \choose i}\right| \leq \frac{1}{\alpha ^{n}} , \label{Identity3}
	\end{align}
	for $i\in \mathbb{N}_{\geq 1}$ and $n\in \mathbb{N}$. By Lemma~$A.3.2$ in \cite{Peran21Thesis}, it follows that every coefficient of $\boldsymbol{\ell }_{m}(z^{\alpha ^{n}})$, $m\in \left\lbrace 1, \ldots ,k\right\rbrace $, is a polynomial in the variables $\frac{1}{\alpha ^{n}}$ and $n\log \alpha $, for $n\in \mathbb{N}$. Let $\mathbf{n} \in \mathbb{Z} ^{k}$, such that $\mathbf{n} > \mathbf{0}_{k}$. Thus, for $n\in \mathbb{N}$, it follows that $\left[ (K(z^{\alpha ^{n}}))^{i}\right] _{(1, \mathbf{n})}$, $i\in \mathbb{N}_{\geq 1} $, is the value $P_{i}(\frac{1}{\alpha ^{n}},n\log \alpha )$ of the polynomial $P_{i}$ in the variables $\frac{1}{\alpha ^{n}}$ and $n\log \alpha $, whose coefficients do not depend on $n$. By the Neumann Lemma and \eqref{Identity1}, there exist $m\in \mathbb{N}$ and $i_{1},\ldots ,i_{m} \in \mathbb{N}_{\geq 1} $, such that:
	\begin{align}
		\left[ ((\mathrm{id}+zK)(z^{\alpha ^{n}}))^{\frac{1}{\alpha ^{n}}}\right] _{1,\mathbf{n}} &= {\frac{1}{\alpha ^{n}} \choose i_{1}}P_{i_{1}} \Big( \frac{1}{\alpha ^{n}}, \, n\log \alpha \Big) + \cdots + {\frac{1}{\alpha ^{n}} \choose i_{m}}P_{i_{m}}\Big( \frac{1}{\alpha ^{n}}, \, n\log \alpha \Big) , \label{Identity6}
	\end{align}
	for every $n\in \mathbb{N}$. From \eqref{Identity3} and \eqref{Identity6}, it follows that:
	\begin{align*}
		& \left| \left[ ((\mathrm{id}+zK)(z^{\alpha ^{n}}))^{\frac{1}{\alpha ^{n}}}\right] _{1,\mathbf{n}}\right| \\
		& \leq \left| {\frac{1}{\alpha ^{n}} \choose i_{1}}\right| \left| P_{i_{1}}\Big( \frac{1}{\alpha ^{n}}, \, n\log \alpha \Big) \right| + \cdots + \left| {\frac{1}{\alpha ^{n}} \choose i_{m}}\right| \left| P_{i_{m}}\Big( \frac{1}{\alpha ^{n}}, \, n\log \alpha \Big) \right| \\
		& \leq \frac{1}{\alpha ^{n}}\Bigg( \sum _{j=1}^{m}\left| P_{i_{j}}\Big( \frac{1}{\alpha ^{n}}, \, n\log \alpha \Big) \right| \Bigg) ,
	\end{align*}
	for $n\in \mathbb{N}$. Now, passing to the limit as $n$ goes to infinity, we get:
	\begin{align}
		\left| \lim _{n\to \infty } \left[ ((\mathrm{id}+zK)(z^{\alpha ^{n}}))^{\frac{1}{\alpha ^{n}}}\right] _{1,\mathbf{n}} \right| \leq \lim _{n\to \infty }\frac{\sum _{j=1}^{m}\left| P_{i_{j}}(\frac{1}{\alpha ^{n}}, \, n\log \alpha )\right| }{\alpha ^{n}} =0. \label{Identity4}
	\end{align}
	From \eqref{Identity4} and $\left[ ((\mathrm{id}+zK)(z^{\alpha ^{n}}))^{\frac{1}{\alpha ^{n}}} \right] _{1, \mathbf{0}_{k}}=1$, for each $n\in \mathbb{N}$, it follows that the sequence $(z^{\frac{1}{\alpha ^{n}}} \circ (\mathrm{id}+zK)\circ z^{\alpha ^{n}})_{n}$ converges to $\mathrm{id}$ in the weak topology on the space $\mathrm{id}+z\mathcal{B}_{\geq 1}^{+}\subseteq \mathcal{L}_{k}^{0}$.
\end{proof}

\medskip

\begin{proof}[Proof of statement 2 of Theorem~A]
	Let $\varphi =\varphi _{2}\circ \varphi _{1}$ be a unique solution of the normalization equation \eqref{EqNorm}, where $\varphi _{1}$ is the canonical solution of a prenormalization equation \eqref{EqPreNorm}, and $\varphi _{2}$ is the fixed point of the B\"ottcher operator $\mathcal{P}_{\varphi _{1}\circ f\circ \varphi _{1}^{-1}}$ defined as in \eqref{DefinitionPf}. Let $h \in \mathfrak L^{0}$ be arbitrary. Suppose that $m\in \mathbb{N}$ is minimal such that $f,h\in \mathcal{L}_{m}$. For $n\in \mathbb{N}$, it follows that:
	\begin{align}
		z^{\frac{1}{\alpha ^{n}}}\circ h\circ f^{\circ n} &= z^{\frac{1}{\alpha ^{n}}}\circ (h\circ \varphi _{1}^{-1}) \circ (\varphi _{1}\circ f\circ \varphi _{1}^{-1})^{\circ n} \circ \varphi _{1} . \label{Equati1}
	\end{align}
	Put $h\circ \varphi _{1}^{-1}=\mathrm{id}+zK+g$, where $g\in \mathcal{L}_{m}$ such that $\mathrm{ord}_{z} \, (g)>1$, and $K\in \mathcal{B}_{\geq 1}^{+}\subseteq \mathcal{L}_{m}$. Put $h_{1}:=\mathrm{id}+zK$ and $h_{2}:=h_{1}^{-1}\circ h\circ \varphi _{1}^{-1}$. Note that $h\circ \varphi _{1}^{-1}=h_{1}\circ h_{2}$, $\mathrm{ord}_{z} \, (h_{2}-\mathrm{id}) > 1$, and $\mathrm{ord}_{z} \, (h_{1}-\mathrm{id}) =1$, if $h_{1} \neq \mathrm{id}$. By \eqref{Equati1}, it follows that:
	\begin{align}
		z^{\frac{1}{\alpha ^{n}}}\circ h\circ f^{\circ n} &= \big( z^{\frac{1}{\alpha ^{n}}}\circ h_{1} \circ z^{\alpha ^{n}} \big) \circ \big( z^{\frac{1}{\alpha ^{n}}}\circ h_{2} \circ (\varphi _{1}\circ f\circ \varphi _{1}^{-1})^{\circ n} \big) \circ \varphi _{1} \nonumber \\
		& = \big( z^{\frac{1}{\alpha ^{n}}}\circ h_{1} \circ z^{\alpha ^{n}} \big) \circ \mathcal{P}_{\varphi _{1}\circ f\circ \varphi _{1}^{-1}}^{\circ n}(h_{2}) \circ \varphi _{1} . \label{Equati22}
	\end{align}
	In Subsection \ref{SubsectionProofOfStat3} we prove that the supports $\mathrm{Supp}\big( z^{\frac{1}{\alpha ^{n}}}\circ f^{\circ n} \big) $, $n\in \mathbb{N}$, are contained in a common well-ordered subset of $\mathbb{R}_{\geq 0}\times \mathbb{Z}^{m}$ whose minimum is strictly bigger than $\mathbf{0}_{m+1}$. Following that idea, it is easy to prove that the supports of $z^{\frac{1}{\alpha ^{n}}}\circ h_{1}\circ z^{\alpha ^{n}}$ and $\mathcal{P}^{\circ n}_{\varphi _{1}\circ f\circ \varphi _{1}^{-1}}(h_{2})$, for $n\in \mathbb{N}$, are subsets of a common well-ordered subset of $\mathbb{R}_{\geq 0}\times \mathbb{Z}^{m}$ whose minimum is strictly bigger than $\mathbf{0}_{m+1}$. By the proof of Lemma~\ref{PropConvergenceFirst}, and by Proposition~\ref{PropContractibility} and Proposition~\ref{PropConvComp}, it follows that the B\"ottcher sequence $(z^{\frac{1}{\alpha ^{n}}}\circ h\circ f^{\circ n})_{n}$ converges to $\varphi =\varphi _{2}\circ \varphi _{1}$ in the weak topology on $\mathfrak L^{0}$. \\
	
	Now, suppose that $\mathrm{Lb}_{z} \, (h) = \mathrm{Lb}_{z} \, (\varphi )$. Let $\varphi :=\varphi _{2}\circ \varphi _{1}$, where $\varphi _{1}$ is the canonical prenormalization, if $\mathrm{ord}_{z} \, (f-z^{\alpha })=\alpha $, or $\varphi _{1}:=\mathrm{id}$, if $\mathrm{ord}_{z}\, (f-z^{\alpha }) > \alpha $, and $\varphi _{2}$ is the normalization of $\varphi _{1}\circ f\circ \varphi _{1}^{-1}$. Now, since $\mathrm{ord}_{z} \, (\varphi _{2}) > 1$, it follows that $\mathrm{Lb}_{z} \, (\varphi ) = \varphi _{1}$. Thus, $\mathrm{Lb}_{z} \, (h) = \varphi _{1}$. Therefore, $\mathrm{ord}_{z} \, (h\circ \varphi _{1}^{-1} -\mathrm{id}) > 1$. By \eqref{Equati1}, it follows that
	\begin{align*}
		z^{\frac{1}{\alpha ^{n}}}\circ h\circ f^{\circ n} &= \mathcal{P}_{\varphi _{1}\circ f\circ \varphi _{1}^{-1}} ^{\circ n} (h\circ \varphi _{1}^{-1}) \circ \varphi _{1} .
	\end{align*}
	The B\"ottcher operator $\mathcal{P}_{\varphi _{1}\circ f\circ \varphi _{1}^{-1}}$ is a contraction on the space $(\mathcal{L}_{m}^{\beta },d_{z})$, for each $\beta >1$. Since $\mathrm{ord}_{z} \, (h\circ \varphi _{1}^{-1} -\mathrm{id}) > 1$, by the Banach Fixed Point Theorem, the sequence of the Picard iterations $(\mathcal{P}^{\circ n}_{\varphi _{1}\circ f\circ \varphi _{1}^{-1}}(h\circ \varphi _{1}^{-1}))_{n}$ converges to $\varphi _{2}$ in the power-metric topology on the space $\mathrm{id}+\mathcal{L}_{m}^{\beta }$. By statement 3 of Lemma~$2.2.17$ in \cite{Peran21Thesis}, it follows that $(z^{\frac{1}{\alpha ^{n}}}\circ h\circ f^{\circ n})_{n}$ converges to $\varphi =\varphi _{2}\circ \varphi _{1}$ in the power-metric topology. \\
	Conversely, suppose that $(z^{\frac{1}{\alpha ^{n}}}\circ h\circ f^{\circ n})_{n}$ converges to the normalization $\varphi :=\varphi _{2}\circ \varphi _{1}$ in the power-metric topology and that $\mathrm{Lb}_{z} \, (\varphi ) \neq \mathrm{Lb}_{z} \, (h)$, where $\varphi _{1}$ is the canonical prenormalization and $\varphi _{2}$ is the normalization of prenormalized transseries $\varphi _{1}\circ f\circ \varphi _{1}^{-1}$. By statement 2 of Lemma~$2.2.17$ in \cite{Peran21Thesis}, it follows that $\mathrm{Lb}_{z} \, (h\circ \varphi _{1}^{-1}) \neq \mathrm{id}$. Let $\mathrm{Lb}_{z} \, (h\circ \varphi _{1}^{-1})=\mathrm{id}+zK$, for $K\in \mathcal{B}_{\geq 1}^{+}\subseteq \mathcal{L}_{m}$, $K\neq 0$. Put $h_{1}:=\mathrm{id}+zK$ and $h_{2}:=h_{1}^{-1}\circ h\circ \varphi _{1}^{-1}$. Consequently, it follows that $\mathrm{ord}_{z} \, (h_{2}-\mathrm{id}) > 1$ and $h\circ \varphi _{1}^{-1} = h_{1}\circ h_{2}$. Now, by \eqref{Equati22} and statement 3 of Lemma~$2.2.17$ in \cite{Peran21Thesis}, it follows that $\Big( (z^{\frac{1}{\alpha ^{n}}}\circ h_{1}\circ z^{\alpha ^{n}})\circ \mathcal{P}_{\varphi _{1}\circ f\circ \varphi _{1}^{-1}}^{\circ n}(h_{2})\Big) _{n}$ converges to $\varphi _{2}=\varphi \circ \varphi _{1}^{-1}$ in the power-metric topology. Note that
	\begin{align*}
		(z^{\frac{1}{\alpha ^{n}}}\circ h_{1}\circ z^{\alpha ^{n}})\circ \mathcal{P}_{\varphi _{1}\circ f\circ \varphi _{1}^{-1}}^{\circ n}(h_{2}) & = z^{\frac{1}{\alpha ^{n}}}\circ h_{1}\circ z^{\alpha ^{n}} + \mathrm{h.o.b.}(z) ,
	\end{align*}
	for every $n\in \mathbb{N}$. Since $\mathrm{ord}_{z} \, (\varphi _{2}) > 1$, by the Taylor Theorem (see \cite[Proposition 3.3]{PRRSFormal21}) and the convergence in the power-metric topology, it follows that there exists $n_{0}\in \mathbb{N}$ such that $z^{\frac{1}{\alpha ^{n}}} \circ h_{1}\circ z^{\alpha ^{n}} = \mathrm{id}$, for every $n\geq n_{0}$. By \eqref{Identity1} and Lemma~$A.3.2$ in \cite{Peran21Thesis}, it follows that
	\begin{align*}
		\mathrm{Lt} \, \big( z^{\frac{1}{\alpha ^{n}}} \circ h_{1}\circ z^{\alpha ^{n}} - \mathrm{id} \big) &= \left( \frac{1}{\alpha ^{n}}\right) ^{1+\mathrm{ord}_{\boldsymbol{\ell }_{1}} \, (K)}\mathrm{Lt} \, (zK) ,
	\end{align*}
	for each $n\in \mathbb{N}$. Since $z^{\frac{1}{\alpha ^{n}}} \circ h_{1}\circ z^{\alpha ^{n}}=\mathrm{id}$, for every $n\geq n_{0}$, it follows that $K=0$, i.e., $\mathrm{Lb}_{z} \, (\varphi ) = \mathrm{Lb}_{z} \, (h)$.
\end{proof}

\medskip

\subsection[Proof of statement 3 of Theorem~A]{Proof of statement 3 of Theorem~A (Control of the support of the normalization)}\label{SubsectionProofOfStat3}

\begin{lem}[Control of the support of a composition]\label{PropContolSuportComp}
	Let $f\in \mathfrak L$ and $g\in \mathfrak L^{H}$. Let $r\in \mathbb{N}$ be minimal such that $f,g\in \mathcal{L}_{r}$. Let $g:=\lambda z^{\alpha }+ \mathrm{h.o.t.}$, for $\alpha , \lambda >0$, and $g_{1}:=g-z^{\alpha }$. Then the support $\mathrm{Supp} \, (f\circ g)$ is contained in the sub-semigroup of $\mathbb{R}_{\geq 0}\times \mathbb{Z}^{r}$ generated by the elements:
	\begin{align*}
		& (0,1,0,\ldots ,0)_{r+1},\ldots ,(0,0,\ldots ,0,1)_{r+1}, \\
		& (\alpha \delta , \mathbf{m}), \textrm{ for each } (\delta ,\mathbf{m})\in \mathrm{Supp} \, (f) , \\
		& (\beta -\alpha ,\mathbf{n}) , \textrm{ for each } (\beta ,\mathbf{n})\in \mathrm{Supp} \, (g_{1}) .
	\end{align*}
\end{lem}
\begin{proof}
	By the definition of a composition we have:
	\begin{align}
		f\circ g & = f(\lambda z^{\alpha }) + \sum _{i\geq 1}\frac{f^{(i)}(\lambda z^{\alpha })}{i!}g_{1}^{i} . \label{EqCompSuport}
	\end{align}
	By Lemma~$A.3.1$ and Lemma~$A.3.2$ in \cite{Peran21Thesis}, it follows that every element $(\rho , \mathbf{r})$ of the support of $f^{(i)}(\lambda z^{\alpha })$, $i\in \mathbb{N}$, can be obtained as:
	\begin{align}
		(\rho , \mathbf{r}) & = ((\delta -i)\alpha , \mathbf{v}) + (0,\mathbf{u}) , \label{EqCompSup}
	\end{align}
	where $(\delta , \mathbf{v}) \in \mathrm{Supp} \, (f)$ and $(0,\mathbf{u})$ is a linear combination (with coefficients in $\mathbb{N}$) of elements $(0,1,0,\ldots ,0)_{r+1},\ldots ,(0,0,\ldots ,0,1)_{r+1}$.
	
	From \eqref{EqCompSup} we conclude that every element $(\gamma , \mathbf{m})$ of the support of the sum on the right-hand side of \eqref{EqCompSuport} can be obtained in the following way:
	\begin{align*}
		(\gamma , \mathbf{m}) & = ((\delta -i)\alpha , \mathbf{v}) + (0,\mathbf{u}) + (\beta _{1},\mathbf{n}_{1}) + \cdots + (\beta _{i},\mathbf{n}_{i}) \\
		& = (\delta \alpha , \mathbf{v}) + (0,\mathbf{u}) + (\beta _{1}-\alpha ,\mathbf{n}_{1}) + \cdots + (\beta _{i}-\alpha ,\mathbf{n}_{i}) ,
	\end{align*}
	where $i\in \mathbb{N}_{\geq 1}$, $(\delta , \mathbf{v}) \in \mathrm{Supp} \, (f)$, $(0,\mathbf{u})$ is a linear combination (with coefficients in $\mathbb{N}$) of $(0,1,0,\ldots ,0)_{r+1}, \ldots , (0,0,\ldots ,0,1)_{r+1}$, and $(\beta _{j},\mathbf{n}_{j})\in \mathrm{Supp} \, (g_{1})$, $1\leq j\leq i$.
\end{proof}

\begin{proof}[Proof of statement 3 of Theorem~A]
	Let $f\in\mathcal{L}_{k}^{H}$ such that $f=z^{\alpha }+\mathrm{h.o.t.}$, for $\alpha \in \mathbb{R}_{>1}$. Let $W$ be the set of elements $(\alpha^{p} ,\mathbf{0}_{k})$, for $p\in \mathbb{N}$, $(0,1,0,\ldots ,0)_{k+1}$,$...$, $(0,0,\ldots ,0,1)_{k+1}$, and
	\begin{align}
		& (\alpha ^{m}(\gamma - \alpha ), \mathbf{n}) , \label{StrongHyperSupport1}
	\end{align}
	for $(\gamma , \mathbf{n}) \in \mathrm{Supp} \, (f-z^{\alpha })$, $m\in \mathbb{N}$. Since $\mathrm{Supp}\, (f-z^{\alpha })$ is a well-ordered set and $\alpha >1$, it is easy to see that $W$ is a well-ordered set. Now, the semigroup $\left\langle W\right\rangle $ generated by $W$ is well-ordered by the Neumann Lemma.
	
	Note that $\mathrm{Supp}\, (f) \subseteq \left\langle W\right\rangle $ and
	\begin{align}
		(\alpha \delta , \mathbf{n}) \in \left\langle W\right\rangle , \label{PropertyW}
	\end{align}
	for every $(\delta , \mathbf{n})\in \left\langle W\right\rangle $.
	
	Since, by statement 2 of Theorem~A (taking $h:=\mathrm{id}$), the sequence $(z^{\frac{1}{\alpha ^{n}}}\circ f^{\circ n})_{n}$ converges to $\varphi $ in the weak topology, it is sufficient to prove that $\mathrm{Supp}\, (z^{\frac{1}{\alpha ^{n}}}\circ f^{\circ n})\subseteq \left\langle W\right\rangle $, for each $n\in \mathbb{N}_{\geq 1} $.
	
	By \eqref{CompEq1} it follows that:
	\begin{align}
		z^{\frac{1}{\alpha ^{n}}}\circ f^{\circ n} &= \mathrm{id}+z\sum _{i\geq 1}{\frac{1}{\alpha ^{n}} \choose i}\Big(\frac{f^{\circ n}-z^{\alpha ^{n}}}{z^{\alpha ^{n}}}\Big)^{i}, \quad n\in \mathbb{N}_{\geq 1} .
	\end{align}
	Therefore, in order to prove that
	\begin{align}
		& \mathrm{Supp}\, (z^{\frac{1}{\alpha ^{n}}}\circ f^{\circ n}) \subseteq \left\langle W\right\rangle , \label{InductiveIdentity}
	\end{align}
	for $n\in \mathbb{N}_{\geq 1}$, it is sufficient to prove that $(\gamma - \alpha ^{n},\mathbf{n}) \in \left\langle W\right\rangle $, for every $(\gamma ,\mathbf{n}) \in \mathrm{Supp} \, (f^{\circ n}-z^{\alpha ^{n}})$, $n\in \mathbb{N}_{\geq 1}$. We prove it inductively. By the definition of $\left\langle W\right\rangle $ it is clear that the statement holds for $n=1$. Suppose that the statement holds for some $n\in \mathbb{N}_{\geq 1}$. By the proof of Lemma~\ref{PropContolSuportComp} every element of the support of $f^{\circ (n+1)}=f^{\circ n}\circ f$, $n\in \mathbb{N}_{\geq 1}$, is of the form:
	\begin{align*}
		& (\delta \alpha , \mathbf{v}) + (0,\mathbf{u}) ,
	\end{align*}
	or
	\begin{align*}
		& (\delta \alpha , \mathbf{v}) + (0,\mathbf{u}) + (\beta _{1}-\alpha ,\mathbf{n}_{1}) + \cdots + (\beta _{i}-\alpha ,\mathbf{n}_{i}) ,
	\end{align*}
	for $i\in \mathbb{N}_{\geq 1}$, $(\delta , \mathbf{v}) \in \mathrm{Supp} \, (f^{\circ n})$, $(0,\mathbf{u})$ is a linear combination (with coefficients in $\mathbb{N}_{\geq 1}$) of $(0,1,0,\ldots ,0)_{k+1}, \ldots , (0,0,\ldots ,0,1)_{k+1}$, and $(\beta _{j},\mathbf{n}_{j})\in \mathrm{Supp} \, (f-z^{\alpha })$, $1\leq j\leq i$. From this fact, the assumption of the induction and \eqref{PropertyW}, it follows that $(\gamma - \alpha ^{n+1},\mathbf{n}) \in \left\langle W\right\rangle $, for every $(\gamma ,\mathbf{n}) \in \mathrm{Supp} \, (f^{\circ (n+1)}-z^{\alpha ^{n+1}})$. This proves \eqref{InductiveIdentity}.
\end{proof}

\medskip

\section{Analytic normalization of analytic maps on admissible domains}\label{SectionNormalizationAnalytic}

In this section we consider analytic normalizations of analytic maps on complex domains with strongly hyperbolic logarithmic asymptotics. More precisely, for analytic map $f=\alpha \zeta +o(1)$, $\alpha \in \mathbb{R}_{>1}$, satisfying certain logarithmic asymptotics on some domain $D\subseteq \mathbb{C}$, we seek for an analytic parabolic solution $\varphi (\zeta )=\zeta +o(1)$ of the equation
\begin{align*}
	\varphi (f(\zeta )) & = \alpha \cdot \varphi (\zeta )
\end{align*}
on a suitable $f$-invariant subdomain of the domain $D$. In Definition~\ref{DefinitionAdmissibleDom} in Subsection~\ref{SubsectionAdmissibleDom} we introduce the so-called \emph{admissible domains}, with the property that they contain $f$-invariant subdomains.

The main result of this section is Theorem~B which can be viewed as an analogue of \cite[Theorem~A]{PRRSDulac21} for maps $f=\alpha \zeta +o(1)$, and as a generalization of the classical B\"ottcher Theorem (see e.g. \cite{CarlesonG93}, \cite{Milnor06}). As opposed to \cite[Theorem~A]{PRRSDulac21}, where only analytic maps with certain hyperbolic logarithmic asymptotics are analytically linearizable, in Theorem~B we prove that any analytic map with strongly hyperbolic logarithmic asymptotics can be analytically normalized on invariant domains.

\medskip

\subsection{Admissible domains}\label{SubsectionAdmissibleDom}

	By abuse, we use the same notation \emph{admissible domain} as in \cite{PRRSDulac21}. The difference is that in \cite{PRRSDulac21} the domains are admissible with the respect to hyperbolic maps. Compare with Subsection 3.1 in \cite{PRRSDulac21}. 
	
	Let $\alpha \in \mathbb{R}_{>1}$, $\varepsilon \in \mathbb{R}_{>0}$, $k\in\mathbb{N}$, and let 
	\begin{align}\label{eq:rhostrongly}
		M_{\varepsilon,k}(x) & :=\frac{1}{(\log^{\circ k} x)^{\varepsilon}}, \\
		\rho_{\alpha , \varepsilon,k}(x) & :=(\alpha -1)x - M_{\varepsilon,k}(x), \nonumber 
	\end{align}
	for $x\in\left(\exp^{\circ k}\left(0\right),+\infty\right) $. The map $M_{\varepsilon,k}$ is positive, strictly decreasing, tending to $0$, as $x\to+\infty$, and $\rho_{\alpha , \varepsilon,k}$ is a strictly increasing. Furthermore, $\rho_{\alpha , \varepsilon,k}(x)$ tends to $+\infty $, as $x\to+\infty$, and therefore $\rho_{\alpha , \varepsilon,k}(x)>0$ for sufficiently large $x$. \\
	
	In order to define the so-called admissible domains, we first define lower-upper pairs of maps. Let $t>\exp^{\circ k}(0)$ be such that $\rho_{\alpha ,\varepsilon,k}(x)>0$, for $x\in[t,+\infty)$. Let $h_{l},\ h_{u}:[t,+\infty)\to\mathbb R$ be two maps satisfying: 
	\begin{enumerate}[1., font=\textup, topsep=0.4cm, itemsep=0.4cm, leftmargin=0.6cm]
		\item $h_{l}(x)<h_{u}(x)$, $x\in[t,+\infty)$;
		\item $h_{l}$ is a decreasing map on $[t,+\infty)$ with the property:
		\begin{align*}
			h_{l}(x+\rho_{\alpha ,\varepsilon,k}(x))-h_{l}(x) & \leq (\alpha-1) \cdot h_{l}(x) -M_{\varepsilon,k}(x),\ x\in[t,+\infty);
		\end{align*}
		\item $h_{u}$ is an increasing map with the property:
		\begin{align*}
			h_{u}(x+\rho_{\alpha ,\varepsilon,k}(x))-h_{u}(x) & \geq (\alpha-1) \cdot h_{u}(x) +M_{\varepsilon,k}(x),\ x\in [t,+\infty).
		\end{align*}
	\end{enumerate}
	
	A map $h_{l}:[t,+\infty)\to\mathbb R$ satisfying property 2 is called a \emph{lower}
	\emph{map of type} $(\alpha ,\varepsilon,k)$, and a map $h_{u}:[t,+\infty)\to\mathbb R$
	satisfying property 3 is called an \emph{upper map of type $\left(\alpha ,\varepsilon,k\right)$}.
	A pair $(h_{l},h_{u})$ of maps $h_{l},\, h_{u}:[t,+\infty)\to\mathbb{R}$, satisfying properties 1-3 is called a \emph{lower-upper pair of type} $(\alpha ,\varepsilon,k)$.
	
	Let $h_{l},h_{u}:[t,+\infty)\to\mathbb R$, $t>\exp ^{\circ k}(0)$, be a lower (resp.) upper map of type $(\alpha ,\varepsilon ,k)\in \mathbb{R}_{>1}\times \mathbb{R}_{>0}\times \mathbb{N}$. Now, put 
	\begin{align*}
		D_{h_{l},h_{u}} & := \left\lbrace \zeta \in \mathbb{C} : \Re (\zeta ) \geq t, \, h_{l}(\Re (\zeta )) < \Im (\zeta ) < h_{u}(\Re (\zeta ))\right\rbrace .
	\end{align*}
	
	As in \cite[Definition 3.1]{PRRSDulac21} for hyperbolic maps, we similarly define the admissible domain of type $(\alpha , \varepsilon ,k)$.
	
	\begin{defn}[Admissible domain]\label{DefinitionAdmissibleDom}
		Let $\alpha \in \mathbb{R}_{>1}$, $\varepsilon \in \mathbb{R}_{>0}$ and $k\in\mathbb{N}$. A \emph{domain of type} $(\alpha ,\varepsilon,k)$ (or an \emph{$(\alpha, \varepsilon ,k)$-domain}) is defined as a union of an arbitrary nonempty collection of subsets of the form $D_{h_{l},h_{u}}\subseteq\mathbb{C}$ defined above.
		
		Similarly, a subset $D\subseteq\mathbb{C}$ which contains a $(\alpha, \varepsilon ,k)$-domain is called an \emph{admissible domain of type} $(\alpha ,\varepsilon,k)$ (or an \emph{$(\alpha , \varepsilon ,k)$-admissible domain}).
	\end{defn}
	
	\begin{rem}\label{rem:unionstrong}
		Note that an arbitrary union of domains of type $(\alpha,\varepsilon,k)$ is again a domain of type $(\alpha,\varepsilon,k)$.
	\end{rem}

	In Proposition~\ref{PropCriteriumForUpper} we state sufficient conditions for upper maps. We use Proposition~\ref{PropCriteriumForUpper} in Example~\ref{ExampleQuadraticDomainStrong}, where we prove that standard quadratic domains are admissible. 
	
	\begin{prop}[Sufficient condition for upper maps]\label{PropCriteriumForUpper}
		Let $h:\left[ t,+\infty \right) \to \mathbb{R}$, $t>0$, be a $C^{1}$-map such that $x\mapsto \frac{h(x)}{x}$ is an increasing map on $\left[ t,+\infty \right) $. Let $d>0$ be such that
		\begin{align}
			h'(x) & \geq d+\frac{h(x)}{x} , \label{PropositionEquationUpper}
		\end{align}
		for $x\geq t$. Then, for every $(\alpha , \varepsilon ,k) \in \mathbb{R}_{>1}\times \mathbb{R}_{>0}\times \mathbb{N}$, there exists $t'\geq t$ large enough, such that the restriction $h|_{\left[ t',+\infty \right) }$ is an upper map of type $(\alpha,\varepsilon,k)$.
	\end{prop}

	\begin{proof}
		Let $(\alpha , \varepsilon ,k) \in \mathbb{R}_{>1}\times \mathbb{R}_{>0}\times \mathbb{N}$. Since $M_{\varepsilon ,k}$ is strictly decreasing, it follows that for sufficiently large $t'\geq t$ the map
		\begin{align*}
			& x\mapsto \frac{(\alpha -1)h(x)+M_{\varepsilon ,k}(x)}{(\alpha -1)x-M_{\varepsilon ,k}(x)} - \frac{(\alpha -1)h(x)}{(\alpha -1)x}, \quad x\geq t',
		\end{align*}
		is strictly positive, tending to $0$ as $x\to +\infty $. Now, there exists $t'\geq t$ such that $\rho _{\alpha , \varepsilon ,k}(x)=(\alpha -1)x-M_{\varepsilon ,k}(x)>0$ and
		\begin{align}
			& 0< \frac{(\alpha -1)h(x)+M_{\varepsilon ,k}(x)}{(\alpha -1)x-M_{\varepsilon ,k}(x)} - \frac{(\alpha -1)h(x)}{(\alpha -1)x} < d , \label{CriteriumCondit1}
		\end{align}
		for every $x\geq t'$.
		
		By the Mean Value Theorem, for every $x\geq t'$ there exists $\theta \in (0,1)$ such that
		\begin{align}
			\frac{h\big( x+(\alpha -1)x-M_{\varepsilon ,k}(x) \big) -h(x)}{(\alpha -1)x-M_{\varepsilon ,k}(x)} & = h'\big( x+\theta ((\alpha -1)x-M_{\varepsilon ,k}(x)) \big) \nonumber \\
			& \geq d+\frac{h\big( x+\theta ((\alpha -1)x-M_{\varepsilon ,k}(x)) \big) }{x+\theta ((\alpha -1)x-M_{\varepsilon ,k}(x))} \nonumber \\
			& \geq d+ \frac{(\alpha -1)h(x)}{(\alpha -1)x} . \label{CriteriumCondit2}
		\end{align}
		The last line in \eqref{CriteriumCondit2} follows since $x\mapsto \frac{h(x)}{x}$ is an increasing map. Now, by \eqref{CriteriumCondit1} and \eqref{CriteriumCondit2}, it follows that
		\begin{align*}
			\frac{h\big( x+(\alpha -1)x-M_{\varepsilon ,k}(x) \big) -h(x)}{(\alpha -1)x-M_{\varepsilon ,k}(x)} & \geq \frac{(\alpha -1)h(x)+M_{\varepsilon ,k}(x)}{(\alpha -1)x-M_{\varepsilon ,k}(x)} ,
		\end{align*}
		for every $x\geq t'$. This implies that
		\begin{align*}
			h\big( x+\rho _{\alpha , \varepsilon ,k}(x) \big) -h(x) & \geq (\alpha -1)h(x)+M_{\varepsilon ,k}(x) ,
		\end{align*}
		for every $x\geq t'$, which proves the proposition.
	\end{proof}

	\begin{prop}[Sufficient condition for lower maps]\label{PropCriteriumForLower}
		Let $h:\left[ t,+\infty \right) \to \mathbb{R}$, $t>0$, be a $C^{1}$-map such that $x\mapsto \frac{h(x)}{x}$ is a decreasing map on $\left[ t,+\infty \right) $. Let $d>0$ be such that
		\begin{align*}
			h'(x) & \leq \frac{h(x)}{x} -d ,
		\end{align*}
		for $x\geq t$. Then, for $(\alpha , \varepsilon ,k) \in \mathbb{R}_{>1}\times \mathbb{R}_{>0}\times \mathbb{N}$, there exists $t'\geq t$ large enough, such that the restriction $h|_{\left[ t',+\infty \right) }$ is a lower map of type $(\alpha,\varepsilon,k)$.
	\end{prop}

	\begin{proof}
		Similarly as the proof of Proposition \ref{PropCriteriumForUpper}.
	\end{proof}
	
	\begin{example}[Standard quadratic domains]\label{ExampleQuadraticDomainStrong}
		Let $\mathcal{R}_{C}\subseteq\mathbb{C}$, $C>0$, be a standard quadratic domain defined in Definition~\ref{DefinitStandardQuadraticDomain} and let $\alpha \in \mathbb{R}_{>1}$, $\varepsilon \in \mathbb{R}_{>0}$ and $k\in \mathbb{N}$. By Remark~\ref{RemarkStandardQ} the upper half of the boundary of $\mathcal{R}_{C}$ is parametrized by:
		\begin{align*}
			r\mapsto x(r)+\mathrm{i}\cdot y(r)= C&\sqrt[4]{r^{2}+1}\cos\left(\frac{1}{2}\mathrm{arctg}\,r\right)+\\
			&+\mathrm{i}\cdot\left(r+C\sqrt[4]{r^{2}+1}\sin\left(\frac{1}{2}\mathrm{arctg}\,r\right)\right),\ r\in\left[0,+\infty\right).
		\end{align*}
		As in \cite[Example (3)]{PRRSDulac21} it can be shown that $r\mapsto y(r)$ and $r\mapsto x(r)$ are strictly increasing $C^{1}$-maps. Now, there exists $t\in \mathbb{R}_{>0}$ such that $x(t) > \exp ^{\circ k}(0)$, $h_{u}:=y\circ x^{-1}$ is strictly increasing $C^{1}$-map on $[x(t),+\infty)$, and
		\begin{align}
			& h_{u}'(r)=\frac{dy(r)}{dx(r)} =\frac{2(r^{2}+1)^{\frac{3}{4}}}{C(rs_{2}(r)-s_{1}(r))} + \frac{rs_{1}(r)+s_{2}(r)}{rs_{2}(r)-s_{1}(r)} \geq \frac{2}{Cs_{2}(r)}\sqrt{r} + \frac{s_{1}(r)}{s_{2}(r)} , \label{DerivationOfBoundaryQuadDomstrong}
		\end{align}
		where $s_{1}(r):=\sin (\frac{1}{2}\mathrm{arctg}\, r)$ and $s_{2}(r):=\cos (\frac{1}{2}\mathrm{arctg}\, r)$, for each $r\in \left[ 0,+ \infty \right) $.
		
		Note that:
		\begin{align}
			\frac{h_{u}(x(r))}{x(r)} & = \frac{y(r)}{x(r)} = \frac{r+Cs_{1}(r)\sqrt[4]{r^{2}+1}}{Cs_{2}(r)\sqrt[4]{r^{2}+1}} = \frac{r}{Cs_{2}(r)\sqrt[4]{r^{2}+1}} + \frac{s_{1}(r)}{s_{2}(r)} \nonumber \\
			& \leq \frac{1}{Cs_{2}(r)}\sqrt{r} + \frac{s_{1}(r)}{s_{2}(r)} , \label{DerivationOfBoundaryQuadDomstrongstrong2}
		\end{align}
		for each $r\in \left[ 0,+ \infty \right) $ such that $x(r)\geq t$. Since $r\mapsto x(r)$ is strictly increasing, from \eqref{DerivationOfBoundaryQuadDomstrongstrong2}, we see that $x\mapsto \frac{h_{u}(x)}{x}$ is an increasing map on $[x(t),+\infty) $. Furthermore, from \eqref{DerivationOfBoundaryQuadDomstrong} and \eqref{DerivationOfBoundaryQuadDomstrongstrong2}, it follows that \eqref{PropositionEquationUpper} holds for $d:=\frac{1}{Cs_{2}(t)}\sqrt{t} >0$ and the restriction $h_{u}|_{\left[ x(t),+\infty \right) }$. By Proposition~\ref{PropCriteriumForUpper}, there exists $t'\geq t$ large enough such that the restriction $h_{u}|_{\left[ x(t),+\infty \right) }$ is an upper map of type $(\alpha , \varepsilon ,k)$.
		
		Since the lower half of the boundary of $\mathcal{R}_{C}$ is symmetric to the upper half, using Proposition~\ref{PropCriteriumForLower}, similarly we can show that an appropriate restriction of the lower half of the boundary of $\mathcal R_{C}$ represents the graph of a lower map of type $\left( \alpha , \varepsilon, k\right) $.
		
		Therefore, $\mathcal{R}_{C}$ is an admissible domain of type $(\alpha , \varepsilon , k)$. Furthermore, there exists $R>0$ such that $(\mathcal{R}_{C})_{R}:=\mathcal{R}_{C}\cap \left( \left[ R,+\infty \right) \times \mathbb{R}\right) $ is a domain of type $(\alpha , \varepsilon , k)$.
	\end{example}

\medskip

\subsection{Analytic normalization on admissible domains (Theorem~B)}
	For $D\subseteq \mathbb{C}$ and $R\in \mathbb{R}_{>0}$ we define
	\begin{align*}
		D_{R} := D\cap \left( \left[ R,+\infty \right) \times \mathbb{R} \right) .
	\end{align*}
	For $f:D\to \mathbb{C}$ we define the \emph{maximal} $f$\emph{-invariant subdomain} of $D$ as the union of all $f$-invariant subdomains of $D$, and denote it by $D^{f}$.
	
	Furthermore, if $D$ is an admissible domain of type $(\alpha , \varepsilon ,k)\in \mathbb{R}_{>1}\times \mathbb{R}_{>0}\times \mathbb{N}$, then we call the union of all subdomains of $D$ of type $(\alpha , \varepsilon ,k)$, the \emph{maximal subdomain of type} $(\alpha , \varepsilon ,k)$ of the domain $D$, and denote it by $\overline{D}$.
	
	Finally, we use the following notation: $D^{f}_{R}:=(D^{f})_{R}$ and $\overline{D}_{R}:=(\overline{D})_{R}$, for every $R\in \mathbb{R}_{>0}$.
	
	\begin{prop}\label{prop:invariance of D strong}
		Let $\alpha \in \mathbb{R}_{>1}$, $\varepsilon \in \mathbb{R}_{>0}$ and $k\in\mathbb{N}$. Let $D\subseteq \mathbb{C}$ be an admissible domain of
		type $(\alpha,\varepsilon,k)$ and let $f:D_C\to\mathbb C$, $C>\exp^{\circ k}(0)$,
		be an analytic map with the following asymptotic behaviour:
		\begin{equation}\label{eq:asystrong}
			f(\zeta)=\alpha \zeta +o(\boldsymbol{L}_{k}^{-\varepsilon }),\ 
			\text { as }\mathrm{\Re }(\zeta)\to +\infty \text{ uniformly on $D_C$}.
		\end{equation}
		Here,
		\begin{align*}
			& \boldsymbol{L}_{1}:=\log\left(\zeta\right),\ldots,\boldsymbol{L}_{k}:=\log\left(\boldsymbol{L}_{k-1}\right),
		\end{align*}
		where $\log$ represents the principal branch of the logarithm\footnote{This is the notation from \cite[Proposition 3.4]{PRRSDulac21}. Note here that, for $C>\mathrm{exp}^{\circ k}(0)$, the iterated logarithms $\boldsymbol{L}_1,\ldots,\boldsymbol L_k$ are well-defined on $D_C$ (using only the principal branch of the logarithm), since $\Re(\zeta)>\mathrm{exp}^{\circ k}(0)$.}. Then, for every $R>C$ sufficiently large, the domain $\overline{D}_{R}$ is $f$-invariant. In particular, $\overline{D}_{R}\subseteq{D}_{R}^f$ and $D_R^f\neq \emptyset$, for $R>C$. 
	\end{prop}
	
	\begin{proof}
		By \eqref {eq:asystrong}, it follows that
		\begin{align}\label{eq:asylogstrong}
			\lim_{\Re(\zeta)\to+\infty}\frac{f(\zeta)- \alpha \zeta }{\boldsymbol{L}_{k}^{-\varepsilon }} & =0,
		\end{align}
		uniformly on $D_{C}$.
		
		Now, there exists $R>C$
		large enough, such that $\rho _{\alpha , \varepsilon ,k}(R)>0$, $\rho _{\alpha , \varepsilon ,k}$ is increasing on $[R,+\infty)$ and
		\begin{equation}
			\left|f(\zeta)- \alpha \zeta \right|\leq\frac{1}{\left| \boldsymbol{L}_{k}^{\varepsilon}\right|}, \label{eq:newstrong}
		\end{equation}
		for $\zeta\in D_R$. As in the proof of \cite[Proposition 3.4]{PRRSDulac21} we inductively get:
		\begin{align}
			|\boldsymbol{L}_{m}| & \geq\log^{\circ m}(\Re \left(\zeta\right)) , \label{eq:newstrong1}
		\end{align}
		for $1\leq m\leq k$ and $\zeta\in D_R$. From \eqref{eq:newstrong} and \eqref{eq:newstrong1}, it follows that:
		\begin{equation}
			|f(\zeta)- \alpha \zeta |\leq\frac{1}{(\log^{\circ k}(\Re(\zeta)))^{\varepsilon}} , \label{PropEqstrong1}
		\end{equation}
		for $\zeta\in D_R$. Now, we get that:
		\begin{align}
			\Re (f(\zeta))-\Re(\zeta) & \geq (\alpha -1) \Re(\zeta)  -\frac{1}{(\log^{\circ k}(\Re(\zeta)))^{\varepsilon}} \nonumber \\
			& =\rho_{\alpha ,\varepsilon,k}\left(\Re \left(\zeta\right)\right) , \label{EqRealPartstrong}
		\end{align}
		\begin{align}
			\Re (f(\zeta))-\Re(\zeta) & \leq (\alpha -1)\Re(\zeta) +\frac{1}{(\log^{\circ k}(\Re(\zeta)))^{\varepsilon}} \nonumber \\
			& =(\alpha -1)\Re(\zeta) + M_{\varepsilon ,k}(\Re  (\zeta )) ,\label{EqRealPart1strong}
		\end{align}
		and
		\begin{align}
			\Im(f(\zeta))-\Im(\zeta) & \geq (\alpha -1)\Im(\zeta) -\frac{1}{(\log^{\circ k}(\Re(\zeta)))^{\varepsilon}} \nonumber \\
			& =  (\alpha -1)\Im(\zeta) - M_{\varepsilon,k}(\Re \left(\zeta\right)), \label{eq:new3strong}
		\end{align}
		\begin{align}
			\Im(f(\zeta))-\Im(\zeta) & \leq  (\alpha -1)\Im(\zeta) +\frac{1}{(\log^{\circ k}(\Re(\zeta)))^{\varepsilon}} \nonumber \\
			& = (\alpha -1)\Im(\zeta) + M_{\varepsilon,k}(\Re \left(\zeta\right)) , \label{eq:new2strong}
		\end{align}
		for $\zeta\in D_R$. Note that $\rho_{\alpha ,\varepsilon,k}(\Re \left(\zeta\right))\geq\rho_{\alpha ,\varepsilon,k}(R)>0$, for every $\zeta\in D_R$, since $\rho_{\alpha,\varepsilon,k}$ is an increasing map on $\left[ R,+\infty \right) $. Now, put
		\begin{align*}
			\mathcal{S}_{\alpha ,\varepsilon,k}(\zeta) & :=\left[ \Re(\zeta) + \rho _{\alpha , \varepsilon ,k}(\Re(\zeta)),\, \alpha \cdot \Re(\zeta)+M_{\varepsilon,k}(\Re(\zeta))\right] \\
			& \times\left[\alpha \cdot \Im (\zeta ) -M_{\varepsilon,k}(\Re(\zeta)), \alpha \cdot \Im (\zeta ) +M_{\varepsilon,k}(\Re(\zeta))\right] ,
		\end{align*}
		for each $\zeta\in D_R$. By \eqref{EqRealPartstrong}-\eqref{eq:new2strong}, we get that
		\begin{align*}
			& f(\zeta)\in\mathcal{S}_{\alpha ,\varepsilon,k}(\zeta) ,
		\end{align*}
		for every $\zeta \in D_{R}$. Now we prove that $\overline D_R$ is $f$-invariant. Let $\zeta \in \overline{D}_{R}$ be arbitrary. Now, there exists an $(\alpha ,\varepsilon,k)$-domain $(D_{h_{l},h_{u}})_R\subseteq\overline{D}_R$,
		such that $\zeta\in (D_{h_{l},h_{u}})_R$. By properties 2 and 3 in the definition of a lower-upper pair of type $(\alpha, \varepsilon ,k)$ it is easy to see that $\mathcal{S}_{\alpha ,\varepsilon,k}(\zeta)\subseteq D_{h_{l},h_{u}}$. Consequently, it follows that $f(\zeta)\in D_{h_{l},h_{u}}\subseteq\overline{D}$. By \eqref{EqRealPartstrong}, since $\zeta\in D_R$ and $\rho_{\alpha ,\varepsilon,k}(\Re(\zeta))>0,$ for $\zeta\in D_R$, it follows that $\Re (f(\zeta ))>\Re(\zeta)\geq R$. Since $f(\zeta)\in D_{h_{l},h_{u}}\subseteq\overline{D}$, we get that $f(\zeta )\in \overline D_R$. Since $\overline{D}_{R}\neq \emptyset $ and $\overline{D}_{R}\subseteq D^{f}_{R}$, we get that $D^{f}_{R}\neq \emptyset $, for every large enough $R>C$. 
	\end{proof}

	\begin{example}\label{ExampleStandardQuadraticDom1}
		Let $\alpha \in \mathbb{R}_{>1}$, $\varepsilon \in \mathbb{R}_{>0}$, $k\in \mathbb{N}$, and let $\mathcal{R}_{C}$, $C>0$, be a standard quadratic domain. By Example~\ref{ExampleQuadraticDomainStrong}, it follows that there exists $R\in \mathbb{R}_{>0}$ large enough, such that $(\mathcal{R}_{C})_{R}$ is a domain of type $(\alpha , \varepsilon ,k)$. It implies that $(\mathcal{R}_{C})_{R}\subseteq (\overline{\mathcal{R}_{C}})_{R}$. Since, $(\overline{\mathcal{R}_{C}})_{R}\subseteq (\mathcal{R}_{C})_{R}$, it follows that $(\mathcal{R}_{C})_{R}=(\overline{\mathcal{R}_{C}})_{R}$. By Proposition~\ref{prop:invariance of D strong}, it follows that $(\mathcal{R}_{C})_{R}=(\overline{\mathcal{R}_{C}})_{R}\subseteq (\mathcal{R}_{C})^{f}_{R}$. Since $(\mathcal{R}_{C})^{f}_{R}\subseteq (\mathcal{R}_{C})_{R}$, we get that $(\mathcal{R}_{C})^{f}_{R}=(\overline{\mathcal{R}_{C}})_{R}=(\mathcal{R}_{C})_{R}$ (see Figure~\ref{Figure2}).
	\end{example}

	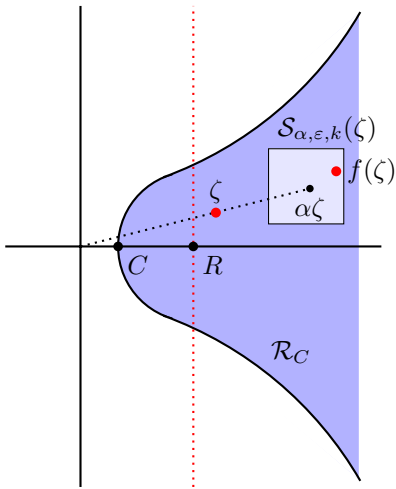
\begin{figure}[h!]
		\centering
		\begin{tikzpicture}
			\filldraw[draw=none,fill=blue!30!white] (3.7,-3.1) -- (1.2,-0.96) -- (1.2,0.96) -- (3.7,3.1) -- cycle;
			\filldraw[draw=black,thick,fill=blue!30!white] (1.5,0) circle (1cm);
			\filldraw[draw=none,thick,fill=blue!30!white,rounded corners] (1.22,-1.05) rectangle (3.1,1.05);
			\filldraw[draw=black,thick,fill=white] (1.1,0.92) arc (290:330:5cm);
			\filldraw[draw=black,thick,fill=white] (1.1,-0.92) arc (70:30:5cm);
			\draw (0.5,0) node[draw=black,thick,fill=black,circle,scale=0.3]{};
			\draw (0.5,0) node[below right]{$C$} (2.8,-1.4) node{$\mathcal{R}_{C}$};
			\draw[black,thick] (-1,0) -- (4,0) (0,-3.2) -- (0,3.2);
			\filldraw[draw=black,fill=blue!10!white] (2.5,0.3) rectangle (3.5,1.3);
			\draw[black,thick,dotted] (0,0) -- (3.05,0.77);
			\draw (2.5,1.25) node[above right]{$\mathcal{S}_{\alpha , \varepsilon ,k}(\zeta )$} (3.05,0.77) node[below]{$\alpha \zeta$} (1.8,0.45) node[above]{$\zeta $} (3.42,1) node[right]{$f(\zeta )$};
			\draw (3.05,0.77) node[draw=none,fill=black,circle,scale=0.3]{} (1.8,0.45) node[draw=none,fill=red,circle,scale=0.4]{} (3.4,1) node[draw=none,fill=red,circle,scale=0.4]{};
			\draw[red,thick,dotted] (1.5,-3.2) -- (1.5,3.2);
			\draw (1.5,0) node[draw=black,thick,fill=black,circle,scale=0.3]{};
			\draw (1.5,0) node[below right]{$R$};
		\end{tikzpicture}
		\caption{Application of Proposition~\ref{prop:invariance of D strong} for $D:=\mathcal{R}_{C}$, $C>0$.}
		\label{Figure2}
	\end{figure}

	Now we prove the main result of this section.

	\begin{thmB}[Normalization of analytic maps with strongly hyperbolic logarithmic asymptotics on admissible domains]\label{TmNormalizationStrongAnalytic}
		Let $\alpha\in\mathbb R_{>1}$, $\varepsilon \in \mathbb{R}_{>0}$ and $k\in\mathbb{N}$. Let $D\subseteq\mathbb C^+$
		be an admissible domain of type $(\alpha ,\varepsilon,k)$. For $C>\exp^{\circ k}(0)$, let $f:D_C\to\mathbb C$ 
		be an analytic map such that
		\begin{equation}\label{eq:conditstrong}
			f(\zeta)= \alpha \zeta +o(\boldsymbol{L}_{k}^{-\varepsilon }),\ 
			\text { as }\mathrm{\Re }(\zeta)\to +\infty \text{ uniformly on $D_C$}.
		\end{equation}
		Here, the iterated logarithm $\boldsymbol{L}_k$ is defined as in Proposition~\ref{prop:invariance of D strong}. Then:
		\begin{enumerate}[1., font=\textup, topsep=0.4cm, itemsep=0.4cm, leftmargin=0.6cm]
			\item \emph{(Existence)} For a sufficiently large $R>\exp^{\circ k}\left(0\right)$ there exists
			an analytic normalizing map $\varphi$ on the $f$-invariant
			subdomain $D_{R}^f\subseteq D$. That is, $\varphi$ 
			satisfies 
			\begin{align}
				(\varphi\circ f)(\zeta) & = \alpha \cdot \varphi(\zeta) ,\text{ for all }\zeta\in D_{R}^f. \label{eq:linera-1strong}
			\end{align}
			Moreover, $\varphi$ is the uniform limit on $D_R^f$ of the B\"ottcher sequence
			\begin{align}
				& \left( \frac{1}{\alpha ^{n}}f^{\circ n}\right) _{n} \label{BotSeqKoenigs}
			\end{align}
			in the $\zeta $-chart.
			\item If $D_R^f\cap \{\zeta\in\mathbb C^+:\Im(\zeta)=0\}$ is $f$-invariant, 
			then it is also $\varphi $-invariant.
			\item {\em(Asymptotics)} The normalization $\varphi$ is \emph{tangent to identity}, i.e., $\varphi (\zeta)=\zeta+o(1)$, uniformly on $D_{R}^f\subseteq \mathbb C^{+}$, as $\Re(\zeta)\to +\infty$.
			
			In particular, for every $0<\nu < \varepsilon $, it follows that $\varphi(\zeta)=\zeta +o(\boldsymbol{L}_{k}^{-\nu})$, uniformly as $\Re(\zeta)\to +\infty$, on every subdomain $D_{h_{l},h_{u}}\subseteq D_R^f$ such that $h_{l}(x)=O(x)$ and $h_{u}(x)=O(x)$. 
			\item \emph{(Uniqueness)} Let $\psi:D_{1}\to\mathbb C$, be
			a normalization of $f$ on an $f$-invariant subset $D_{1}\subseteq D$, such that $\psi(\zeta)=\zeta+o(1)$ uniformly on $D_{1}$, as $\Re(\zeta)\to +\infty$. Then $\psi\equiv \varphi$ on $(D_{1})_{R}$, where $R$ is from statement 1. 
		\end{enumerate}
	\end{thmB}
	
	\begin{proof}
		1. By Proposition~\ref{prop:invariance of D strong},
		there exists $R>\exp^{\circ k}\left(0\right)$ such
		that $\overline{D}_{R}\subseteq{D}_{R}^f$ and $\rho _{\alpha , \varepsilon ,k}(R) >0$. Consequently, it follows that ${D}_{R'}^f\neq\emptyset$, for all $R'\geq R$. Let $\zeta\in D_{R}^f$. By \eqref{EqRealPartstrong}, since $\rho_{\alpha,\varepsilon,k}$ is increasing on $[R,+\infty)$, it follows that:
		\begin{equation}\label{Equat1strong}
			\Re (f^{\circ n}(\zeta))\geq \Re \left(\zeta\right)+n\rho_{\alpha ,\varepsilon,k}(\Re \left(\zeta\right))\geq R+n\rho_{\alpha,\varepsilon,k}(R),\ n\in\mathbb{N} .
		\end{equation}
		By \eqref{PropEqstrong1}, it follows that
		\begin{equation}\label{Equati2strong}
			\left|f(\zeta)-\alpha \zeta \right|\leq\frac{1}{(\log^{\circ k}(\Re \left(\zeta\right)))^{\varepsilon}}=M_{\varepsilon,k}(\Re \left(\zeta\right)) ,
		\end{equation}
		for every $\zeta\in D_R^f$. Inductively, from \eqref{Equat1strong} and \eqref{Equati2strong}, since $M_{\varepsilon,k}$ is decreasing on $[R,+\infty)$, for every $n\in\mathbb{N}$ and $\zeta\in D_R^f$ we get that:
		\begin{align}\label{eq:jojstrong}
			\left|\frac{1}{\alpha ^{n+1}}f^{\circ(n+1)}(\zeta)- \frac{1}{\alpha ^{n}}f^{\circ n}(\zeta) \right| & =\frac{1}{\alpha ^{n+1}}\left|f\left(f^{\circ n}(\zeta)\right)- \alpha f^{\circ n}(\zeta) \right| \nonumber \\
			& \leq \frac{1}{\alpha ^{n+1}} M_{\varepsilon,k}\bigl(\Re (f^{\circ n}(\zeta))\bigr)\nonumber \\
			& \leq \frac{1}{\alpha ^{n+1}} M_{\varepsilon,k}\left(R+n\rho_{\alpha ,\varepsilon,k}(R)\right).
		\end{align}
		Since $M_{\varepsilon,k}\left(R+n\rho_{\alpha ,\varepsilon,k}(R)\right) $ tends to $0$, as $n\to +\infty $, it follows that $$\sum \frac{1}{\alpha ^{n+1}} M_{\varepsilon,k}\left(R+n\rho_{\alpha ,\varepsilon,k}(R)\right) $$ converges. Consequently, the B\"ottcher sequence $\big( \frac{1}{\alpha ^{n}}f^{\circ n} \big) _{n}$ is uniformly Cauchy, and, therefore, converges uniformly on $D_R^f$. Let $\varphi$ be its uniform limit on ${D}_{R}^f$. By the Weierstrass' Theorem, $\varphi$ is analytic on  ${D}_{R}^f$.
		
		In the end, we check that $\varphi $ is a solution of the normalization equation:
		\begin{align*}
			(\varphi\circ f)(\zeta) & =\lim _{n\to \infty }\left(\frac{1}{\alpha ^{n}}f^{\circ n}(f(\zeta)) \right) \\
			& =\alpha \cdot \lim _{n\to \infty }\left(\frac{1}{\alpha ^{n+1}}f^{\circ(n+1)}(\zeta) \right) \\
			& = \alpha \cdot \varphi(\zeta) , \quad \zeta \in D^{f}_{R} .
		\end{align*}
		
		2. Let $D_R^f\cap \{\zeta\in\mathbb C: \Im(\zeta)=0\}$ be $f$-invariant. Recall that
		\begin{equation}\label{eq:fistrong}
			\varphi(\zeta)=\lim_{n\to\infty}\left( \frac{1}{\alpha ^{n}}f^{\circ n}(\zeta) \right) ,
		\end{equation}
		for every $\zeta \in D_R^f\cap \{\zeta\in\mathbb C: \Im(\zeta)=0\} $. Since $\{\zeta \in \mathbb{C} : \Im(\zeta) =0\}$ is closed in $\mathbb{C}^+$, by \eqref{eq:fistrong}, it follows that $D_R^f\cap \{\zeta\in\mathbb C: \Im(\zeta)=0\}$ is $\varphi$-invariant.
		
		3. Now, taking the sum of the terms
		\begin{align*}
			& \frac{1}{\alpha ^{n+1}}f^{\circ(n+1)}(\zeta)- \frac{1}{\alpha ^{n}}f^{\circ n}(\zeta)
		\end{align*}
		in \eqref{eq:jojstrong}, for $n\in \left\lbrace 0, \ldots , m-1\right\rbrace $, since $M_{\varepsilon ,k}$ is decreasing, it follows that:
		\begin{align}
			\left|\frac{1}{\alpha ^{m}}f^{\circ(m)}(\zeta)- \zeta \right| & \leq \sum_{n=0}^{m-1}\frac{1}{\alpha ^{n+1}}M_{\varepsilon,k}\left(\Re \left(\zeta\right)+n\rho_{\alpha ,\varepsilon,k}(\Re \left(\zeta\right))\right) \nonumber \\
			& \leq M_{\varepsilon,k}\left( \Re  \left( \zeta \right) \right) \cdot \sum_{n=0}^{m-1}\frac{1}{\alpha ^{n+1}} \nonumber \\
			& \leq M_{\varepsilon,k}\left( \Re  \left( \zeta \right) \right) \cdot \frac{1}{1- \frac{1}{\alpha }} . \label{eqConvergenceStrong}
		\end{align}
		Now, as ${m\to+\infty}$, by \eqref{eqConvergenceStrong}, it follows that 
		\begin{equation}\label{eq:hahhstrong}
			\left|\varphi(\zeta) - \zeta\right| \leq \frac{1}{1- \frac{1}{\alpha }} \cdot M_{\varepsilon,k}\left( \Re  \left( \zeta \right) \right) ,
		\end{equation}
		for each $\zeta \in D_R^f$. From \eqref{eq:hahhstrong} we conclude that $\varphi(\zeta )=\zeta + o(1)$, uniformly on $D_{R}^f$ as $\Re  (\zeta )\to + \infty $.
		
		Now, the remainder of the proof of statement 3 follows as in the proof of statement 3 of \cite[Theorem~A]{PRRSDulac21}. 
		
		4. Let $\psi$ be an analytic germ such that $\psi\circ f=\alpha \cdot \psi $, on an $f$-invariant subset $D_{1}\subseteq D$, and $\psi (\zeta)=\zeta+o(1)$ uniformly on $D_{1}$, as $\Re(\zeta)\to+\infty$. Since $D^f$ is the maximal $f$-invariant subdomain of $D$, it follows that $D_{1}\subseteq D^{f}$, and, consequently, $(D_{1})_{R}\subseteq {D}_{R}^f$. Since $D_{1}$ is $f$-invariant, by \eqref{EqRealPartstrong}, it follows that $(D_1)_R$ is $f$-invariant, and by \eqref{Equat1strong}, nonempty. By statement 1, $\varphi$ is the analytic linearization on $D_{R}^f$ obtained as the limit of the B\"ottcher sequence, for sufficiently large $R>\exp ^{\circ k}(0)$ from statement 1, which satisfies $\varphi(\zeta)=\zeta+o(1)$, uniformly as $\Re(\zeta)\to+\infty$ on $D_R^f$. Now, put
		\begin{align*}
			E(\zeta) & :=\varphi(\zeta)-\psi(\zeta) ,
		\end{align*}
		for every $\zeta\in (D_{1})_{R}$. Note that $E$ is analytic on $(D_{1})_{R}$, such that $E(\zeta)=o(1)$, uniformly on $(D_{1})_{R}$, as $\Re(\zeta)\to +\infty $, and it satisfies
		\begin{align*}
			\frac{1}{\alpha }\cdot (E\circ f)(\zeta) &=E(\zeta) ,
		\end{align*}
		for every $\zeta\in (D_{1})_{R}$. Inductively, we get:
		\begin{equation}\label{eq:contrstrong}
			\frac{1}{\alpha ^{n}} \cdot E(f^{\circ n}(\zeta))=E(\zeta) ,
		\end{equation}
		for $\zeta\in (D_{1})_{R}$, $n\in\mathbb{N}$. By \eqref{Equat1strong}, $\Re \left(f^{\circ n}(\zeta)\right)\geq R+n\rho_{\alpha ,\varepsilon,k}(R)$, for $n\in \mathbb{N}$, and $\zeta\in (D_1)_R\subseteq D_R^f$. This implies that
		\begin{equation}\label{eq:dvastrong}
			\lim _{n\to \infty }\Re \left(f^{\circ n}(\zeta)\right)=+\infty ,
		\end{equation}
		for every $\zeta\in (D_1)_{R}$. Using \eqref{eq:dvastrong} and the fact that $E(\zeta)=o(1)$, as $\Re(\zeta)\to +\infty$, by \eqref{eq:contrstrong}, we get that $E(\zeta)=0$, for each $\zeta\in (D_1)_R$, which implies that $\varphi (\zeta ) = \psi (\zeta )$, for each $\zeta \in (D_1)_{R}$.
	\end{proof}

	\begin{rem}
		We call \eqref{BotSeqKoenigs} the B\"ottcher sequence, although it is in fact the Koenigs sequence in the logarithmic chart. In the $z$-chart it becomes the B\"ottcher sequence $\left( \left( f^{\circ n}(z)\right) ^{\frac{1}{\alpha ^{n}}}\right) _{n}$.
	\end{rem}

\medskip

\section{Analytic normalization of a strongly hyperbolic complex Dulac germs}\label{sec:proofThmB}

	This section is dedicated to the proof of Theorem~C in Subsection~\ref{SubsectionProofCompleteB}. The three main parts of the proof of Theorem~C are:
	\begin{enumerate}[1., font=\textup, topsep=0.4cm, itemsep=0.4cm, leftmargin=0.6cm]
		\item \emph{The formal part.} We prove in Proposition~\ref{PropNormalDulac} in Subsection~\ref{SubsectionFormalNorm} that the formal normalization (obtained in Theorem~A) of a strongly hyperbolic complex Dulac germ is a parabolic complex Dulac series.
		\item \emph{The analytic part.} Using the \emph{quasi-analyticity} property of complex Dulac germs and applying Theorem~B from Section~\ref{SectionNormalizationAnalytic}, we obtain the analytic normalization of a strongly hyperbolic complex Dulac germ.
		\item \emph{The asymptotics.} We relate formal and analytic normalizations of a strongly hyperbolic complex Dulac germ via asymptotic expansions on a standard quadratic domain. This is done using lemmas in Subsection~\ref{SubsectionPartialNorm} about solutions of suitable \emph{Schr\"oder's type} homological equations corresponding to partial normalizations.
	\end{enumerate}

	\medskip

	\subsection{Formal normalization of a strongly hyperbolic complex Dulac series}\label{SubsectionFormalNorm}
	
	By Remark~\ref{RemarkTheoremAGeneraliz}, Theorem~A also holds in the larger differential algebra $\mathfrak L(\mathbb{C})$. Therefore, for a strongly hyperbolic complex Dulac series $f$ there exists a unique parabolic normalization $\varphi \in \mathcal{L}_{1}^{0}(\mathbb{C})$. In Proposition~\ref{PropNormalDulac} below we prove that $\varphi $ is a complex Dulac series.
	
	\begin{prop}[Normalization of a strongly hyperbolic complex Dulac series]\label{PropNormalDulac}
		Let $\widehat{f}=z^{\alpha }+\mathrm{h.o.t.}$, $\alpha >1$, be a strongly hyperbolic complex Dulac series and let $\widehat{\varphi }$ be the parabolic normalization obtained as in Theorem~A. Then $\widehat{\varphi }$ is a parabolic complex Dulac series. Moreover, if $\widehat{f}$ is a real Dulac series, so is $\widehat{\varphi }$.
	\end{prop}
	\begin{proof}
		Let $\mathcal{P}_{f}$ be the B\"ottcher operator defined in \eqref{DefinitionPf}, in Section~\ref{sec:proofThmA}. Recall from Theorem~A that $\mathcal{P}_{\widehat{f}}^{\circ n}(\mathrm{id})=z^{\frac{1}{\alpha ^{n}}}\circ \widehat{f}^{\circ n}$, $n\in \mathbb{N} $. Since $\widehat{f}$ is a complex Dulac series and the set of all complex Dulac series is a subgroup of $\mathcal{L}_{1}^{H}(\mathbb{C})$, we deduce that $\mathcal{P}_{\widehat{f}}^{\circ n}(\mathrm{id}) $ is a complex Dulac series, for every $n\in \mathbb{N} $. Since $\widehat{f}$ is a complex Dulac series, note that $\mathrm{ord}_{z}(\widehat{f}-z^{\alpha })>\alpha $.
		
		By statement 2 of Theorem~A, it follows that $(\mathcal{P}_{\widehat{f}}(\mathrm{id}))_{n}$ converges to the parabolic normalization $\widehat{\varphi }$ in the power-metric topology on the space $\mathcal{L}_{1}(\mathbb{C})$.
		
		Let
		\begin{align*}
			\widehat{\varphi } &:= \mathrm{id}+\sum _{\beta \in \mathrm{Supp}_{z}\, (\widehat{\varphi })}z^{\beta }R_{\beta } ,
		\end{align*}
		and let $\gamma >1$. Then there exists $n\in \mathbb{N} $, such that $\mathrm{ord}_{z}(\mathcal{P}_{\widehat{f}}^{\circ n}(\mathrm{id})-\widehat{\varphi }) > \gamma $. Since $\mathcal{P}_{\widehat{f}}^{\circ n}(\mathrm{id})$ is a complex Dulac series, it follows that $R_{\beta }$ is a polynomial in the variable $\boldsymbol{\ell }_{1}^{-1}=-\log z$, for $1<\beta \leq \gamma $, and there are only finitely many $1<\beta \leq \gamma $, such that $R_{\beta } \neq 0$. Thus, the sequence of exponents of $z$ is strictly increasing, strictly bigger than or equal to $1$ and tending to $+\infty $, which implies that $\widehat{\varphi }$ is a parabolic complex Dulac series.
		
		Moreover, if $\widehat{f}$ is a real Dulac series, then $\widehat{\varphi }$ is also a real Dulac series.
	\end{proof}

	\begin{rem}
		$(1)$ Proposition~\ref{PropNormalDulac} can be viewed as an analogue of \cite[Lemma 4.2]{PRRSDulac21} for strongly hyperbolic complex Dulac series. Since the B\"ottcher sequence converges in the power-metric topology, the proof of Proposition~\ref{PropNormalDulac} is simpler than the proof of \cite[Lemma 4.2]{PRRSDulac21}. \\
		
		\noindent $(2)$ By Proposition~\ref{PropNormalDulac} and \cite[Lemma 4.2]{PRRSDulac21}, it follows that normalizations of strongly hyperbolic and hyperbolic complex Dulac germs are again complex Dulac germs. This is not true in general for normalizations of parabolic Dulac germs, which makes relations between analytic and formal normalizations much more complicated (see \cite{mrrz2} and \cite{mrrz3}). \\
		
		\noindent $(3)$ Let $\widehat{f}(\zeta )=\alpha \zeta + \sum _{i=1}^{+\infty }P_{i}(\zeta )\mathrm{e}^{- \alpha _{i}\zeta } $ be a strongly hyperbolic complex Dulac series given in the $\zeta $-chart and let $\widehat{\varphi }(\zeta )$ be its parabolic normalization obtained in Theorem~A, given in the $\zeta $-chart. By Proposition \ref{PropNormalDulac}, $\widehat{\varphi }(\zeta )$ is parabolic complex Dulac series, which implies that
		\begin{align*}
			\widehat{\varphi }(\zeta ) & =\zeta + \sum _{i=1}^{+\infty }Q_{i}(\zeta )\mathrm{e}^{- \beta _{i}\zeta } ,
		\end{align*}
		for the sequence of complex polynomials $(Q_{i})$ and the strictly increasing sequence $(\beta _{i})$ of positive real numbers tending to $+\infty $, as $i\to +\infty $. Note that $\widehat{\varphi }(\zeta ) \circ \widehat{f}(\zeta ) =\alpha \cdot \widehat{\varphi }(\zeta )$ in the $\zeta $-chart. \\
		
		\noindent $(4)$ Let $\widehat{f}(\zeta )$ be a formal sum $\widehat{f}(\zeta ):=\sum _{i=1}^{+\infty }R_{i}(\zeta )\mathrm{e}^{- \alpha _{i}\zeta }$, where $(R_{i})$ is a sequence of Laurent series (in the variable $\zeta ^{-1}$) with complex coefficients and $(\alpha _{i})$ strictly increasing sequence of positive real numbers such that $\left\lbrace \alpha _{i}:i\in \mathbb{N}\right\rbrace $ is a well-ordered set. Then we define \emph{the order of} $\widehat{f}(\zeta )$ \emph{in} $\mathrm{e}^{-1}$ as minimal $\alpha _{i}$ such that $R_{i}\neq 0$ (if $\widehat{f}(\zeta ) \neq 0$), and denote it by $\mathrm{ord}_{\mathrm{e}^{-1}} \, \widehat{f}(\zeta )$. If $\widehat{f}(\zeta )=0$, then we put $\mathrm{ord}_{\mathrm{e}^{-1}} \, \widehat{f}(\zeta ):=+\infty $. 
	\end{rem}

\medskip

\subsection{Partial normalizations and homological equations}\label{SubsectionPartialNorm}

	In this subsection we define partial normalizations of strongly hyperbolic complex Dulac germs that solve \emph{Schr\"oder's type homological equations}. This is used in the following subsection in the proof of Theorem~C.
	
	\begin{defn}[Formal and analytic partial normalizations]\label{DefinitionPartial}
		Let $f$ be a strongly hyperbolic complex Dulac germ and let $\widehat{\varphi }$ be its unique formal parabolic complex Dulac normalization obtained in Theorem~A. Let
		\begin{align}
			\widehat{\varphi }(\zeta ) & :=\zeta + \sum _{i=1}^{+\infty }Q_{i}(\zeta )\mathrm{e}^{- \beta _{i}\zeta } \label{NormPartial}
		\end{align}
		and let
		\begin{align*}
			\widehat{\varphi }_{0}(\zeta ) & := \zeta , \\
			\widehat{\varphi }_{n}(\zeta ) & :=\zeta + \sum _{i=1}^{n}Q_{i}(\zeta )\mathrm{e}^{- \beta _{i}\zeta } , \qquad n\in \mathbb{N}_{\geq 1} .
		\end{align*}
		We call $(\widehat{\varphi }_{n})$ the \emph{sequence of formal partial normalizations} of $f$.
		
		Note that every $\widehat{\varphi }_{n}$ trivially defines the corresponding analytic map on $\mathbb{C}^{+}$ that we denote by $\varphi _{n}$, $n\in \mathbb{N}$. We call $(\varphi _{n})$ the \emph{sequence of analytic partial normalizations} of $f$.
	\end{defn}

	\begin{lem}\label{LemmaDulacExpan}
		Let $f$ be a strongly hyperbolic complex Dulac germ defined on a standard quadratic domain $\mathcal{R}_{C}$, $C>0$, and let $\widehat{f}(\zeta )=\alpha \zeta +\mathrm{h.o.t.}$, $\alpha \in \mathbb{R}_{>1}$, be its complex Dulac asymptotic expansion in the $\zeta $-chart. Let $\widehat{\varphi }$ be the formal normalization of $\widehat{f}$ from Proposition~\ref{PropNormalDulac} in the $\zeta $-chart given in \eqref{NormPartial}. Let $(\varphi _{n})$ be the related sequence of analytic partial normalizations of $f$, as defined in Definition~\ref{DefinitionPartial}. Then, for every $n\in \mathbb{N}$, there exists $\varepsilon _{n}>0$ such that
		\begin{align}
			(\varphi _{n} \circ f-\alpha \varphi _{n})(\zeta )  & =o \big( \mathrm{e}^{- (\beta _{n}+ \varepsilon_{n} )\zeta } \big) , \quad (\beta _{0}:=0) , \label{BoundLemmaEq2}
		\end{align}
		uniformly on $\mathcal{R}_{C}$, as $\Re(\zeta ) \to + \infty $.
	\end{lem}
	
	\begin{proof}
		Let $\widehat{f}:=\alpha \zeta + \sum _{n\geq 1}R_{n}(\zeta )\mathrm{e}^{-\alpha _{n}\zeta }$, where $\alpha >1$, $(R_{n})$ is a sequence of polynomials in the variable $\zeta $ and $(\alpha_{n})$ a strictly increasing sequence of positive real numbers tending to $+\infty $. Put:
		\begin{align*}
			\widehat{f}_{0} & :=\alpha \zeta , \\
			\widehat{f}_{n} & :=\alpha \zeta + \sum _{i\in \mathbb{N}_{\geq 1}, \, \alpha _{i}\leq \beta  _{n}}R_{i}(\zeta )\mathrm{e}^{-\alpha _{i}\zeta } , \, \qquad n\in \mathbb{N}_{\geq 1} .
		\end{align*}
		Since $\widehat{f}_{n}$, $n\in \mathbb{N}$, are finite sums of terms, we denote the related sequence of maps as $(f_{n})$. Put $\widehat{\varphi }_{>n}:=\widehat{\varphi }-\widehat{\varphi }_{n}$ and $\widehat{f}_{>n}:=\widehat{f}-\widehat{f}_{n}$, $n\in \mathbb{N} $. Since $\widehat{\varphi }\circ \widehat{f}=\alpha \cdot \widehat{\varphi }$, it follows that:
		\begin{align*}
			(\widehat{\varphi }_{n}+\widehat{\varphi }_{>n}) \circ (\widehat{f}_{n}+\widehat{f}_{>n}) &= \alpha \cdot (\widehat{\varphi }_{n}+\widehat{\varphi }_{>n}) , \\
			\widehat{\varphi }_{n} \circ (\widehat{f}_{n}+\widehat{f}_{>n})+ \widehat{\varphi }_{>n}\circ (\widehat{f}_{n}+\widehat{f}_{>n}) &= \alpha \cdot \widehat{\varphi }_{n}+\alpha \cdot \widehat{\varphi }_{>n} , \quad n\in \mathbb{N} .
		\end{align*}
		Using the formal Taylor Theorem, this implies that, for every $n\in \mathbb{N}$, there exists $\mu _{n}>0$ such that:
		\begin{align*}
			\mathrm{ord}_{\mathrm{e}^{-\zeta }}(\widehat{\varphi }_{n} \circ \widehat{f}_{n}-\alpha \widehat{\varphi }_{n}) & = \beta  _{n}+\mu _{n} .
		\end{align*}
		Therefore, for every $n\in \mathbb{N}$, it follows that:
		\begin{align}
			\varphi _{n} (f_{n}(\zeta ))-\alpha \varphi _{n}(\zeta ) & = o\big( \mathrm{e}^{-(\beta  _{n}+\mu _{n}) \zeta } \big) , \label{BoundForCompEq}
		\end{align}
		uniformly on $\mathcal{R}_{C}$, as $\Re (\zeta ) \to + \infty $. Put
		\begin{align}
			f_{>n}(\zeta ) &:=(f-f_{n})(\zeta ), \quad \zeta \in \mathcal{R}_{C}, \: n\in \mathbb{N} . \nonumber 
		\end{align}
		It is easy to see that $f_{>n} \sim \widehat{f}_{>n}$, and, therefore, for every $n\in \mathbb{N}$, there exists $\eta _{n}>0$ such that:
		\begin{align}
			f_{>n}(\zeta ) &=o\big( \mathrm{e}^{-(\beta  _{n}+\eta _{n})\zeta } \big) , \nonumber 
		\end{align}
		uniformly on $\mathcal{R}_{C}$, as $\Re  (\zeta )\to + \infty $. Since $\mathcal{R}_{C}$ is admissible, by Proposition \ref{prop:invariance of D strong} and Example \ref{ExampleStandardQuadraticDom1}, for each $n\in \mathbb{N}$, there exists $R>0$ such that $(\mathcal{R}_{C})_{R}$ is $f_{n}$-invariant and $f$-invariant. Therefore, $f(\zeta ),f_{n}(\zeta )\in (\mathcal{R}_{C})_{R}\subseteq \mathbb{C}^{+}$, for each $\zeta \in (\mathcal{R}_{C})_{R}$. Now, since $\varphi _{n}$ is analytic on $\mathbb{C}^{+}$, by the Taylor Theorem, we get:
		\begin{align*}
			\varphi _{n}(f(\zeta )) &= \varphi _{n}((f_{n}+f_{>n})(\zeta )) \\
			&= \varphi _{n}(f_{n}(\zeta )) +\sum _{i\geq 1}\frac{\varphi _{n}^{(i)}(f_{n}(\zeta ))}{i!}(f_{>n}(\zeta ))^{i} ,
		\end{align*}
		for each $\zeta \in (\mathcal{R}_{C})_{R}$ and $n\in \mathbb{N}$. This implies that
		\begin{align*}
			\varphi _{n}(f(\zeta )) &= \varphi _{n}(f_{n}(\zeta )) + o(\mathrm{e}^{-(\beta  _{n}+\eta _{n})\zeta }) ,
		\end{align*}
		uniformly on $\mathcal{R}_{C}$, as $\Re  (\zeta )\to + \infty $, for each $n\in \mathbb{N}$. This, together with \eqref{BoundForCompEq}, implies that, for every $n\in \mathbb{N}$, and $\varepsilon _{n}:=\min \left\lbrace \mu _{n}, \eta _{n}\right\rbrace $, \eqref{BoundLemmaEq2} holds.
	\end{proof}
	
	\begin{lem}[A solution of a Schr\"oder's type homological equation]\label{LemmaHomologicalEq}
		Let $f$ be a strongly hyperbolic complex Dulac germ defined on a standard quadratic domain $\mathcal{R}_{C}$, $C>0$, and let $\widehat{f}(\zeta )=\alpha \zeta +\mathrm{h.o.t.}$, $\alpha \in \mathbb{R}_{>1}$, be its asymptotic expansion in the $\zeta $-chart. Let $g(\zeta )=o(\mathrm{e}^{-\nu \zeta })$,  uniformly on $\mathcal{R}_{C}$, as $\Re  (\zeta ) \to +\infty $, $\nu \in \mathbb{R}_{>0}$, be an analytic germ defined on $\mathcal{R}_{C}$, in the $\zeta $-chart. Then:
		\begin{enumerate}[1., font=\textup, topsep=0.4cm, itemsep=0.4cm, leftmargin=0.6cm]
			\item $($\emph{Existence}$)$ There exist $R>0$ such that $D:=(\mathcal{R}_{C})_{R}$ is an $f$-invariant subdomain $D\subseteq\mathcal{R}_{C}$, and an analytic solution
			\begin{align}
				\varphi _{g}(\zeta ):= - \sum _{n=0}^{+ \infty }\frac{1}{\alpha ^{n+1}}(g\circ f^{\circ n})(\zeta ) \label{SolutionHomolEq}
			\end{align}
			of the \emph{Schr\"oder's type homological equation}:
			\begin{align}
				\varphi _{g} (f(\zeta )) - \alpha f(\zeta ) & = g(\zeta ) , \quad \zeta \in D. \label{EqHomological}
			\end{align}
			\item $($\emph{Asymptotics}$)$ $\varphi _{g}(\zeta) = O(\mathrm{e}^{-\nu \zeta })$, uniformly on $D$, as $\Re  (\zeta ) \to +\infty $.
			\item $($\emph{Uniqueness}$)$ If $\psi _{g} = o(1)$ is a solution of homological equation \eqref{EqHomological} on $f$-invariant subdomain $D_{1}\subseteq \mathcal{R}_{C}$ then $\psi _{g} (\zeta ) = \varphi _{g}(\zeta )$, for each $\zeta \in (D_{1})_{R}=D\cap D_1$.
		\end{enumerate}
	\end{lem}
	
	The proof of Lemma~\ref{LemmaHomologicalEq} is motivated by \cite[Subsections A.5, A.6]{Loray98}, \cite[Subsection 2.4]{Loray21}, where an explicit solution of an \emph{Abel's type of homological equation} is constructed.
	
	\begin{proof}
		Let $\varepsilon>0$ and $k\in\mathbb N$ be arbitrary. By Example~\ref{ExampleStandardQuadraticDom1}, we take $R>0$ sufficiently large such that $(\mathcal{R}_{C})_{R}=(\mathcal{R}_{C})_{R}^f$, that is, such that $(\mathcal{R}_{C})_R$ is $f$-invariant. Now, put $D:=(\mathcal{R}_{C})_{R}$. Since $g(\zeta )=o(\mathrm{e}^{-\nu \zeta })$, $\nu >0$, uniformly on $\mathcal{R}_{C}$, as $\Re  (\zeta )\to + \infty $, it follows that there exists $R>0$ large enough, such that $\left| g(\zeta )\right| \leq \frac{1}{\mathrm{e}^{\nu \cdot \Re  (\zeta) }}$, $\zeta \in D$. Therefore,
		\begin{align*}
			\left| \sum _{i=0}^{n}\frac{1}{\alpha ^{i+1}}(g\circ f^{\circ i})(\zeta ) - \sum _{i=0}^{n-1}\frac{1}{\alpha ^{i+1}}(g\circ f^{\circ i})(\zeta )\right| & \leq \frac{1}{\alpha ^{n+1}}\cdot \frac{1}{\mathrm{e}^{\nu \cdot \Re  (f^{\circ n}(\zeta )) }} \\
			& \leq \frac{1}{\alpha ^{n+1}}\cdot \frac{1}{\mathrm{e}^{\nu R}}, \quad \zeta \in D, \: n\in \mathbb{N}_{\geq 1} . 
		\end{align*}
		Since $\sum \frac{1}{\alpha ^{n}}$ converges, it implies that $\sum \frac{1}{\alpha ^{n+1}}(g\circ f^{\circ n})$ converges uniformly on $D$. Put
		\begin{align*}
			\varphi _{g}(\zeta ) &:= - \sum _{n=0}^{+\infty }\frac{1}{\alpha ^{n+1}}(g\circ f^{\circ n})(\zeta ), \quad \zeta \in D.
		\end{align*}
		By the Weierstrass' Theorem, $\varphi _{g}$ is analytic on $D$. Now, it is easy to see that $\varphi _{g}$ is a solution of homological equation \eqref{EqHomological}. \\
		
		The statement 2 follows similarly as the proof of statement 2 of \cite[Lemma 4.4]{PRRSDulac21}. \\
		
		3. Suppose that $\psi _{g}(\zeta )=o(1)$, uniformly, as $\Re  (\zeta )\to +\infty $, is a solution of homological equation \eqref{EqHomological} on $D_{1}$. Put $E(\zeta ):=\varphi _{g}(\zeta)-\psi _{g}(\zeta )$, $\zeta \in D_{1}\cap D$. Now, $E\circ f=\alpha E$ and $E(\zeta )=o(1)$, uniformly on $D_{1}\cap D$ as $\Re  (\zeta )\to + \infty $. By the proof of statement 3 of Theorem~B, it follows that $E(\zeta ) = 0$, i.e., $\varphi _{g}(\zeta ) = \psi _{g}(\zeta )$, $\zeta \in D_{1}\cap D=(D_{1})_{R}$.
	\end{proof}

\medskip

\subsection{Proof of Theorem~C}\label{SubsectionProofCompleteB}
	
	\begin{proof}[Proof of Theorem~C]
		Let $f(\zeta )=\alpha \zeta + o(1)$, $\alpha \in \mathbb{R}_{>1}$, be a strongly hyperbolic complex Dulac germ on a standard quadratic domain $\mathcal{R}_{C}$, given in the $\zeta $-chart, and let $\widehat{f}(\zeta )$ be its complex Dulac asymptotic expansion in the same chart. We distinguish two cases.
		
		If $\widehat{f}(\zeta )=\alpha \zeta $, by \emph{quasi-analyticity} (see the end of Subsection~\ref{SubsubsectionDulac}), it follows that $f(\zeta )=\alpha \zeta $. In this case $f$ is already normalized, so we put $\varphi :=\mathrm{id}$. Since the formal normalization $\widehat{\varphi }$ obviously equals $\mathrm{id}$, it follows $\varphi \sim \widehat{\varphi }$.
		
		Now suppose that $\widehat{f}$ is nontrivial. This implies that $f$ satisfies the assumptions of Theorem~B, since the standard quadratic domain $\mathcal{R}_{C}$ is an admissible domain of any type (Example~\ref{ExampleQuadraticDomainStrong}). Therefore, by Theorem~B, for $f$ we obtain an analytic normalization $\varphi $ on the maximal $f$-invariant subdomain $(\mathcal{R}_{C})_{R}^f$ of a standard quadratic domain $\mathcal{R}_{C}$, $C>0$. By Example~\ref{ExampleStandardQuadraticDom1}, it follows that we can take $R>0$ sufficiently large such that $(\mathcal{R}_{C})_{R}=(\mathcal{R}_{C})_{R}^f$, that is, such that $(\mathcal{R}_{C})_R$ is $f$-invariant.
		
		On the other hand, by Proposition~\ref{PropNormalDulac}, in Subsection~\ref{SubsectionFormalNorm}, the normalization of the strongly hyperbolic complex Dulac series $\widehat{f}$ is a parabolic complex Dulac series $ \widehat{\varphi}$.
		
		By statement 2 of Remark~\ref{RemarkStandardQ}, there exists $C_{1}>0$ large enough such that $\mathcal{R}_{C_{1}}\subseteq (\mathcal{R}_{C})_{R}$. Therefore, it is left to prove that $\widehat{\varphi }$ is the asymptotic expansion of $\varphi $ on the standard quadratic domain $\mathcal{R}_{C_{1}}$.
		
		Let
		\begin{align*}
			\widehat{\varphi }(\zeta ) & :=\zeta + \sum _{i\geq 1}Q_{i}(\zeta )\mathrm{e}^{-\beta _{i}\zeta } 
		\end{align*}
		and let $(\varphi _{n})$ be the sequence of partial analytic normalizations defined in Definition~\ref{DefinitionPartial}. Furthermore, let
		\begin{align}
			g_{n}(\zeta ) & :=-(\varphi _{n}(f(\zeta ))-\alpha \varphi _{n}(\zeta )) , \quad \zeta \in \mathcal{R}_{C_{1}}, \, n\in \mathbb{N} . \label{EqFirstLemma}
		\end{align}
		Since $\varphi $ is a normalization of $f$, we have:
		\begin{align}
			\varphi (f(\zeta )) - \alpha \varphi (\zeta ) & =0, \quad \zeta \in \mathcal{R}_{C_{1}}. \label{EqSecondLemma}
		\end{align}
		Note that, by \eqref{EqFirstLemma} and \eqref{EqSecondLemma},
		\begin{align}
			(\varphi  - \varphi _{n})(f(\zeta ))-\alpha (\varphi  - \varphi _{n})(\zeta ) &= g_{n}(\zeta ) , \quad \zeta \in \mathcal{R}_{C_{1}}, \, n\in \mathbb{N} . \label{EqHomological1}
		\end{align}
		By Lemma~\ref{LemmaDulacExpan}, it follows that $g_{n}(\zeta)=o(\mathrm{e}^{-(\beta _{n}+\varepsilon _{n})\zeta })$, for some $\varepsilon _{n}>0$, uniformly on $\mathcal{R}_{C_{1}}$, as $\Re  (\zeta )\to + \infty $, for every $n\in \mathbb{N}$. From statement 3 of Theorem~B, it follows that $(\varphi - \varphi _{n})(\zeta )=o(1)$, for each $n\in \mathbb{N} $. Since $\varphi - \varphi _{n}$ is a solution of homological equation \eqref{EqHomological1}, by statements 2 and 3 of Lemma~\ref{LemmaHomologicalEq}, it follows that $(\varphi - \varphi _{n})(\zeta )=O(\mathrm{e}^{- (\beta _{n}+\varepsilon _{n})\zeta })$, uniformly on $\mathcal{R}_{C_{1}}$, as $\Re  (\zeta ) \to + \infty $, $n\in \mathbb{N} $. Therefore, $(\varphi - \varphi _{n})(\zeta )=o(\mathrm{e}^{- \beta _{n}\zeta })$, uniformly on $\mathcal{R}_{C_{1}}$, as $\Re  (\zeta ) \to + \infty $, for every $n\in \mathbb{N}$. This proves that $\widehat{\varphi }$ is the asymptotic expansion of $\varphi $ on the standard quadratic domain $\mathcal{R}_{C_{1}}$. Therefore, $\varphi $ is a parabolic complex Dulac germ. Furthermore, since $\varphi \sim \widehat{\varphi }$ and $\widehat{\varphi }$ is a complex Dulac series, by statement 4 of Theorem~B, it follows that $\varphi $ is the unique analytic solution of the normalization equation \eqref{EqAnalyticNorm}, which is a parabolic complex Dulac germ.
		
		Furthermore, by statement 2 of Theorem~B, it follows that $\varphi $ is a real Dulac germ if $f$ is a real Dulac germ.
	\end{proof}

	\medskip
	
	\emph{Addresses:}\
	
	$^{1}$: University of Split, Faculty of Science, Ru\dj era Bo\v skovi\' ca 33, 21000 Split, Croatia, email: dino.peran@pmfst.hr


\begin{thebibliography}{10}
	
		\bibitem[ABS19]{abs19}
		I.~Aniceto, G.~Ba\c{s}ar, and R.~Schiappa, \emph{A primer on resurgent
			transseries and their asymptotics}, Phys. Rep. \textbf{809} (2019), 1--135.
		
		\bibitem[ADH13]{adh13}
		M.~Aschenbrenner, {L. van den} Dries, and {J. van der} Hoeven, \emph{Towards a
			model theory for transseries}, Notre Dame J. Form. Log. \textbf{54} (2013),
		no.~3-4, 279--310.
		
		\bibitem[CG93]{CarlesonG93}
		L.~Carleson and T.~W. Gamelin, \emph{Complex dynamics}, Universitext: Tracts in
		Mathematics, Springer-Verlag, New York, 1993.
		
		\bibitem[DMM01]{Dries}
		{L. van den} Dries, A.~Macintyre, and D.~Marker, \emph{Logarithmic-exponential
			series}, Proceedings of the {I}nternational {C}onference ``{A}nalyse \&
		{L}ogique'' ({M}ons, 1997), vol. 111, 2001, pp.~61--113.
		
		\bibitem[Dulac23]{dulac}
		H. Dulac, \emph{Sur les cycles limites}, Bull. Soc. Math. France \textbf{51} (1923), 45--188.
		
		\bibitem[{\'E}ca92]{ecalle92}
		J.~{\'E}calle, \emph{Introduction aux fonctions analysables et preuve
			constructive de la conjecture de {D}ulac}, Actualit{\'e}s Math{\'e}matiques,
		Hermann, Paris, 1992.
		
		\bibitem[Il'84]{Ily84}
		Y.~Il'yashenko, \emph{Limit cycles of polynomial vector fields with
			nondegenerate singular points on the real plane}, Functional Anal. Appl.
		\textbf{18} (1984), no.~3, 199--209.
		
		\bibitem[Il'91]{Ily91}
		\bysame, \emph{Finiteness theorems for limit cycles}, Translations of
		Mathematical Monographs, vol.~94, American Mathematical Society, Providence,
		RI, 1991.
		
		\bibitem[IY08]{IlyYak08}
		Y.~Ilyashenko and S.~Yakovenko, \emph{Lectures on analytic differential
			equations}, Graduate Studies in Mathematics, vol.~86, American Mathematical
		Society, Providence, RI, 2008.
		
		\bibitem[Koe84]{Koenigs84}
		G.~Koenigs, \emph{Recherches sur les int\'{e}grales de certaines \'{e}quations
			fonctionnelles}, Ann. Sci. \'{E}cole Norm. Sup. (3) \textbf{1} (1884), 3--41.
		
		\bibitem[Lor98]{Loray98}
		F.~Loray, \emph{Analyse des s{\'e}ries divergentes}, Math{\'e}matiques pour le
		2e cycle, vol. Quelques aspects des math{\'e}matiques actuelles, Ellipses,
		1998.
		
		\bibitem[Lor21]{Loray21}
		\bysame, \emph{Pseudo-groupe d'une singularit{\'e} de feuilletage holomorphe en dimension deux.}, hal-00016434v2, 2021, \url{https://hal.archives-ouvertes.fr/hal-00016434v2}
		
		\bibitem[Mil06]{Milnor06}
		J.~Milnor, \emph{Dynamics in one complex variable}, third ed., Annals of
		Mathematics Studies, vol. 160, Princeton University Press, Princeton, NJ,
		2006.
		
		\bibitem[MR21]{mrrz3}
		P.~Marde{\v s}i{\' c} and M.~Resman, \emph{Analytic moduli for parabolic Dulac germs}, Russian Mathematical Surveys \textbf{76} (2021), no.~3, 13--92.
		
		\bibitem[MRR{\v{Z}}16]{mrrz1}
		P.~Marde{\v{s}}i{\'c}, M.~Resman, J.-P. Rolin, and V.~{\v{Z}}upanovi{\'c},
		\emph{Normal forms and embeddings for power-log transseries}, Adv. Math.
		\textbf{303} (2016), 888--953.
		
		\bibitem[MRR{\v{Z}}19]{mrrz2}
		P.~Marde\v{s}i\'{c}, M.~Resman, J.-P. Rolin, and V.~{\v{Z}}upanovi\'{c},
		\emph{The {F}atou coordinate for parabolic {D}ulac germs}, J. Differential
		Equations \textbf{266} (2019), no.~6, 3479--3513.
		
		\bibitem[MS2016]{MiSa} C.~Mitschi, D.~Sauzin, \emph{Divergent Series, Summability and Resurgence I (Monodromy and Resurgence)}, Lecture Notes in Mathematics, Springer International Publishing Switzerland, 2016 \url{https://doi.org/10.1007/978-3-319-28736-2} 
		
		\bibitem[Neu49]{Neumann49}B.~H.~Neumann, \emph{On ordered division rings}, Trans. Amer. Math. Soc.,
		\textbf{66} (1949), 202--252.
		
		\bibitem[P21Thesis]{Peran21Thesis}
		D.~Peran, \emph{Normal forms for transseries and Dulac germs}, doctoral thesis, University of Zagreb, 2021, \url{https://urn.nsk.hr/urn:nbn:hr:217:394321}
		
		\bibitem[P21P]{Peran21Parabolic}
		D.~Peran, \emph{Normal forms of parabolic logarithmic transseries}, submitted 2022, \url{https://arxiv.org/pdf/2112.12187.pdf}
		
		\bibitem[PRRS21D]{PRRSDulac21}
		D.~Peran, J.-P. Rolin, M.~Resman, and T.~Servi, \emph{Linearization of complex hyperbolic Dulac germs}, Journal of Mathematical Analysis and Applications, 508(1), 1-27, 2022, \url{https://doi.org/10.1016/j.jmaa.2021.125833}
		
		\bibitem[PRRS21]{PRRSFormal21}
		D.~Peran, J.-P. Rolin, M.~Resman, and T.~Servi, \emph{Normal forms of hyperbolic logarithmic transseries}, Journal of Differential Equations \textbf{348} (2023) 154-190, \url{https://doi.org/10.1016/j.jde.2022.12.002}
		
		\bibitem[Rou98]{Roussarie98}
		R.~Roussarie, \emph{Bifurcation of planar vector fields and {H}ilbert's	sixteenth problem}, Progress in Mathematics, vol. 164, Birkh\"auser Verlag, Basel, 1998. 
		
		\bibitem[Sch70]{Schroder70}
		E.~Schr\"{o}der, \emph{Ueber iterirte {F}unctionen}, Math. Ann. \textbf{3}
		(1870), no.~2, 296--322.
		
		\bibitem[XG16]{XiangG16}
		T.~Xiang and S.~G. Georgiev, \emph{Noncompact-type {K}rasnoselskii fixed-point
			theorems and their applications}, Math. Methods Appl. Sci. \textbf{39}
		(2016), no.~4, 833--863.
		
	\end{thebibliography}
\end{document}